\documentclass[a4paper]{amsart}
\usepackage[utf8x]{inputenc}
\usepackage{amsmath, amsfonts, amsthm, amssymb, fullpage, enumerate, verbatim}

\usepackage[bitstream-charter]{mathdesign}
\usepackage[T1]{fontenc}
\usepackage{color}

\usepackage[colorlinks]{hyperref}

\usepackage{soul}
\hypersetup{linkcolor=blue, urlcolor=blue, citecolor=red}

\newcommand{\A}{\mathcal{A}}
\newcommand{\I}{\mathcal{I}}
\newcommand{\N}{\mathbb{N}}
\newcommand{\Z}{\mathbb{Z}}
\renewcommand{\P}{\mathcal{P}}
\newcommand{\aut}{\operatorname{Aut}}
\newcommand{\id}{\operatorname{id}}
\newcommand{\dom}{\operatorname{dom}}
\newcommand{\ran}{\operatorname{ran}}
\newcommand{\set}[2]{\{#1:#2\}}
\newcommand{\supp}{\operatorname{supp}}
\newcommand{\sym}{\operatorname{Sym}}
\newcommand{\fix}{\operatorname{fix}}
\newcommand{\genset}[1]{\langle #1 \rangle}
\renewcommand{\S}{\mathcal{S}}
\renewcommand{\to}{\longrightarrow}

\newcommand{\changed}[1]{#1}

\makeatletter
\newtheorem*{rep@theorem}{\rep@title}
\newcommand{\newreptheorem}[2]{%
\newenvironment{rep#1}[1]{%
 \def\rep@title{#2 \ref{##1}}%
  \begin{rep@theorem}}%
 {\end{rep@theorem}}}
\makeatother

\numberwithin{equation}{section}
\newtheorem{thm}[equation]{Theorem}
\newreptheorem{thm}{Theorem}
\newtheorem{prop}[equation]{Proposition}
\newtheorem{lem}[equation]{Lemma}
\newtheorem{cor}[equation]{Corollary}

\let\oldmarginpar\marginpar
\renewcommand\marginpar[1]{\-\oldmarginpar[\raggedleft\footnotesize #1]%
{\raggedright\footnotesize #1}}

\title{Topological 2-generation of automorphism groups of countable ultrahomogeneous graphs}
\author{J. Jonu\v sas and J. D. Mitchell}
\date{\today}

\thanks{We thank the Carnegie Trust for the Universities of Scotland for
funding the PhD scholarship of J. Jonu\v{s}as.} 

\begin{document}
\maketitle

\begin{abstract}
  A countable graph is \textit{ultrahomogeneous} if every isomorphism between
  finite induced subgraphs can be extended to an automorphism. Woodrow and
  Lachlan showed that there are essentially four types of such countably
  infinite graphs: the random graph; infinite disjoint unions of complete
  graphs $K_n$ with $n\in\N$ vertices; the $K_n$-free graphs; finite unions of the
  infinite complete graph $K_{\omega}$; and duals of such graphs. The groups
  $\aut(\Gamma)$ of automorphisms of such graphs $\Gamma$ have a natural
  topology, which is compatible with multiplication and inversion, i.e.\  the
  groups $\aut(\Gamma)$ are topological groups. We consider the problem of
  finding minimally generated dense subgroups of the groups $\aut(\Gamma)$
  where $\Gamma$ is ultrahomogeneous. We show that if $\Gamma$ is
  ultrahomogeneous, then $\aut(\Gamma)$ has 2-generated dense subgroups, and
  that under certain conditions given $f \in \aut(\Gamma)$ there exists $g\in
  \aut(\Gamma)$ such that the subgroup generated by $f$ and $g$ is dense. We
  also show that, roughly speaking, $g$ can be chosen with a high degree of
  freedom. For example, if $\Gamma$ is either an infinite disjoint unions of
  $K_n$ or a finite union of $K_{\omega}$, then $g$ can be chosen to have any
  given finite set of orbit representatives.
\end{abstract}

\section{Introduction}

  A \emph{graph} $\Gamma$ is a set of vertices and undirected edges between those
  vertices. Two vertices of a graph are \emph{adjacent} if there is an edge
  between them.
  The \emph{complete graph} $K_n$ is the graph with $n\in \N$ vertices and an
  edge between every pair of distinct vertices. The complete graph with a
  countable infinite set of vertices is denoted $K_{\omega}$.
  If $\Gamma$ and $\Delta$ are graphs with disjoint sets of vertices (and hence
  edges), then the \emph{disjoint union} of $\Gamma$ and $\Delta$ is the graph
  whose vertices and edges are the unions of the vertices and edges,
  respectively, of $\Gamma$ and $\Delta$, and no additional edges.
  The \emph{dual} $\Delta$ of a graph $\Gamma$ has the same vertices as $\Gamma$
  and has an edge between every pair of two distinct vertices which are not
  adjacent in $\Gamma$.  If $U$ is a set of vertices of a graph $\Gamma$, then
  the \emph{subgraph induced} by $U$ is the graph with vertices $U$ and
  edges between $u\in U$ and $v\in U$ if and only if $u$ and $v$ are
  adjacent in $\Gamma$.

  If $\Gamma$ is a graph, then we say that $\Gamma$ satisfies the \emph{Alice's
  restaurant property} if for every pair of disjoint \changed{finite} subsets
  $U$ and $V$ of vertices of $\Gamma$ there exists a vertex $w\in \Gamma
  \setminus (U\cup V)$ such that $w$ is adjacent to every vertex in $U$ and to
  no vertex in $V$. Classical results (for example \cite{Erdos:1963qf}) show
  that there exists a countable infinite graph with the Alice's results
  property and that any two countably infinite graphs with the property are
  isomorphic. As such we refer to any such graph as the \emph{random graph};
  denoted $R$.

  A graph is \emph{$K_n$-free} if none of the subgraphs induced by sets
  consisting of $n\in \N$ vertices is a complete graph. Obviously, for this
  definition to meaningful $n$ must be at least $2$.  If $n\in \N$ is fixed and
  $\Gamma$ is a $K_n$-free graph, then we say that $\Gamma$ has the \emph{Alice's
  restaurant property for $K_n$-free graphs} if for every pair of disjoint
  induced subgraphs $U$ and $V$ of $\Gamma$ where $U$ is $K_{n-1}$-free, there
  exists a vertex $w\in \Gamma \setminus (U\cup V)$ such that $w$ is adjacent to
  every vertex in $U$ and to no vertex in $V$. Again, countably infinite graphs
  satisfying the Alice's restaurant property for $K_n$-free graphs, $n > 1$,
  exist, any two such graph are isomorphic, and we refer to any such graph as the
  \emph{universal $K_n$-free graph}; denoted $H(n)$.

  Although it is not relevant for this paper, the universal $K_n$-free graphs, $n
  > 1$, and the random graph are the \emph{Fra\"iss\'e limits} of the classes of
  finite $K_n$-free graphs and finite graphs, respectively; see
  \cite{Hodges1997aa} for more details about Fra\"iss\'e limits.

  If $\Gamma$ and $\Delta$ are graphs, then a function $f: \Gamma\to
  \Delta$ is an \emph{isomorphism} if $f$ is a bijection which maps adjacent
  vertices in $\Gamma$ to adjacent vertices in $\Delta$.  An isomorphism from a
  graph $\Gamma$ to itself is an \emph{automorphism}, and the group of all
  automorphisms of $\Gamma$ is denoted by $\aut(\Gamma)$.
  A countable graph is \textit{ultrahomogeneous} if every isomorphism between
  finite induced subgraphs can be extended to an automorphism. Woodrow and
  Lachlan showed that there are essentially four types of such countably
  infinite graphs; described in the following theorem.

  \begin{thm}[cf. \cite{lachlan1980aa}]
    \label{thm-lachlan}
    The countable ultrahomogeneous graphs up to isomorphism are:
    \begin{enumerate}[\rm (i)]

      \item
        the random graph $R$;

      \item
        the $K_n$-free universal graph $H(n)$, for every $n\in \N$, $n\geq
        3$;

      \item
        the graph $\omega K_n$ consisting of the disjoint union of countably
        many copies of $K_n$, for every $n\in \N$;

      \item
	the graph $nK_{\omega}$ consisting of the disjoint union of $n \in
	\mathbb{N}$ copies of $K_{\omega}$, for $n \geq 2$;

    \end{enumerate}
    and the duals of these graphs.
  \end{thm}

  In this paper, we are primarily concerned with the groups of automorphisms of
  the graphs in Theorem~\ref{thm-lachlan}. Since the automorphism group of a
  graph and its dual are equal, it will suffice to consider the graphs in
  Theorem~\ref{thm-lachlan}(i) -- (iv), and not their duals.

  Suppose that $\Gamma$ is a graph.
  If $\phi$ is an isomorphism between finite induced subgraphs of $\Gamma$, then we
  denote the domain of $\phi$ by $\dom(\phi)$, and the range by $\ran(\phi)$.
  The groups $\aut(\Gamma)$ of automorphisms  of such graphs $\Gamma$ have a
  natural topology with basis consisting of the sets
  \[
    [\phi] := \{f\in \aut(\Gamma): (x)f = (x)\phi \text{ for all } x\in
    \dom(\phi)\}
  \]
  where $\phi$ is an isomorphism of finite induced subgraphs of $\Gamma$.  If
  $X$ is any subset of $\aut(\Gamma)$, then we denote by $X^{<\omega}$ the set
  of isomorphisms between finite induced subgraphs of $\Gamma$ with an
  extension in $X$. The set $\set{[\phi]}{\phi\in \aut(\Gamma)^{<\omega}}$ is
  the basis for the topology on $\aut(\Gamma)$ given above.  It can be shown
  that multiplication, thought of as a function from $\aut(\Gamma)\times
  \aut(\Gamma)$, with the product topology to $\aut(\Gamma)$, is continuous
  with respect to this topology, and that inversion, ${}^{-1}:\aut(\Gamma) \to
  \aut(\Gamma)$, is also continuous. As such, $\aut(\Gamma)$ is a topological
  group.  The topology on $\aut(\Gamma)$ is completely-metrizable, i.e.\ there
  exists a complete metric inducing the topology on $\aut(\Gamma)$.  A subset
  of a topological space is \textit{dense} if it has non-empty intersection
  with every open set.  The basis defined above is countable, and so
  $\aut(\Gamma)$ is separable, and hence a Polish group.  A topological space
  is a \textit{Baire space} if every countable intersection of open dense set
  is dense. If $X$ is a Baire space, and $Y \subseteq X$, then $Y$ is a
  \textit{comeagre} subset of $X$ if $Y$ contains an intersection of open dense
  sets. Since $\aut(\Gamma)$ is a Polish space, it is a Baire space;
  \cite[Theorem 8.4]{kechris1995classical}. It is well-known (see, for example,
  \cite[Theorem 3.11]{kechris1995classical}) that $G_{\delta}$ subspaces of
  Polish spaces are Polish, it is also easy to show that in a metric space
  every closed set is a $G_\delta$ set. Hence every $G_\delta$ subspace, and
  every closed subspace of $\aut(\Gamma)$ is Polish, and thus Baire.

  In this paper, we consider the problem of finding minimally generated dense
  subgroups of the groups $\aut(\Gamma)$ where $\Gamma$ is an ultrahomogeneous
  graph. In particular, we show that, under certain assumptions, if $f \in
  \aut(\Gamma)$, then there exists a Baire subspace of $\aut(\Gamma)$
  containing a comeagre set $C$ with the property that every $g \in C$
  generates a dense subgroup together with $f$.
  If $f \in \aut(\Gamma)$ is arbitrary, then the subspaces we will consider are:
  \begin{eqnarray}\label{eq-definitions-1}
    D_f               & = & \{ g \in \aut(\Gamma) : \langle f,g \rangle \text{
                            is dense in } \aut(\Gamma) \},\nonumber\\
    \I(\Gamma)        & = & \{ g \in \aut(\Gamma) : g \text{ has no finite
                            orbits} \},\\
    \I_\Sigma(\Gamma) & = & \{ g \in \I\changed{(\Gamma)} : \Sigma \subset \Gamma \text{ is a set
                            of orbit representatives for } g\},\nonumber
  \end{eqnarray}
  where the \textit{set of orbit representatives} of an automorphism $g$
  consists of exactly one vertex in every orbit of $g$.

  Suppose that $\Gamma$ is a graph consisting of the disjoint union of
  countably many copies of $K_n$ or finitely many copies of $K_{\omega}$. We
  denote by $L_1, L_2, \ldots$ the connected components of $\Gamma$.  Every
  $f\in \aut(\Gamma)$ induces a permutation $\overline{f}$ of
  the indices of the connected components of $\Gamma$, $\N$ or
  $\{1,2,\ldots, n\}$, which is defined by
  $$(i)\overline{f}=j\quad\text{if }\quad(L_i)f=L_j.$$
  If $f\in \aut(nK_\omega)$ is a non-identity element and $\Sigma\subseteq n
  K_\omega$, then we define:
  \begin{equation}\label{eq-definitions-2}
    \begin{array}{rcl}
      \mathcal{A}_{f} & = & \{ g \in \aut(n K_\omega) : \langle \overline{f},
      \overline{g} \rangle = S_n \}\\
      \mathcal{A}_{f, \Sigma} & = & \{g\in \A_f : \Sigma\subseteq n K_{\omega}
      \text{ is a set of orbits representatives for }g\}.
    \end{array}
  \end{equation}
  If $n \not= 4$, then, by a classical theorem \cite{piccard1939aa},
  $\mathcal{A}_f \not= \varnothing$ for all non-identity $f$.

  In the next section, we will show that $\I(\Gamma)$ and $\I_{\Sigma}(\Gamma)$
  are Baire spaces with the subspace topology in $\aut(\Gamma)$, and that
  $\A_{f, \Sigma}$ and $\A_{f}$ are Baire subspaces of $\aut(n K_{\omega})$.

  The main results of this paper are contained in the following theorem.

  \begin{thm}
    \label{thm-main}\mbox{}
    \begin{enumerate}[\rm (i)]

      \item
        $D_f \cap \I(H(n))$ is comeagre in $\I(H(n))$ for all $f \in
        \aut(H(n))$ with infinite support;

      \item
	$D_f \cap \I_\Sigma(\omega K_n)$ is comeagre in $\I_\Sigma(\omega
	K_n)$, for all  $f \in \aut(\omega K_n)$ such that support of
	$\overline{f}$ is infinite, and $\Sigma$ is a finite subset of $\omega
	K_n$;

     \item
       Suppose that $f \in \aut(n K_\omega)$ is such that for every finite
       subset $\Gamma$ of $nK_{\omega}$ that is setwise stabilised by $f$ there
       are components $L$ and $L'$ of $nK_{\omega}$ such that $|L\cap \Gamma|
       \not= |L'\cap \Gamma|$. Then $D_f\cap \A_{f,\Sigma}$ is comeagre in
       $\A_{f, \Sigma}$ for every finite subset $\Sigma$ of $nK_{\omega}$.

    \end{enumerate}
  \end{thm}

  The analogue of Theorem~\ref{thm-main}(i) for the random graph was proven in
  \cite[Theorem 1.6]{darji2011aa} and for the symmetric group in \cite[Theorem
  3.3]{darji2008highly}.

  The paper is organised as follows: in Section~\ref{section-prelim} we define
  some further notation and give some results that are common to the proofs
  of the three parts of Theorem~\ref{thm-main}. We prove the three parts of
  Theorem~\ref{thm-main} in the final three sections of the paper.


\section{Preliminaries}\label{section-prelim}

We denote the cardinality of the natural numbers by $\omega$, and suppose that
$\N = \{0, 1, \ldots\}$.


A \emph{graph} $\Gamma$ is a pair $(V(\Gamma), E(\Gamma))$ of sets: $V(\Gamma)$
of \emph{vertices} and $E(\Gamma)\subseteq
\left\{ \{x, y\} : x, y \in V(\Gamma) \text{ and } x \neq y \right\}$
of edges. Where appropriate we identify $\Gamma$ and $V(\Gamma)$ so that we may
write $x\in \Gamma$ to mean $x$ is a vertex of $\Gamma$. If $\{x,y\}$ is an
edge of a graph $\Gamma$, then we say that $x$ and $y$ are \emph{adjacent} in
$\Gamma$. If $x$ is a vertex of $\Gamma$, then the subgraph induced by the set
of all vertices adjacent to $x$ is denoted $N(x)$.

Suppose that $f: X \to Y$ for some sets $X$ and $Y$. Then we refer to $X$ and
$Y$ as the \emph{domain} and \emph{range} of $f$, denoted by $\dom(f)$ and
$\ran(f)$. If $Z\subseteq \dom(f)$, then we define $(Z)f = \set{(z)f}{z\in Z}$.
If $f: X \to Y$ and $Z\subseteq X$, then the \textit{restriction} of
$f$ to $Z$ is the function $f|_{Z}:Z \to Y$ such that $(z)f|_{Z} = (z)f$ for
all $z\in Z$. We say that $f$ is an \textit{extension} of any of its
restrictions.  We refer to any isomorphism between finite induced subgraphs of
a graph $\Gamma$ as a \emph{partial isomorphism of $\Gamma$}.

If $f$ and $g$ are arbitrary bijections, then we define their composition
\[
  f \circ g : \dom(f) \cap \left( \dom(g) \cap \ran(f) \right) f^{-1} \to
  \ran(g) \cap \left(\dom(g) \cap \ran(f)\right)g
\]
to be $(x) f\circ g = ((x)f)g$ whenever $(x)f \in \dom(g)$. We denote the
composite $f \circ f ^ {-1}$ by $f ^ 0$, being the identity on $\dom(f)$.

If $f$ is a bijection and $x \in \dom(f) \cup \ran(f)$, we define the
\emph{component} of $x$ under $f$ to be the set
$$\{(x)f ^ k : k \in \Z\ \text{and}\ x \in \dom(f ^ k)\}.$$
A component of a bijection $f$ is \emph{complete} if $(x)f ^ k$ is defined for
every $k\in \Z$. A component that is not complete is \emph{incomplete}.
If $f: X\to X$ is a permutation, then every component of $f$ is complete,
and in this context, complete components are called \emph{orbits}.

Next, we will show that $\I(\Gamma)$ and $\I_{\Sigma}(\Gamma)$ (as defined in
\eqref{eq-definitions-1}) are Baire spaces with the subspace topology in
$\aut(\Gamma)$ for any countably infinite graph $\Gamma$, and that $\A_{f,
\Sigma}$ and $\A_{f}$ (defined in \eqref{eq-definitions-2}) are Baire subspaces
of $\aut(n K_{\omega})$.

If $\Gamma$ is any countably infinite graph and $f \in \aut(\Gamma) \setminus
\I(\Gamma)$, then $f$ has a finite orbit $O$ and hence $[f|_O]$ is a subset of
$\aut(\Gamma) \setminus \I(\Gamma)$. In other words, $\I(\Gamma)$ is closed,
and hence Baire.

\begin{lem}\label{lem-Af-closed}
  \changed{The subset $\A_f$ of $\aut(nK_\omega)$} is a Baire space.
\end{lem}

\begin{proof}
  Let $g \in \aut(n K_\omega) \setminus \A_f$. Then $\langle \overline{f},
  \overline{g} \rangle \neq S_n$. Let $\Gamma \subseteq n K_\omega$ be a finite
  set containing at least one vertex in every connected component of
  $nK_{\omega}$.  Then for all $h \in [g|_{\Gamma}]$ we have that
  $\overline{h} = \overline{g}$ and thus $h \notin \A_f$. Therefore, the open
  set $[g|_{\Gamma}]$ is a subset of $\aut(n K_\omega) \setminus \A_f$ and
  thus $\A_f$ is closed, and hence Baire.
\end{proof}

That $\I_{\Sigma}(\Gamma)$ and $\A_{f, \Sigma}$ are Baire follows immediately
from the next lemma, and the preceding discussion.

\begin{lem}\label{lem-G-delta}
  \changed{Let $\Omega$ be countable}, let $T$ be a subspace of $\sym(\Omega)$,
  and let $\Sigma \subseteq \Omega$ be finite. If $T$ is Baire, then
  \[
    T_{\Sigma} = \{ f \in T : \Sigma \text{ is a set of
      orbit representatives of }f\}
  \]
  is also Baire.
\end{lem}

\begin{proof}
  Let $K$ be the set of those $g\in T$ such that distinct elements of $\Sigma$
  belong to different orbits of $g$. We will show that $K$ is a closed subset
  of $T$.  If $T = K$, then $K$ is closed in $T$. Otherwise,
  let $g\in T\setminus K$. Then there exist $x,y\in \Sigma$ and $m\in \N$
  such that $(x)g ^ m = y$. If $\Gamma = \{(x)g^i: 0\leq i \leq m\}$, then
  $[g|_{\Gamma}] \cap T$ is a subset of $T\setminus K$. Hence $T\setminus K$ is
  open, and so $K$ is closed.  Since closed subspaces of Baire spaces are
  Baire, it follows that $K$ is Baire.

  If $x \in \Omega$ is arbitrary, then we denote by $A_x$ the set of all those
  $g\in K$ such that the orbit of $x$ under $g$ has non-trivial intersection
  with $\Sigma$.  Then $T_{\Sigma} = \bigcap_{x \in \Omega} A_x \subseteq K$.
  Suppose that $g \in A_x$. Then there is $n \in \Z$ and $y \in \Sigma$ such
  that $(y)g^n = x$. If $\Gamma' = \{(y)g^i : 0\leq i \leq n\text{ or }n\leq i
  \leq 0\}$.  Then $[g|_{\Gamma'}]\cap K$ is a subset of $A_x$ and so
  $A_x$ is open in $K$ for all $x$. Therefore $T_{\Sigma}$, being a
  $G_{\delta}$ subset of $K$, is Baire.
\end{proof}

We end this section by stating two lemmas that will be used repeatedly later in
the paper.  We omit the easy proof of the first lemma.

\begin{lem}\label{lem-E-open}
  Let $\Gamma$ be any graph. Then for every $f \in
  \aut(\Gamma)$ and any $p \in \aut(\Gamma)^{<\omega}$
  \[
    \{ g \in \aut(\Gamma) : \langle f,g \rangle \cap [p] \neq \varnothing \}
  \]
  is an open set in $\aut(\Gamma)$.
\end{lem}

\begin{lem}\label{lem-no-sigma}
  Let $\Gamma$ be any graph, let $f\in \aut(\Gamma)$, and let $S \subseteq
  \aut(\Gamma)$ be such that every $q \in S^{< \omega}$ has an extension in $S$
  with only finitely many orbits. If $D_f \cap S_\Sigma$ is \changed{dense} in
  $S_\Sigma$ for every finite $\Sigma\subseteq \Gamma$, then $D_f \cap S$ is
  comeagre in $S$.
\end{lem}

\begin{proof}
  Since $\{[q] : q \in \aut(\Gamma)^{< \omega}\}$ is a basis for
  the topology on $\aut(\Gamma)$, it follows that
  \[
    D_f \cap S = \{ g \in S: \langle f, g\rangle \text{ is dense in }
    \aut(\Gamma) \} = \bigcap_{p \in \aut(\Gamma)^{< \omega}} \{ g
    \in S : \langle f,g \rangle \cap [p] \neq \varnothing \}.
  \]
  Since $\{ g \in S : \langle f,g \rangle \cap [p] \neq \varnothing \}$ is open
  in $S$ by Lemma~\ref{lem-E-open}, it suffices to show that $\{ g \in S :
  \langle f,g \rangle \cap [p] \neq \varnothing \} $ is dense in $S$ for all $p
  \in \aut(\Gamma)^{< \omega}$.

  Let $q \in S^{< \omega}$. By the hypothesis there is $g \in S$ which extends
  $q$ and has a finite number of orbits. Let $\Sigma$ be a set of orbit
  representatives of $g$. Then $q \in S_\Sigma^{< \omega}$.  Since $D_f \cap
  S_\Sigma$ is dense in $S_\Sigma$ there is $h \in [q]$ such that $h \in D_f
  \cap S_\Sigma$. In other words, $\langle f, h \rangle$ is dense in
  $\aut(\Gamma)$ and so $\{ g \in S : \langle f,g \rangle \cap [p] \neq
  \varnothing \}$ is dense in $S$.
\end{proof}


\section{$K_n$-free graphs}

In this section we will consider the ultrahomogeonous $K_n$-free graphs, denoted by
$H(n)$, for $n \geq 3$. The case $n = 2$ gives a graph with no edges and it's
automorphism group is just the symmetric group on countably many points,
which was already considered \changed{in~\cite{darji2008highly}}.

If for $x\in H(n)$, the subgraph $N(x)$ has a subgraph $\Gamma$ isomorphic to
$K_{n - 1}$, then $\Gamma \cup \{x\}$ is isomorphic to $K_n$, which is
impossible. Hence $N(x)$ is $K_{n - 1}$-free for every vertex $x \in H(n)$. We
will repeatedly make use of this fact without reference.

Let $U$ and $V$ be finite disjoint subsets of vertices of $H(n)$ such that $U$
is $K_{n-1}$-free. Then, by the Alice's restaurant property for $H(n)$, there is
a vertex $w \notin U \cup V$ such that there are no edges between $w$ and $V$,
and there is an edge between $u$ and $w$ for all $u \in U$. In other words,
$N(w) \cap \left( U \cup V \right) = U$.

The purpose of this section is to prove Theorem~\ref{thm-main}(i), which we
restate for the sake of convenience.

\begin{thm}\label{thm-5-main}
  Let $f \in \aut(H(n))$ have infinite support. Then $D_f \cap \I(H(n))$ is
  comeagre in $\I(H(n))$.
\end{thm}

We will proceed by first proving a number of technical results. First,
we will show that the set $D_f \cap \I$ can be written as a countable
intersection of sets of a certain type. The rest of the argument will then be
dedicated to showing that these sets are open and dense.


\begin{lem}\label{lem-K_n-basis}
 Let $\P \subseteq \aut(H(n))^{< \omega}$ be such that \changed{$p \in \P$ if and only
 if $\dom(p) \cap \ran(p) = \varnothing$ and there are no edges between
 $\dom(p)$ and $\ran(p)$}. Then
 \[
   D_f \cap \I(H(n)) = \bigcap_{p \in \P} \{g \in \I(H(n)) : \langle f, g \rangle \cap [p]
   \neq \varnothing \}.
 \]
\end{lem}

\begin{proof}
  Recall that
  \begin{align*}
    D_f \cap \I(H(n)) &= \{g \in \I(H(n)): \langle f, g\rangle
    \text{ is dense in } \aut(H(n)) \} \\
    &= \bigcap_{q \in \aut(H(n))^{<\omega}} \{g \in \I(H(n)) :
    \langle f, g \rangle \cap [q] \neq \varnothing \}.
  \end{align*}

  $(\subseteq)$ This follows immediately since $\mathcal{P} \subseteq
  \aut(H(n))^{<\omega}$.

  $(\supseteq)$ Let $g \in \bigcap_{p \in \P} \{g \in \I(H(n)) : \langle f, g \rangle
  \cap [p] \neq \varnothing \}$ and suppose that $q \in \aut(H(n))^{< \omega}$.
  By repeated application of the Alice's restaurant propery we can find a
  subgraph $\Gamma$ of $H(n)$ such that $\Gamma$ is isomorphic to $\dom(q)$,
  $\Gamma \cap \left( \dom(q) \cup \ran(q) \right) = \varnothing$,  and such
  that there are no edges between $\Gamma$ and $\dom(q) \cup \ran(q)$. Let $p$
  be the isomorphism between $\dom(q)$ and $\Gamma$. Since $H(n)$ is
  ultrahomogeonous, we have that $p \in \aut(H(n))^{< \omega}$. Then $\dom(p) =
  \dom(q)$, $\ran(p) = \dom(p^{-1}q) = \Gamma$ and $\ran(p^{-1}q) = \ran(q)$.
  Hence $p, p^{-1}q \in \mathcal{P}$. By the choice of $g$ there are $h_1, h_2
  \in \langle f, g \rangle$ such that $h_1 \in [p]$ and $h_2 \in [p^{-1}q]$.
  Therefore $h_1h_2 \in [q]$ and $h_1h_2 \in \langle f, g \rangle$, thus
  $\langle f, g\rangle \cap [q] \neq \varnothing$. Since $q$ was arbitrary, $g
  \in \bigcap_{q \in \aut(H(n))^{<\omega}} \{g \in \I(H(n)) : \langle f, g \rangle
  \cap [q] \neq \varnothing \}$.
\end{proof}


The following lemma provides a condition under which it is possible to extend a
given partial isomorphism of $H(n)$ to another partial isomorphism of $H(n)$.
Although we will only apply the following lemma to the graphs $H(n)$, we state
it for arbitrary ultrahomogeneous graphs, since the proof is no harder in the
general case.

\begin{lem}
\label{lem-5-neigh}
  Let $\Gamma$ be an ultrahomogeneous graph, let $q \in \aut(\Gamma)^{<
  \omega}$, and let $x, y \in \Gamma$. Suppose that $x \notin \dom(q)$ and
  $N(y) \cap \ran(q) = \left(N(x)\right)q$. Then $q \cup \{(x, y)\} \in
  \aut(\Gamma)^{< \omega}$.
\end{lem}
\begin{proof}
  Since $\Gamma$ is ultrahomogeneous, it is sufficient to show that $q \cup
  \{(x, y)\}$ an isomorphism between two subgraphs of $\Gamma$. By the
  hypothesis, $q$ is an isomorphism, and so it suffices to show that there is
  an edge between vertices $x$ and $z \in \dom(q)$ if and only if there is an
  edge between vertices $y$ and $(z)q$. Let $z \in \dom(q)$. Then there is an
  edge between $z$ and $x$ if and only if $z \in N(x)$ which is equivalent to
  $(z)q \in N(y) \cap \ran(q)$, i.e. there is an edge between $(z)q$ and $y$.
\end{proof}


The following easy corollary shows that any incomplete component of an
isomorphism of $H(n)$ can be extended.

\begin{cor}\label{cor-5-one-point-ext}
  Let $q \in \aut(H(n))^{< \omega}$, let $x \in H(n) \setminus \dom(q)$, and
  let $\Sigma \subseteq H(n)$ be finite. Then there is $y \in H(n) \setminus
  \left( \{x\} \cup \Sigma \right)$ such that $q \cup \{(x, y)\}
  \in \aut(H(n))^{< \omega}$.
\end{cor}

\begin{proof}
  Let $U = \left( N(x)\right) q$ and let $V = \left(\ran(q) \cup \{x\} \cup
  \Sigma\right) \setminus U$. Since $N(x)$ is $K_{n - 1}$-free and $q$
  is a partial isomorphism, $U$ is also $K_{n -1}$-free. Hence by the Alice's
  restaurant propery there is $y \in H(n) \setminus \left( \ran(q) \cup \{x\}
  \cup \Sigma \right)$  such that $N(y) \cap \ran(q) = \left( N(x) \right)q$.
  Therefore we are done by Lemma~\ref{lem-5-neigh}.
\end{proof}


Let $q$ be a partial isomorphism of $H(n)$ such that $q$ has no complete
components, set $\Sigma$ to be $\dom(q) \cup \ran(q)$, and let $x \in
H(n)\setminus \dom(q)$. Then by Corollary~\ref{cor-5-one-point-ext} there is a
partial isomorphism $h$ of $H(n)$ extending $q$ such that $x \in \dom(h)$ and
$h$ has no complete components. Repeatedly applying
Corollary~\ref{cor-5-one-point-ext} in a back and forth argument we may deduce
that there is an $r \in \I(H(n))$ extending $q$, which gives us the follow
lemma.

\begin{cor}\label{cor-5-infty}
  Let $q \in \aut(H(n))^{< \omega}$. Then $q \in \I(H(n))^{< \omega}$ if and only if
  $q$ has no complete components.
\end{cor}


The following two technical lemmas form the essential part of the proof of
Theorem~\ref{thm-5-main}.

\begin{lem}\label{lem-5-step-in-main}
  Let $q \in \I(H(n))^{< \omega}$ be such that $\ran(q) \cup \dom(q) = \Delta\cup
  \Gamma$ where $\Delta\cap \Gamma = \varnothing$ and
  $\Gamma$ is the union of incomplete components of $q$ of fixed length $m$,
  let $x, y \notin \dom(q) \cup \ran(q)$ be such that
  $$N(x) \cap \Delta \subseteq \dom(q^{2m})\quad \text{and} \quad
  \left( N(x) \cap \Delta \right) q^{2m} = N(y) \cap \Delta$$
  and let $\Sigma_1, \Sigma_2 \subseteq
  H(n) \setminus \Gamma$ be finite subsets such that
  $\Sigma_1 \cap \ran(q) = \varnothing$ and $\Sigma_2 \cap \dom(q) =
  \varnothing$.  Then there are $x_1, \ldots, x_{2m - 1}\in H(n)\setminus
  \Sigma_1 \cup \Sigma_2$ such that
  there are no edges between $x_i$ and $\Sigma_1 \cup \Sigma_2$ for all $i \in
  \{1, \ldots, 2m - 1\}$, and
  $$q \cup \{(x_i, x_{i + 1}) : 0 \leq i
  \leq 2m - 1\} \in \I(H(n))^{< \omega}$$
  where $x_0 = x$ and $x_{2m} = y$.
\end{lem}

\begin{proof}
  Define $q_0 = q$, $x_0 = x$ and $\Gamma_i = \dom(q_i) \cup \ran(q_i) \cup
  \Sigma_1 \cup \Sigma_2 \cup \{x, y \}$ for all $i$. Suppose that for $i \in
  \{0, \ldots, m - 1\}$ there is an extension $q_i \in \I(H(n))^{< \omega}$ of $q_0$
  such that $q_i = q_0 \cup \{ (x_j, x_{j + 1}) : 0 \leq j \leq i - 1 \}$ with
  $x_0 \notin \ran(q_i)$, $x_i \notin \dom(q_i)$, $y \notin \ran(q_i) \cup
  \dom(q_i)$, and
  \begin{align}
    x_j &\notin \Sigma_1 \cup \Sigma_2 \cup \Delta \label{condition-2} \\
    N(x_j)&\cap \left( \Sigma_1 \cup \Sigma_2 \cup \{x_0, \ldots, x_{j - 1}, y
    \}\right) = \varnothing \label{condition-3} \\
    N(x_i) \cap \Gamma_i &= \left( N(x_ 0) \cap \Gamma_0 \right) q_i^i
    \label{condition-1}
  \end{align}
  for all $j \in \{1, \ldots, i\}$.

  If $i = 0$, then we have that $x_0, y \notin \dom(q_0) \cup \ran(q_0)$ and
  \eqref{condition-2}, \eqref{condition-3}, \eqref{condition-1} are
  trivially satisfied.

  Suppose that $i > 0$. Let $U = \left(N(x_i)\right) q_i \subseteq \Gamma_i$
  and $V = \Gamma_i \setminus U$. If $N(x_i)$ contains a subgraph isomorphic to
  $K_{n - 1}$, then the subgraph together with $x_i$ forms $K_n$, which is
  impossible. Hence $N(x_i)$ is $K_{n - 1}$-free and since $q_i$ is an
  isomorphism, $U$ is also $K_{n - 1}$-free. Therefore the sets $U$ and $V$
  satisfy the hypothesis of the Alice's restaurant propery and thus there is a
  vertex $x_{i + 1} \in H(n) \setminus \Gamma_i$ such that there is an edge
  between $x_{ i + 1}$ and every vertex in $U$ and there are no edges between
  $x_{i + 1}$ and $V$, i.e. $N(x_{i + 1}) \cap \Gamma_i = U$. Also it follows
  from $\ran(q_i) \subseteq \Gamma_i$ that
  \[
    N(x_{i + 1}) \cap \ran(q_i) =  \left( N( x_{i + 1}) \cap \Gamma_i\right)
    \cap \ran(q_i) = U \cap \ran(q_i) = \left( N(x_i) \right) q_i.
  \]
  Then $q_{i + 1} = q_i \cup \{ (x_i , x_{i + 1}) \} = q_0 \cup \{(x_j, x_{j +
  1}): 0 \leq j \leq i\}  \in \aut(H(n))^{< \omega}$ by
  Lemma~\ref{lem-5-neigh}, and so $q_{i + 1} \in \I(H(n))^{< \omega}$ by
  Corollary~\ref{cor-5-infty}. Since $x_{i + 1} \notin \Gamma_i$, we have that
  $x_{i + 1} \notin \{x_0, x_i, y\}$ implying that $x_0 \notin \ran(q_{i +
  1})$, $x_{i + 1} \notin \dom(q_{i + 1})$, and $y \notin \ran(q_{i + 1}) \cup
  \dom(q_{i + 1})$. It also follows from $\dom(q_i) \subseteq \Gamma_i$ and
  \eqref{condition-1} that
  \begin{equation} \label{equation-3}
    N( x_{i + 1}) \cap \Gamma_i = U = \left( N(x_i) \right) q_i  = \left(
    N(x_i) \cap \Gamma_i \right) q_i = \left( N(x_0) \cap \Gamma_0 \right)
    q_i^{i + 1}.
  \end{equation}

  Since $\Sigma_1 \cup \Sigma_2 \cup \Delta \subseteq \Gamma_i$ and $x_{i + 1}$
  is chosen outside the set $\Gamma_i$ it follows that $x_{i + 1} \notin
  \Sigma_1 \cup \Sigma_2 \cup \Delta$. Then $x_j \notin \Sigma_1 \cup \Sigma_2
  \cup \Delta$ for all $j \in \{1, \ldots, i + 1\}$ by \eqref{condition-2}.

  We will now show that \eqref{condition-3} holds for $j = i
  + 1$. First of all note that $x_0, y \notin \ran(q_i)$, and since $U
  \subseteq \ran(q_i)$ we have that $x_0, y \notin U$. From
  \eqref{condition-3} we may deduce that $x_j \notin N(x_i)$, and thus $x_{j +
  1} \notin \left( N(x_i) \right) q_i = U$, for all $j \in \{0, \ldots, i -
  1\}$, i.e. $\{x_0, \ldots, x_i, y\} \cap U = \varnothing$. It
  follows from the hypothesis that $\Sigma_1 \cap \ran(q_0) = \varnothing$, and
  so \eqref{condition-2} implies that $\Sigma_1\cap \ran(q_i) =
  \varnothing$. Since $U \subseteq \ran(q_i)$, we have that $\left(\Sigma_1
  \cup \{x_0, \ldots, x_i, y\}\right) \cap U = \varnothing$.

  It remains to show that $\Sigma_2 \cap U = \varnothing$. Suppose $z
  \in \Sigma_2 \cap U$. Then $z \in \left( N(x_0) \cap \Gamma_0 \right) q_i^{i
  + 1}$ by \eqref{equation-3}. Then $z \in \ran(q_i)$ and by above $z \neq x_j$
  for all $j \in \{0, \ldots,i\}$, thus $z \in \ran(q_0) \subseteq \Gamma \cup
  \Delta$. However by the hypothesis of the lemma, $\Sigma_2 \subseteq H(n)
  \setminus \Gamma$, implying that $z \in \Delta$. Since $x_j \notin \Delta$
  for all $j \in \{1, \ldots, i\}$ by \eqref{condition-2} and $x_0 \notin
  \Delta$ by the hypothesis of the lemma, it follows that the incomplete component of
  $q_0$ containing $z$  was not extended in $q_i$. Moreover $\Delta$ is a union
  of incomplete components of $q_0$ whence $(z)q_i^{-(i + 1)} \in N(x_0) \cap \Delta$.
  Also from $\Sigma_2 \cap \dom(q_0) = \varnothing$ and
  \eqref{condition-2} we may deduce that $\Sigma_2 \cap \dom(q_i) =
  \varnothing$ and so $z \notin \dom(q_i)$. It also follows from the hypothesis
  of the lemma that $(z)q_i^{-(i + 1)} \in \dom(q_i^{2m})$. Then $z
  \in \dom(q_i^{2m - (i + 1)})$, which is impossible since $i + 1 < 2m$. Hence
  $U \cap \left( \Sigma_1 \cup \Sigma_2 \cup \{x_0, \ldots, x_{i}, y\} \right)
  = \varnothing$. Since $\left( \Sigma_1 \cup \Sigma_2 \cup \{x_0, \ldots, x_{i
  }, y\} \right) \subseteq \Gamma_{i + 1}$ we have that
  \begin{align*}
    N(x_{i + 1}) \cap \left( \Sigma_1 \cup \Sigma_2 \cup \{x_0, \ldots, x_i, y
    \} \right)
    &= \left( N(x_{i + 1}) \cap \Gamma_{i + 1}\right) \cap  \left( \Sigma_1
    \cup \Sigma_2 \cup \{x_0, \ldots, x_i, y \} \right) \\
    &= U  \cap \left( \Sigma_1 \cup \Sigma_2 \cup \{x_0, \ldots, x_i, y \}
    \right) = \varnothing.
  \end{align*}

  Hence \eqref{condition-2} and \eqref{condition-3} are satisfied,
  and so it only remains to verify \eqref{condition-1}. It is routine
  to verify that $\dom(q_{i + 1}^{i + 1}) \setminus \dom(q_i^{i + 1}) = \{ x_0
  \}$. It follows from $x_0 \notin N(x_0)$ and \eqref{equation-3} that
  $N(x_{i + 1}) \cap \Gamma_i = \left( N(x_0) \cap \Gamma_0 \right) q_{i +
  1}^{i + 1}$. Since $\Gamma_{i + 1} = \Gamma_i \cup \{x_{i + 1}\}$ and $x_{i +
  1} \notin N(x_{i + 1})$
  \[
    N(x_{i + 1}) \cap \Gamma_i = N(x_{i + 1}) \cap \Gamma_{i + 1} =
    \left( N(x_0) \cap \Gamma_0 \right) q_{i + 1}^{i + 1}.
  \]

  Therefore, $q_{i + 1}$ satisfies the inductive hypothesis. Thus by induction
  on $i$, there is $q_m \in \I(H(n))^{< \omega}$ such that $q_m = q_0 \cup \{ (x_j,
  x_{j + 1}) : 0 \leq j \leq m - 1 \}$, $x_0 \notin \ran(q_m)$, $x_m \notin
  \dom(q_m)$, $y \notin \ran(q_m) \cup \dom(q_m)$, $q_m$ satisfies
  \eqref{condition-2}, \eqref{condition-3} and \eqref{condition-1}.

  Note that if $z \in \Sigma_1 \cup \Sigma_2 \setminus \{x, y\}$ then
  $z \notin \Gamma$ and by \eqref{condition-2} either $z \notin \dom(q_m) \cup
  \ran(q_m)$ or $z \in \Delta$. Hence
  \begin{align}
    \label{equation-between-ind}
    N( x_m) \cap \Gamma_m
    = \left( N(x_0) \cap \Gamma_0 \right) q_m^m
    &= \left( \left(N(x_0) \cap \left(\Gamma \cup \{x, y\} \cup \Sigma_1 \cup
    \Sigma_2 \setminus \Delta \right)\right) \cup \left(N(x_0) \cap \Delta
    \right) \right) q_m^m\\
    &= \left( N(x_0) \cap \Delta \right) q_m^m,\notag
  \end{align}
  since $x \notin N(x_0)$, $y \notin \dom(q_m)$ and all incomplete components on
  $\Gamma$ of $q$ are of length $m$.

  The next step is to inductively construct an extension $h = q_{2m} \in
  \I(H(n))^{< \omega}$ of $q_m$. Suppose that for $i \in \{ m, \ldots, 2m -
  2 \}$ there is an extension $q_i \in \I(H(n))^{< \omega}$ of the form
  $q_i = q_m \cup \{ (x_j, x_{j + 1}): m \leq j \leq i - 1 \}$ such that $x_0
  \notin \ran(q_i)$, $x_i \notin \dom(q_i)$, $y \notin \dom(q_i) \cup
  \ran(q_i)$,  and
  \begin{align}
    x_j &\notin \Sigma_1 \cup \Sigma_2 \cup \Delta \tag{\ref{condition-2}} \\
    N(x_j)&\cap \left( \Sigma_1 \cup \Sigma_2 \cup \{x_0, \ldots, x_{j - 1}, y
    \}\right) = \varnothing \tag{\ref{condition-3}} \\
    N(x_i) \cap \Gamma_i &= \left( N(y) \cap \dom\left( q_i^{i - 2m}\right)
    \right) q_i^{i - 2m} \label{condition-4}
  \end{align}
  for all $j \in \{1, \ldots, i\}$.

  We will now show that $q_m$ satisfies the inductive hypothesis. Note that
  \eqref{condition-2} and \eqref{condition-3} are the same as
  before, so we only need to verify \eqref{condition-4}. Since no
  incomplete components of $q_0$, which intersect $\Delta$, were extended in $q_m$,
  \eqref{condition-2} implies that $\left( \Delta \right)
  q_0^k = \left( \Delta \right) q_m^k$ for any $k \in \Z$. It follows from the
  hypothesis of the lemma that $\left(N(x_0) \cap \Delta\right)q_m^m \subseteq
  \dom(q_m^m)$ and $\left( N(x_0) \cap \Delta \right) q_m^m = \left( N(y) \cap
  \Delta \right)q_m^{-m}$. Hence by \eqref{equation-between-ind}
  \[
    N(x_m) \cap \Gamma_m =  \left(N(x_0) \cap \Delta \right) q_m^m = \left(
    N(y) \cap \Delta \right) q_m^{-m}.
  \]
  Suppose that $z \in N(y) \cap \dom(q_m^{-m})$. Then $z \in \dom(q_m) \cup
  \ran(q_m) =  \Gamma \cup \Delta \cup \{x_0, \ldots, x_m\}$. Note that all
  incomplete components of $q_m$, intersecting $\Gamma$ not trivially, are of length
  $m$. Hence $z \in \Delta \cup \{x_m\}$ and by \eqref{condition-3}
  we have that $x_m \notin N(y)$, thus $z \in \Delta$. Therefore $N(y) \cap
  \dom(q_m^{-m}) \subseteq N(y) \cap \Delta$, and so
  \[
    N(x_m) \cap \Gamma_m  = \left( N(y) \cap \Delta \right) q_m^{-m}
    = \left( N(y) \cap \dom(q_m^{-m}) \right) q_m^{-m}.
  \]
  Hence $q_m$ satisfies \eqref{condition-4} and the inductive hypothesis is
  satisfied for $i = m$.

  Let $U = \left( N(x_i)  \right) q_i$ and $V = \Gamma_i \setminus U$.  The
  sets $U$ and $V$ satisfy the hypothesis of the Alice's restaurant propery
  and thus we can find $x_{i + 1} \in H(n) \setminus \Gamma_i$  with $N(x_{i +
  1}) \cap \Gamma_i = U = \left( N(x_i) \right)q_i$. Then $N(x_{i + 1}) \cap
  \ran(q_i) = \left(N(x_i)\right)q_i$, and so $q_{i + 1} = q_i \cup \{ (x_i,
  x_{i + 1}) \} \in \I(H(n))^{< \omega}$ by Lemma~\ref{lem-5-neigh} and
  Corollary~\ref{cor-5-infty}. Since $x_{i + 1} \notin \Gamma_i$, we have that
  $x_{i + 1} \notin \{x_0, x_i, y\}$ implying that $x_0 \notin \ran(q_{i +
  1})$, $x_{i + 1} \notin \dom(q_{i + 1})$, and $y \notin \dom(q_{i + 1}) \cup
  \ran(q_{i + 1})$.

  Since $\dom(q_i) \subseteq \Gamma_i$, it follows from \eqref{condition-4}
  that
  \begin{equation}\label{equation-4}
    N( x_{i + 1}) \cap \Gamma_i = U = \left( N(x_i) \right) q_i  = \left(
    N(x_i) \cap \Gamma_i \right) q_i = \left( N(y) \cap \dom\left( q_i^{i - 2m}
    \right) \right) q_i^{i + 1 - 2m}.
  \end{equation}

  Since $\Sigma_1 \cup \Sigma_2 \cup \Delta \subseteq \Gamma_i$ and $x_{i + 1}$
  is chosen outside the set $\Gamma_i$ it follows that $x_{i + 1} \notin
  \Sigma_1 \cup \Sigma_2 \cup \Delta$. Then $x_j \notin \Sigma_1 \cup \Sigma_2
  \cup \Delta$ for all $j \in \{1, \ldots, i + 1\}$.

  We will now show that \eqref{condition-3} holds for $j = i + 1$. First of all
  note that $x_0, y \notin \ran(q_i)$, and since $U \subseteq \ran(q_i)$ we
  have that $x_0, y \notin U$. From \eqref{condition-3} we may deduce that $x_j
  \notin N(x_i)$, and thus $x_{j + 1} \notin \left( N(x_i) \right) q_i = U$,
  for all $j \in \{0, \ldots, i - 1\}$ i. e. $\{x_0, \ldots, x_i, y\} \cap U =
  \varnothing$. It follows from the hypothesis that $\Sigma_1 \cap \ran(q_0) =
  \varnothing$, and so \eqref{condition-2} implies \changed{$\Sigma_1 \cap \ran(q_i) =
  \varnothing$}. Since $\left(\Sigma_1 \cup \{x_0, y\} \right) \cap \ran(q_i) =
  \varnothing$ and $U \subseteq \ran(q_i)$, it follows that $\left(\Sigma_1
  \cup \{x_0, \ldots, x_i, y\}\right) \cap U = \varnothing$.

  It remains to show that $\Sigma_2 \cap U = \varnothing$. Suppose $z \in
  \Sigma_2 \cap U$. Then $z  \in \left(N(y) \cap \dom(q_i^{i - 2m})
  \right)q_i^{i + 1 - 2m}$ by \eqref{equation-4}. Note that $z \in U \subseteq
  \ran(q_i)$. Also it was shown in the previous paragraph that $z \neq x_j$ for
  all $j \in \{0, \ldots,i\}$. Hence $z \in \ran(q_0) \subseteq \Gamma \cup
  \Delta$. However by the hypothesis of the lemma $\Sigma_2 \subseteq H(n)
  \setminus \Gamma$, implying that $z \in \Delta$. Since $x_j \notin \Delta$
  for all $j \in \{1, \ldots, i\}$ by \eqref{condition-2} and $x_0 \notin
  \Delta$ by the hypothesis of the lemma, it follows that the incomplete component of
  $q_0$ containing $z$  was not extended in $q_i$. Moreover, $\Delta$ is a
  union of incomplete components of $q_0$,  and $z \in \dom(q_i^{2m - (i +
  1)})$, so $(z)q_i^{2m - (i + 1)} \in N(y) \cap \Delta$. By the
  hypothesis of the lemma $(z)q_i^{2m - (i + 1)} \in \ran(q_i^{2m})$.
  Then there is $u \in \dom(q_i^{2m})$ such that $(z)q_i^{2m - (i + 1)} =
  (u)q_i^{2m}$, and so $z = (u)q_i^{i + 1} \in \dom(q_i^{2m - (i + 1)})$. Hence
  $z \in \dom(q_i)$, since $2m > i + 1$. However, $z \in \Sigma_2$ and so $z
  \notin \dom(q_0)$, implying that $z \in \{x_0, \ldots, x_i\}$, which
  contradicts \eqref{condition-2}. Hence $U \cap \left( \Sigma_1
  \cup \Sigma_2 \cup \{x_0, \ldots, x_{i}, y\} \right) = \varnothing$, and
  since $\left( \Sigma_1 \cup \Sigma_2 \cup \{x_0, \ldots, x_{i }, y\} \right)
  \subseteq \Gamma_{i + 1}$ we have that for all $j \in \{1, \ldots, i + 1\}$
  \[
    N(x_j) \cap \left( \Sigma_1 \cup \Sigma_2 \cup \{x_0, \ldots, x_{j - 1}, y
    \} \right) = \varnothing.
  \]

  It is routine to verify that $\dom(q_{i + 1}^{i + 1 - 2m}) \setminus
  \dom(q_i^{i + 1 - 2m}) = \{ x_{i + 1}\}$. Since $x_{i + 1} \notin N(y)$ it
  follows from \eqref{equation-4} that $N(x_{i + 1}) \cap \Gamma_i =
  \left( N(y) \cap \dom \left(q_i^{i - 2m} \right) \right) q_{i + 1}^{i + 1 -
  2m}$. It is routine to check that $\dom \left( q_{i + 1}^{i + 1 -2m} \right)
  \setminus \dom \left( q_i^{i - 2m} \right)  \subseteq \{x_1, \ldots, x_{i +
  1}\}$. Then, by \eqref{condition-3}, we have that $x_j \notin N(y)$
  for all $j \in \{x_0, \ldots, i + 1\}$. Hence
  \[
    N(x_{i + 1}) \cap \Gamma_i = \left( N(y) \cap \dom \left(q_{i + 1}^{i + 1 -
    2m} \right) \right) q_{i + 1}^{i + 1 - 2m}.
  \]
  From the definition of $\Gamma_{i + 1}$ we obtain that $\Gamma_{i + 1} =
  \Gamma_i \cup \{x_{i + 1}\}$. However, $x_{i + 1} \notin N(x_{i + 1})$ thus
  \[
    N(x_{i + 1}) \cap \Gamma_{i + 1} = \left( N(y) \cap \dom\left(q_{i + 1}^{i
    + 1- 2m} \right) \right) q_{i + 1}^{i + 1 - 2m}.
  \]

  Therefore $q_{i + 1}$ satisfies the inductive hypothesis and hence we
  obtain $q_{2m - 1} = q_0 \cup \{(x_j, x_{j + 1}) : 0 \leq j \leq 2m -2 \} \in
  \aut(H(n))^{< \omega}$ such that $y \notin \dom(q_{2m - 1}) \cup \ran(q_{2m -
  1})$, $x_j \notin \Sigma_1 \cup \Sigma_2$, there are no edges between $x_j$
  and $\Sigma_1 \cup \Sigma_2$ for all $j \in \{1, \ldots, 2m - 1\}$,  and
  \[
    N( x_{2m - 1} ) \cap \Gamma_{2m - 1} =  \left( N(y) \right) q_{2m - 1}^{
    -1}.
  \]
  Therefore  $h = q_{2m} = q_{2m - 1} \cup \{ (x_{2m - 1},y) \} \in \I(H(n))^{<
  \omega}$ by Lemma~\ref{lem-5-neigh} and Corollary~\ref{cor-5-infty} is as
  required.
\end{proof}


\begin{lem}\label{lem-K_n-main}
  Let $q \in \I(H(n))^{< \omega}$, and let $p \in \P$ be such that the sets $\dom(q)
  \cup \ran(q)$ and $\dom(p) \cup \ran(p)$ are disjoint. Then there is an
  extension $h \in \I(H(n))^{<\omega}$ of $q$ and $m \in \mathbb{N}$ such that
  $h^{2m}$ extends $p$.
\end{lem}

\begin{proof}
  If necessary by extending $q$, using Corollary~\ref{cor-5-one-point-ext}, we may
  assume that all of the components of $q$ have length $m$ for some $m \in
  \mathbb{N}$.

  Let $\dom(p) = \{ x_1, \ldots, x_d\}$ for some $d \in \N$, let $q_0 = q$,
  $\Gamma = \dom(q_0) \cup \ran(q_0)$. We will now inductively define $q_i \in
  \I(H(n))^{< \omega}$, and once they are defined let $\Delta_i = \dom(q_i) \cup
  \ran(q_i) \setminus \Gamma$ for $i \in \{0, \ldots, d\}$. Suppose that for
  some $k \in \{0, \ldots, d - 1\}$ we have defined $q_k \in \I(H(n))^{< \omega}$, an
  extension of $q_0$, such that both $\Gamma$ and $\Delta_k$ are unions of
  incomplete components of $q_k$, that incomplete components of $q_k$ contained in $\Delta_k$
  are of length $2m + 1$, and the following are true
  \begin{align}
    x_j, (x_j)p &\notin \dom(q_k) \cup \ran(q_k) \label{condition-5} \\
    (x_i)q_k^{2m} &= (x_i) p \label{condition-6} \\
    N(x_j) \cap \Delta_k &\subseteq \dom(q_k^{2m}) \label{condition-7} \\
    \left( N(x_j) \cap \Delta_k \right) q_k^{2m} &= N\left((x_j)p\right) \cap
    \Delta_k \label{condition-8}
  \end{align}
  for all $i \in \{1, \ldots, k\}$ and $j \in \{k + 1, \ldots,
  d\}$.

  Let $\Sigma_1 = \dom(p)$ and $\Sigma_2 = \ran(p)$. We will show that the
  hypothesis of Lemma~\ref{lem-5-step-in-main} is satisfied by $q_k$, $x_{k +
  1}$, $(x_{k + 1})p$, $\Sigma_1$, and $\Sigma_2$. First of all, note that
  $x_{k + 1}, (x_{k + 1})p \notin \dom(q_k) \cup \ran(q_k)$ by condition
  \eqref{condition-5}. Also by the hypothesis of the lemma $\Sigma_1, \Sigma_2
  \subseteq H(n) \setminus \Gamma$. Note that \changed{$N(x_{k + 1}) \cap
  \Delta \subseteq \dom(q_k^{2m})$ and $(N(x_{k + 1}) \cap \Delta) q_k^{2m} =
  N((x_{k + 1})p) \cap \Delta$ immediately follow} from conditions
  \eqref{condition-7} and \eqref{condition-8}. Hence to apply
  Lemma~\ref{lem-5-step-in-main} we only need verify that $\Sigma_1 \cap
  \ran(q_k)  = \Sigma_2 \cap \dom(q_k) = \varnothing$. We will do so in the
  next two paragraphs.

  We will first show that $x_i \notin \ran(q_k)$ for all $i \in \{1, \ldots,
  d\}$. Suppose that $x_i \in \dom(q_k) \cup \ran(q_k)$, by the inductive
  hypothesis we can deduce that $i \leq k$. Since $\dom(p) \cap \Gamma =
  \varnothing$ by the hypothesis of the lemma, it then follows that $x_i \in
  \Delta_k$. Therefore, $x_i$ is on an incomplete component of length $2m + 1$ and $x_i
  \in \dom(q_k^{2m})$ by the inductive hypothesis, implying that $x_i \in
  \dom(q_k) \setminus \ran(q_k)$. Hence $\Sigma_1 \cap \ran(q_k) =
  \varnothing$.

  The argument that $\Sigma_2 \cap \dom(q_k) = \varnothing$ is similar to
  above. Let $(x_i)p \in \Sigma_2$. Suppose that $(x_i)p \in \dom(q_k) \cup
  \ran(q_k)$. Then we can deduce that $i \leq k$. Since $\ran(p) \cap \Gamma =
  \varnothing$ by the hypothesis of the lemma, it then follows that $(x_i)p \in
  \Delta_k$. Therefore, $(x_i)p$ is on an incomplete component of length $2m + 1$ and
  $(x_i)p \in \ran(q_k^{2m})$ by the inductive hypothesis, implying that
  $(x_i)p \in \ran(q_k) \setminus \dom(q_k)$.

  Hence by Lemma~\ref{lem-5-step-in-main} there is an extension $q_{k + 1} \in
  \I(H(n))^{< \omega}$ of $q_k$ such that $q_{k + 1} = q_k \cup \{(y_i, y_{i + 1}) :
  0 \leq i \leq 2m -1 \}$, $y_0 = x_{k + 1}$, $y_{2m} = (x_{k + 1})p$, there
  are no edges between $y_i$ and $\Sigma_1 \cup \Sigma_2$, and $y_i \notin
  \Sigma_1 \cup \Sigma_2$ for $i \in \{1, \ldots, 2m - 1\}$. Then by the choice
  of $\Sigma_1$, $\Sigma_2$, and the definition of $q_{k + 1}$
  \begin{align*}
    x_j, (x_j)p &\notin  \dom(q_{k + 1}) \cup \ran(q_{k + 1})\\
    (x_i)q_{k + 1}^{2m} &= (x_i) p
  \end{align*}
  for all $i \in \{1, \ldots, k + 1\}$ and $j \in \{k + 2, \ldots, d\}$.  It
  also follows from the definition of $q_{k + 1}$ that $\Delta_{k + 1} =
  \Delta_k \cup \{ y_i : 0 \leq i \leq 2m\}$ and thus $\Delta_{k + 1}$ is a
  union of incomplete components of $q_{k + 1}$ each of length $2m + 1$.

  Let $j \in \{k + 2, \ldots, d\}$, and let $z \in N(x_j) \cap \Delta_{k + 1}$.
  If $z \in \Delta_k$, then by the inductive hypothesis $z \in \dom(q_k^{2m})
  \subseteq \dom(q_{k + 1}^{2m})$ and
  \[
    (z)q_{k + 1}^{2m} = (z)q_k^{2m} \in N\left((x_j)p\right) \cap \Delta_k
    \subseteq N\left((x_j)p\right) \cap \Delta_{k + 1}.
  \]
  Otherwise $z \in \Delta_{k + 1} \setminus \Delta_k$. Hence $z = y_t$ for some
  $t \in \{0, \ldots, 2m\}$. However, $y_t$ is such that there are no edges
  between $y_t$ and $\dom(p)$ for $t \in \{1, \ldots, 2m - 1\}$. Then $z$ is
  either $y_0$ or $y_{2m}$. Since $p \in \mathcal{P}$ there are no edges
  between $x_j \in \dom(p)$ and $y_{2m} = (x_{k + 1})p \in \ran(p)$. Hence $z =
  y_0$ and thus $z \in \dom(q_{k + 1}^{2m})$. Since $z \in N(x_j)$  there is an
  edge between $x_j$ and $z = y_0 = x_{k + 1}$. Then it follows from the fact
  that $p$ is an isomorphism that there is an edge between $(x_j)p$ and $(x_{k
  + 1})p$. Hence $(z)q_{k + 1}^{2m} = y_{2m} = (x_{k + 1})p \in
  N\left((x_j)p\right)\cap \Delta_{k+1}$.  Since $z$ was arbitrary $N(x_j) \cap
  \Delta_{k + 1} \subseteq \dom(q_{k + 1}^{2m})$ and $\left(N(x_j) \cap
  \Delta_{k + 1}\right) q_{k  + 1}^{2m} \subseteq N\left((x_j)p\right) \cap
  \Delta_{k + 1}$.

  Let $z \in N\left((x_j)p\right) \cap \Delta_{k + 1}$. If $z \in \Delta_k$
  then it follows from the inductive hypothesis that
  \[
    z \in N\left((x_j)p\right) \cap \Delta_k = \left(N(x_j) \cap \Delta_k
    \right)q_k^{2m} \subseteq \left( N(x_j) \cap \Delta_{k + 1} \right)q_{k +
    1}^{2m}.
  \]
  Otherwise $z = y_j$ for some $j \in \{0, \ldots, 2m\}$. Similarly to above $z
  = y_{2m} = (x_{k + 1})p$ and since $p$ is an isomorphism $(z)q_{k + 1}^{-2m}
  = y_0 = x_{k + 1} \in N(x_j)$. Hence $z \in \left( N(x_j) \cap \Delta_{k +
  1}\right)q_{k + 1}^{2m}$, as $x_{k + 1} \in \Delta_{k + 1}$, and so
  \[
    \left( N(x_j) \cap \Delta_{k + 1} \right)q_{k + 1}^{2m} =
    N\left((x_j)p\right) \cap \Delta_{k +1}
  \]
  for all $j \in \{k + 2, \ldots, d\}$.

  Therefore $q_{k + 1}$ satisfies the inductive hypothesis an by induction
  there is $h = q_d \in \I(H(n))^{< \omega}$ and extension of $q$ such that $h^{2m}$
  is an extension of $p$.
\end{proof}


Finally we prove the main result of this section.

\begin{repthm}{thm-5-main}
  Let $f \in \aut(H(n))$ have infinite support. Then $D_f \cap \I(H(n))$ is
  comeagre in $\I(H(n))$.
\end{repthm}
\begin{proof}
 By Lemma \ref{lem-K_n-basis}
 \[
   D_f \cap \I(H(n)) = \bigcap_{p \in \P} \{g \in \I(H(n)) : \langle f, g \rangle \cap [p]
   \neq \varnothing \},
 \]
 and $\{g \in \aut(H(n)) : \langle f, g \rangle \cap [p] \neq \varnothing \}$
 is open by Lemma \ref{lem-E-open}, thus it is enough to show that $\{g \in \I(H(n))
 : \langle f, g \rangle \cap [p] \neq \varnothing \}$ is dense in $\I(H(n))$ for all
 $p\in \mathcal{P}$.

 Fix $p \in \P$, and let $q \in \changed{\I(H(n))^{< \omega}}$. If necessary by
 extending $q$ using Corollary~\ref{cor-5-one-point-ext}, we may assume that
 all of the components of $q$ have length $m$ for some $m \in \mathbb{N}$, and that
 $\ran(p) \cup \dom(p) \subseteq \dom(q)$. Suppose that $\ran(q) \backslash
 \dom(q) = \{ x_{1,0}, x_{2, 0}, \ldots, x_{d,0} \}$. Let $q_{1, 0 } = q$, and
 once $q_{i, j}$ is defined let $\Gamma_{i,j} = \dom(q_{i,j}) \cup
 \ran(q_{i,j})$ for all $i, j$. We will perform an induction on the elements of
 the set $\{(i, j): i \in \{1, \ldots, d\} \text{ and }j \in \{0, \ldots, m\}
 \}$, ordered lexicographically, to construct $q_{d, m} \in \I(H(n))^{< \omega}$ of
 the form $q_{d, m} = q_{1, 0} \cup \{ (x_{i, j}, x_{i, j + 1}) : 1 \leq i \leq
 d \text{ and } 0 \leq j \leq m - 1\}$ such that $x_{i, j} \in \supp(f)$ and
 $(x_{i, j})f \notin \ran(q_{d, m}) \cup \dom(q_{d, m})$ for all $i$ and all $j
 \geq 1$. In order to make the rest of the proof shorter, once we have defined
 $q_{i, m}$ for some $i < d$, we will set $q_{i + 1, 0} = q_{i, m}$, and
 similarly we denote $\Gamma_{i, -1} = \varnothing$ for all $i$.

 Suppose that for $k \in \{1, 2, \ldots, d\}$ and $t \in \{0, 1, \ldots, m -
 1\}$ we defined $q_{k, t} = q_{1, 0} \cup \{ (x_{i, j}, x_{i, j + 1}) : 1 \leq
 i \leq k \text{ and } 0 \leq j \leq t - 1\} \in \I(H(n))^{< \omega}$ such that
 $x_{k,t} \in \supp(f)$ and
 \[
   x_{k, t} \notin \Gamma_{k, t - 1} \cup \left( \Gamma_{k, t - 1} \right)f
   \cup \left( \Gamma_{k, t - 1} \right)f^{ -1}.
 \]
 Choose $x \in supp(f)$ such that $x \notin \Gamma_{k,t}$ which is possible
 since $supp(f)$ is infinite. Then by the Alice's restaurant propery there is
 a vertex $y \not \in \Gamma_{k,t} \cup \left( \Gamma_{k, t}\right) f^{-1}\cup
 \{x, (x)f \}$ such that there is an edge between $x$ and $y$, and there are no
 edges between $y$ and $\Gamma_{k,t} \cup \left( \Gamma_{k, t}\right)
 f^{-1}\cup \{(x)f \}$. Let
 \[
   U = \left( N(x_{k,t})\right)q_{k, t} \cup \{y\} \quad \text{and} \quad
   V = \left( \Gamma_{k,t} \cup \left( \Gamma_{k,t} \right)f \cup \left(
   \Gamma_{k, t}\right)f^{-1} \cup \{(y)f\}\right) \setminus
   U.
 \]
 Since  $\left( N(x_{k, t}) \right)q_{k, t}$ is $K_{n - 1}$-free and there are
 no edges between $y$ and $\left( N(x_{k, t}) \right)q_{k, t}$, the set $U$ is
 also $K_{n - 1}$-free. Hence by Alice's Restaurant Property there is a vertex
 \[
   x_{k, t + 1} \notin \Gamma_{k, t} \cup \left( \Gamma_{k, t} \right)f \cup
   \left( \Gamma_{k, t} \right)f^{-1} \cup \{y, (y)f\}
 \]
 such that $N(x_{k, t + 1}) \cap \left( U \cup V \right) = U$. It follows from
 $\ran(q_{k, t}) \subseteq \Gamma_{k, t}$ and $y \notin \Gamma_{k, t}$ that
 \[
   N(x_{k, t + 1}) \cap \ran(q_{k, t})  = U \cap \ran(q_{k, t}) = \left( \left(
   N(x_{k, t}) \right)q_{k, t} \cup \{y\} \right) \cap \ran(q_{k, t}) = \left(
   N(x_{k, t}) \right)q_{k, t},
 \]
 and so $q_{k, t + 1} = q_{k, t} \cup \{ (x_{k, t}, x_{k, t + 1}) \} \in \I(H(n))^{<
 \omega}$ by Lemma~\ref{lem-5-neigh} and Corollary~\ref{cor-5-infty}.

 It follows from $f$ being an automorphism and the existence of an edge between
 $x$ and $y$, that there is an edge between $(x)f$ and $(y)f$. However, there
 is no edge between $y$ and $(x)f$, thus it follows that $y \in \supp(f)$. The
 vertex $y$ was chosen so that $y \notin \Gamma_{k, t} \cup \left( \Gamma_{k,
 t} \right)f^{-1}$, and so $y, (y)f \notin \Gamma_{k, t}$. Since $\left(
 N(x_{k, t}) \right)q_{k, t} \subseteq \Gamma_{k, t}$ and $y \neq (y)f$, it
 follows that $(y)f \notin U$. By the choice of $x_{k, t + 1}$ there is an edge
 between $x_{k,t + 1}$ and $y$ and there are no edges between $x_{k, t + 1}$
 and $(y)f$, thus $x_{k, t + 1} \in \supp(f)$. Hence $q_{k, t + 1}$ satisfies
 the inductive hypothesis.

 This way we can obtain $q_{d, m} \in \I(H(n))^{< \omega}$ such that for all $i$ and
 all $j \geq 1$
 \[
   x_{i, j} \notin \Gamma_{i, j - 1} \cup \left( \Gamma_{i, j - 1} \right)f
   \cup \left( \Gamma_{i, j - 1} \right)f^{ -1}.
 \]
 Hence $(x_{i, j})f \notin \Gamma_{i, j - 1}$. Also if $(x_{i, j})f = x_{i',
 j'}$, where $(i, j) < (i', j')$ lexicographically, then $x_{i', j'} \in \left(
 \Gamma_{i, j} \right)f$ which is impossible. Therefore, $(x_{i, j})f \notin
 \ran(q_{d, m}) \cup \dom(q_{d, m})$ and thus
  \[
    \left( \left( \dom(q) \right )q_{k,m}^mf \right) \cap (\dom(q_{k,m}) \cup
    \ran(q_{k,m})) = \varnothing .
  \]
  Then since $p \in \P$ and $\P$ is closed under conjugation, $u = \left(
  q_{k,m}^mf \right)^{-1} p q_{k,m}^m f \in \P$. Recall that $\ran(p) \cup
  \dom(p) \subseteq \dom(q)$, thus the partial isomorphisms $q_{k,m}$ and $u$
  satisfy the hypothesis of Lemma \ref{lem-K_n-main}.  Hence there is an
  extension $h \in \I(H(n))^{< \omega}$ of $q_{k,m}$ and \changed{$l \in \Z$
  such that $h^{2l}$} extends $u$. Therefore $h^m f \changed{h^{2l}} \left( h^m
  f \right)^{-1}$ extends $p$ and thus
  \[
    \{g \in \I(H(n)) : \langle f, g \rangle \cap [p] \neq \varnothing \} \cap [q]
    \neq \varnothing.
  \]
  Since $q \in \I(H(n))^{<\omega}$ was arbitrary we get that $\{g \in \I(H(n)) : \langle f,
  g \rangle \cap [p] \neq \varnothing \}$ is dense in $\I(H(n))$.
\end{proof}

\section{Infinitely many finite complete graphs: $\omega K_n$}

In this section, we consider the ultrahomogeneous graphs $\omega K_n$ for $n\in
\N$, $n>0$. Throughout the section we assume that $n\in\N$, $n > 0$, is fixed
and that the connected components of $\omega K_n$ are $\set{L_i}{i\in \Z}$. We
will first prove a couple of technical results.

We begin by characterising the elements of $\I_{\Sigma}(\omega K_n)$ in a lemma analogous
to Corollary~\ref{cor-5-infty}.


\begin{lem}
  \label{lem-extensions-infinite-finite}
  Let $q \in \aut(\omega K_n)^{< \omega}$ be such that $\dom(q)$ is a union of
  connected components, and there is $\Sigma \subseteq \dom(q)$ which intersects
  every component of $q$ in exactly one vertex. Then $q
  \in \I_\Sigma(\omega K_n)^{<\omega}$ if and only if $\overline{q}$ has no complete
  components.
\end{lem}


\begin{prop}
  Let $\Sigma$ be a finite subset of $\omega K_n$.  Then $\I_\Sigma(\omega K_n)$ is
  non-empty if and only if  $|\Sigma|$ is a multiple of $n$ and if $r =
  |\Sigma|/n$, there is partition $\{P_1, \ldots, P_r\}$ of $\Z$ such that,
  $P_i$ is infinite and \[\sum_{j \in P_i} |L_j \cap \Sigma| = n\] for all $i
  \in \{1, \ldots, r\}$.
\end{prop}
\begin{proof}
  $(\Rightarrow)$
  Let $f \in \I_\Sigma(\omega K_n)$. If $x\in L_i$ and $(x)f\in L_j$, then, since $f$ is an
  automorphism, $(L_i)f=L_j$.  Moreover, if $(L_i)f^m = L_i$ for some $m\in
  \Z$, then $(L_i)f^{rm} = L_i$ for all $r\in \Z$, and since $L_i$ is finite,
  $f$ would have a finite cycle. Hence $(L_i)f^m \neq L_i$ for all $m \in \Z$,
  and so every vertex in $L_i$ is on a separate orbit of $f$.

  Let $k_1, \ldots, k_r \in \Z$ be orbit representatives of $\overline{f}$.
  Since for every orbit of $\overline{f}$ there are $n$ orbits in $f$, it
  follows that $rn = |\Sigma|$. 
  It follows from Lemma~\ref{lem-extensions-infinite-finite}, $\overline{f}$
  has no complete component. So $\left(L_{k_i}\right)f^m =
  \left(L_{k_{i'}}\right)f^{m'}$ if and only if $i = i'$ and $m = m'$.  Hence
  \[
    n =  \left| \Sigma \cap \left( \bigcup_{m \in \mathbb{Z}} \left( L_{k_i}
         \right)f^m \right)\right|
      = \left|  \bigcup_{m \in \mathbb{Z}} \left( \Sigma \cap \left( L_{k_i}
        \right)f^m \right)\right|
      = \sum_{m \in \mathbb{Z}} \left|\Sigma \cap \left( L_{k_i} \right)f^m
        \right|
  \]
  for all $i \in \{1, \ldots, r\}$. Let $P_i = \{\left( k_i \right)
  \overline{f}^m : m \in \Z\}$ where $i \in \{1, \ldots, r\}$. Then $\{P_1,
  \ldots, P_r\}$ is the required partition.

  ($\Leftarrow$)
  For $i \in \{1, \ldots, r\}$, let $P_i = \{ k_{i, j} : j \in \Z\}$. Define $f
  \in \I_{\Sigma}(\omega K_n)$ to be such that
  \[
    (k_{i,j})\overline{f}=k_{i, j + 1}
  \]
  for all $i \in \{1, \ldots, r\}$ and $j \in \Z$ by inductively defining $f$
  on $\bigcup_{j\in\Z} L_{k_{i, j}}$ for each $i$ independently.

  Let $i \in \{1, \ldots,r\}$ be arbitrary. Then $| L_{k_{i, 0}}\cap \Sigma | +
  | L_{k_{i, 1}} \cap \Sigma | \leq n$. Since $L_{k_{i, 0}}$ and $L_{k_{i, 1}}$
  are both of size $n$, there exists a bijection $q_1:L_{k_{i, 0}}\to
  L_{k_{i, 1}}$ such that for every $x \in L_{k_{i, 0}}$ at most one of the
  points $x$ and $(x)q_1$ is in $\Sigma$. Suppose that for some $m \in \N$
  we have defined a bijection
  \[
    q_{2m + 1} : \bigcup_{j = - m}^m L_{k_{i, j}} \rightarrow \bigcup_{k = -
    m + 1}^{m + 1} L_{k_{i, j}}
  \]
  such that every incomplete component of $q_{2m+1}$ intersects $\Sigma$ in at most
  one point.

  Let $t = \sum_{j = -m }^{m + 1} |L_{k_{i, j}} \cap \Sigma|$. Then there are
  $n - t$ incomplete components of $q_{2m+1}$ which have empty intersection with
  $\Sigma$. Since $\sum_{j = -m - 1}^{m + 1} |L_{i(j,r)} \cap \Sigma| \leq n$,
  it follows that $|L_{k_{i, - m - 1}} \cap \Sigma| \leq n - t$. Hence there
  exists a bijection $\phi : L_{k_{i, -m - 1}} \rightarrow L_{k_{i, -m}}$  such
  that for every $x \in L_{k_{i, - m - 1}} \cap \Sigma$, the value $(x)\phi$
  belongs to an incomplete component of $q_{2m+1}$ which contains no points from
  $\Sigma$. If we set $q_{2m + 2} = q_{2m + 1} \cup \phi$, then every
  incomplete component of $q_{2m + 2}$ intersects $\Sigma$ in at most one point.
  Similarly we can extend $q_{2m + 2}$ to $q_{2m + 3}$ by adding a
  bijection from $L_{k_{i, m + 1}}$ to $L_{k_{i, m + 2}}$.

  Hence by induction induction
  \[
    f_i = \bigcup_{m \in \Z} q_{2m + 1}
  \]
  is an automorphism of $\bigcup_{j \in \Z} L_{k_{i, j}}$ and every orbit of
  $f_i$ intersects $\Sigma$ exactly once. The required $f$ is then just the
  function $\bigcup_{i = 1}^{r} f_i$.
\end{proof}

\begin{lem}\label{lem-omegaK_n-basis}
 Let $\Sigma \subseteq \omega K_n$ be finite, and let $\mathcal{F}$ consist of
 those $g\in \aut(\omega K_n)^{< \omega}$ where \changed{the sets $\dom(g)$ and
 $\ran(g)$ are disjoint, both are} unions of connected components of $\omega
 K_n$, and $\overline{g}$ does not have any complete components. Then
 \[
   D_f \cap \I_\Sigma(\omega K_n) = \bigcap_{p \in \mathcal{F}} \{g \in
   \I_\Sigma : \langle f, g \rangle \cap [p] \neq \varnothing \}.
 \]
\end{lem}

\begin{proof}
  Recall that
  \[
    D_f \cap \I_\Sigma(\omega K_n)
    = \{g \in \I_\Sigma(\omega K_n): \langle f, g\rangle \text{ is dense in } \aut(\omega
    K_n) \}
    = \bigcap_{q \in \aut(\omega K_n)^{<\omega}} \{g \in \I_\Sigma(\omega K_n) : \langle f,
    g \rangle \cap [q] \neq \varnothing \}.
  \]

  $(\subseteq)$ This follows immediately since $\mathcal{F} \subseteq
  \aut(\omega K_n)^{<\omega}$.

  $(\supseteq)$ Let $g \in \bigcap_{p \in \mathcal{F}} \{g \in \I_\Sigma(\omega
  K_n) : \langle f, g \rangle \cap [p] \neq \varnothing \}$, let $q \in
  \aut(\omega K_n)^{< \omega}$, and let
  \[
    \Gamma = \bigcup \{L_i: \dom(q) \cap L_i \neq \varnothing \}.
  \]
  If $h \in \aut(\omega K_n)$ is an extension of $q$, $x \in L_i$, and $(x)h
  \in L_j$, then $(L_i)h = L_j$. Hence $(\Gamma)h$ is a union of connected
  components of $\omega K_n$. Let $r = h|_\Gamma$. Then $[r] \subseteq [q]$.

  Let $\Gamma$ be a subgraph of $\omega K_n$ such that $\Gamma$ is isomorphic
  to $\dom(r)$, $\Gamma \cap \left( \dom(r) \cup \ran(r) \right) =
  \varnothing$,  and such that there are no edges between $\Gamma$ and $\dom(r)
  \cup \ran(r)$. Let $p$ be any isomorphism between $\dom(r)$ and $\Gamma$.
  Note that since $\dom(r)$ is a union of connected components of $\omega K_n$
  so is $\Gamma$. Since $\omega K_n$ is ultrahomogeonous, we have that $p \in
  \aut(\omega K_n)^{< \omega}$. Then $\dom(p) = \dom(r)$, $\ran(p) =
  \dom(p^{-1}r) = \Gamma$ and $\ran(p^{-1}r) = \ran(r)$. Hence $p, p^{-1}r \in
  \mathcal{F}$. By the choice of $g$ there are $h_1, h_2 \in \langle f, g
  \rangle$ such that $h_1 \in [p]$ and $h_2 \in [p^{-1}r]$. Therefore $h_1h_2
  \in [r] \subseteq [q]$ and $h_1h_2 \in \langle f, g \rangle$, thus $\langle
  f, g\rangle \cap [q] \neq \varnothing$. Since $q$ was arbitrary, $g \in
  \bigcap_{q \in \aut(\omega K_n)^{<\omega}} \{g \in \I_\Sigma(\omega K_n) : \langle f, g
  \rangle \cap [q] \neq \varnothing \}$.
\end{proof}

We will now prove Theorem~\ref{thm-main}(ii), which we restate for the sake of
convenience.

\begin{thm}
  \label{thm-infinite-finite}
  Let $f \in \aut(\omega K_n)$ be such that $\supp(\overline{f})$ is infinite,
  and let $\Sigma$ be a finite subset of $\omega K_n$. Then $D_f \cap
  \I_\Sigma(\omega K_n)$ is comeagre in $\I_\Sigma(\omega K_n)$.
\end{thm}
\begin{proof}
  If $\I_{\Sigma}(\omega K_n)$ is empty, then the result holds trivially. So,
  for the remainder of the proof, we will suppose that $\I_{\Sigma}(\omega
  K_n)$ is non-empty.

  By Lemma~\ref{lem-omegaK_n-basis}
  \[
    D_f \cap \I_\Sigma(\omega K_n) = \bigcap_{p \in \mathcal{F}} \{g \in
    \I_\Sigma(\omega K_n) : \langle f, g \rangle \cap [p] \neq \varnothing \},
  \]
  and by Lemma~\ref{lem-E-open} the set $\{g \in \I_\Sigma(\omega K_n) :
  \langle f, g \rangle \cap [p] \neq \varnothing \}$ is open, so it's suffices
  to show that the aforementioned set is dense in $\I_\Sigma(\omega K_n)$.

  Let $p\in \mathcal{F}$ and let $q \in \I_\Sigma(\omega K_n)^{<\omega}$. We
  will show that there exists an extension $\changed{h \in \I_\Sigma(\omega
  K_n)^{<\omega}}$ of $q$ such that every extension $g\in \I_{\Sigma}(\omega
  K_n)$ of $h$ satisfies $\langle f,g\rangle\cap [p]\not=\varnothing$. If
  necessary, by extending $q$, we can assume without loss of generality that
  $\dom(q)$ is a union of connected components of $\omega K_n$, and that $q$
  has $|\Sigma|$ incomplete components each of some fixed length $m$, and that
  $\Sigma \cup \dom(p) \cup \ran(p) \subseteq \dom(q)$. Then $\ran(q) \setminus
  \dom(q)$ is a union connected components $L_{1, 0}, \ldots, L_{N, 0}$ for
  some $N \in \N$.

  Let $q_{1,0} = q$ and once $q_{i, j}$ is defined let $\Gamma_{i, j} =
  \dom(q_{i, j}) \cup \ran(q_{i, j})$. Suppose there is $i \in \{0, \ldots, m -
  1\}$ such that $q_{1, i} \in \I_\Sigma(\omega K_n)^{< \omega}$ is defined
  such that $\dom(q_{1, i})$ is a union of connected components, and $(x)q_{1,
  i}^j \in L_{1, j} $ for all $x \in L_{1, 0}$ and $j \in \{1, \ldots, i\}$.
  Since $\overline{f}$ has infinite support, there exists a connected component
  $L_{1,i + 1}$ of $\omega K_n$ such that $(L_{1,i + 1})f\not=L_{1,i + 1}$ and
  \[
    L_{1,i + 1} \cap \left( \Gamma_{1,i} \cup (\Gamma_{1,i})f \cup
    (\Gamma_{1,i})f^{-1}\right) =\varnothing.
  \]
  Let $\phi : L_{1, i} : \to L_{1, i + 1}$ be a bijection, and let  $q_{1,i +
  1} = q_{1, i} \cup \phi$. Then $q_{1, i + 1} \in \I_\Sigma(\omega K_n)^{< \omega}$ by
  Lemma~\ref{lem-extensions-infinite-finite}. Also by definition of $q_{1, i +
  1}$ the set $\dom(q_{1, i + 1}) = \dom(q_{1, i}) \cup L_{1, i}$ is a union of
  connected components, and $(x)q_{1, i + 1}^{i + 1} \in L_{1, i + 1}$ for all
  $x \in L_{1, 0}$. Hence by induction there is $q_{1, m} \in
  \I_\Sigma(\omega K_n)^{<\omega}$ such that $\dom(q_{1, m})$ is a union of connected
  components of $\omega K_n$, and $(x)q_{1, m}^j \in L_{1, j}$ for all $x \in
  L_{1, 0}$ and $j \in \{1, \ldots, m\}$.

  Let $q_{2, 0} = q_{1, m}$ and suppose for some $i \in \changed{\{2, \ldots,
  N\}}$ there is $q_{i, 0} \in \I_\Sigma(\omega K_n)^{< \omega}$ an extension
  of $q$ such that $\dom(q_{i, 0})$ is a union of connected components of
  $\omega K_n$, and $(x)q_{i, 0}^k \in L_{j, k}$ for all $x \in L_{j, 0}$, all
  $j \in \{1,\ldots i - 1\}$, and all $k \in \{1, \ldots, m\}$. The same
  argument as before can be used to define $q_{i, m} \in \I_\Sigma(\omega
  K_n)^{< \omega}$ an extension of $q$ such that such that $\dom(q_{i, m})$ is
  a union of connected components of $\omega K_n$, and $(x)q_{i, m}^k \in L_{j,
  k}$ for all $x \in L_{j, 0}$, all $j \in \{1, \ldots, i\}$, and all $k \in
  \{1, \ldots, m\}$. Hence by induction $\dom(q_{N, m})$ is a union of
  connected components of $\omega K_n$, and $(x)q_{N, m}^k \in L_{j, k}$ for
  all $x \in L_{j, 0}$, all $j \in \{1, \ldots, N\}$, and all $k \in \{1,
  \ldots, m\}$.

  We will show that $q_{N, m}$ is the desired extension of $q$. Let $r = q_{N,
  m}$. If $x \in L_{i, 0}$ for some $i \in \{1, \ldots, N\}$ and $j \in \{1,
  \ldots, m\}$, then
  \begin{equation}
    \label{equation-1}
    (x)r^j\in L_{i,j}\subseteq \Gamma_{i, j}
  \end{equation}
  and so by the choice of $L_{i, j}$ we have $(x)r^j \notin \Gamma_{i, j-1} \cup
  (\Gamma_{i, j-1})f \cup (\Gamma_{i, j-1})f^{-1}$ for all $j\in \{1,\ldots,
  m\}$. In particular,
  \begin{equation}
    \label{equation-2}
    (x)r^j f\not\in \Gamma_{i,j-1}\quad\text{and}\quad(x)r^j f^{-1}\not\in
    \Gamma_{i,j-1}
  \end{equation}
  for all $i \in \{1, \ldots, N\}$ and $j \in \{1, \ldots, m\}$.

  Let $x \in L_{i, 0}$ and $y \in L_{j, 0}$ for any $i,j \in \{1,\ldots, N\}$.
  We will show that $\left((x)r^k\right)f  \neq (y)r^l$ for all $k\in
  \{1,\ldots, m\}$ and $l\in \{-m+1, \ldots, m\}$. If $i=j$ and $k=l$, then,
  since $(x)r^k, (y)r^l \in L_{i,k}$ by \eqref{equation-1}, and
  $(L_{i,k})f\not= L_{i, k}$ by the choice of $L_{i, j}$, it follows that
  $\left((x)r^k\right)f\not=(y)r^{l}$. Hence we may assume that  $(i, k) \not=
  (j, l)$. There are three cases to consider.

  If $l \leq 0$, then $(y)r^{l} \in \dom(q_{1,0})\cup
  \ran(q_{1,0})=\Gamma_{1,0}\subseteq \Gamma_{i,k}$ and $(x)r^kf\notin
  \Gamma_{i,k}$ by \eqref{equation-2}, and so $(x)r^kf\not=(y)r^{l}$.

  Suppose that $i > j$ and $l > 0$, or $i = j$ and $k > l > 0$. Then $(y)r^{l}
  \in\Gamma_{j, l}$ by \eqref{equation-1}. By the assumption of this case,
  $\Gamma_{j, l} \subseteq \Gamma_{i, k-1}$ and $((x)r^k)f \not\in \Gamma_{i,
  k-1}$ by \eqref{equation-2}. Thus $((x)r^k)f\not=(y)r^{l}$, in this case.

  Suppose that $i < j$ and $l > 0$, or $i = j $ and $k < l$. Then $\Gamma_{i,
  k} \subseteq \Gamma_{j, l - 1}$. Since $((y)r^l)f^{-1} \notin \Gamma_{j,l-1}$
  by \eqref{equation-2}, it follows that $((y)r^l)f^{-1} \not\in \Gamma_{i,
  k}$, and so $((x)r^k)f\not=(y)r^{l}$. Therefore, in all three cases
  $((x)r^k)f \notin \ran(r) \cup \dom(r)$.

  Recall that $\dom(p) \cup \ran(p) \subseteq \dom(q)$ and that every point in
  $\dom(q)$ can be expressed as $(x)r^j$ for some $x \in \bigcup_{i = 1}^N
  L_{i, 0}$ and $j \in \changed{\{-m + 1, \ldots, -1\}}$. Define $u = (r^m
  f)^{-1}p(r^mf)$. Since $\overline{p}$ has no complete components, the same is true
  for $\overline{u}$. Also
  \[
    \dom(u) \cup \ran(u) \subseteq \{\left( (x)r^j \right)f : 1 \leq j \leq m,
    \text{ and } x \in  L_{i, 0} \text{ for some } i\}
  \]
  and hence $(\dom(u)\cup \ran(u)) \cap (\dom(r) \cup \ran(r)) =
  \varnothing$.

  Suppose $\dom(u) \backslash \ran(u) = \bigcup_{k=1}^M L_{i_k}$, and let $n_k$
  be the largest integer such that $(L_{i_k})u^{n_k}$ is defined for some $k
  \in \{1, \ldots, M\}$. Define $v$ to be an extension of $u$ by bijections
  $(L_{i_k})u^{n_k} \to L_{i_{k + 1}}$ for all $k \in \{1, \ldots, M - 1\}$.
  Then the domain of $v$ is a union of connected components of the graph, and
  $v$ has no complete components, since neither $p$ nor $u$ do. Finally choose
  any bijection $\psi : \changed{L_{N,m}} \to L_{i_1}$ \changed{and define $h
  = r \cup \psi \cup v$}. Then the number of components in
  $h$ \changed{is $|\Sigma|$ and so $h \in \I_\Sigma(\omega
  K_n)^{< \omega}$} by Lemma~\ref{lem-extensions-infinite-finite}.  Let $g \in
  \I_{\Sigma}(\omega K_n)$ be an extension of $h$. By definition of $u$ we have
  that $(h^mf)h(h^mf)^{-1}$ extends $p$, thus $\langle f,g \rangle \cup [p]
  \neq \varnothing$ and $g \in [q]$. Therefore the set $\{g \in
  \I_\Sigma(\omega K_n) : \langle f, g \rangle \cap [p] \neq \varnothing \}$ is
  dense in $\I_\Sigma(\omega K_n)$.
\end{proof}

The following is an immediate corollary of Lemma~\ref{lem-no-sigma}.

\begin{cor}
  Let $f \in \aut(\omega K_n)$ be such that $\supp(\overline{f})$ is infinite.
  Then $D_f \cap \I(\omega K_n)$ is comeagre in $\I(\omega K_n)$.
\end{cor}


\section{Finitely many infinite complete graphs: $n K_\omega$}

In this section we will consider the ultrahomogeneous graph $n K_\omega$ for a
fixed $n \in \mathbb{N}$ such that $n \geq 2$. Throughout this section let
$L_1, L_2, \ldots, L_n$ be the connected components of $n K_\omega$.  Recall
that, if $f \in \aut(nK_\omega)$ and $\Sigma \subseteq n K_\omega$ is finite,
then
\[
  \mathcal{A}_f = \{ g \in \aut(nK_\omega) : \langle \overline{f}, \overline{g}
  \rangle = S_n \}
\]
and
\[
  \mathcal{A}_{f, \Sigma} = \{g \in \mathcal{A}_f : \Sigma \text{ is a set of
  orbit representatives of } g\}.
\]

To specify when $\A_f$ is non-empty, we require the following classical
theorem.

\begin{prop}[cf. \cite{piccard1939aa}]\label{prop-piccard}
  Let $a\in S_n$ be a non-identity element and let $n\in \N$ be such that
  $n\not=4$, or $n = 4$ and $a \notin \{(1 \; 2)(3 \; 4), (1 \; 3)(2
  \; 4), (1 \; 4)(2 \; 3) \}$. Then there exists $b\in S_n$ such that $\langle
  a,b\rangle=S_n$.
\end{prop}

It follows by Proposition~\ref{prop-piccard}, that $\A_f\not=\varnothing$
if and only if  $n\not=4$, or $n = 4$ and $\overline{f} \notin \{(1 \; 2)(3 \;
4), (1 \; 3)(2 \; 4), (1 \; 4)(2 \; 3) \}$.

Next, we show that $\A_f$ and $\A_{f, \Sigma}$ are Baire spaces and
thus we can consider their comeagre subsets.

\begin{lem}\label{lem-Af-closed}
  Let $\Sigma \subseteq n K_\omega$ be finite. Then $\A_f$ is closed and
  $\A_{f, \Sigma}$ is a Baire space.
\end{lem}

\begin{proof}
  Let $g \in \aut(n K_\omega) \setminus \A_f$. Then $\langle \overline{f},
  \overline{g} \rangle \neq S_n$. Let $\Gamma \subseteq n K_\omega$ be a finite
  set such that $L_i \cap \Gamma \neq \varnothing$ for all $i \in \{1, \ldots,
  n \}$. Then for all $h \in [g|_{ \Gamma}]$ we have that $\overline{h} =
  \overline{g}$ and thus $h \notin \A_f$. Therefore, the open set $[g|_{
  \Gamma}]$ is a subset of $\aut(n K_\omega) \setminus \A_f$ and thus $\A_f$ is
  closed, and hence Baire. Then, by Lemma~\ref{lem-G-delta}, $\A_{f, \Sigma}$
  is a Baire space.
\end{proof}

The following lemma combined with Lemma~\ref{lem-Af-closed} demonstrates that
$D_f$ is not dense, and thus not comeagre, in any set which is not contained
in $\A_f$.
\begin{lem}
  If $g \in \aut(nK_\omega)$ is such that $\langle f, g \rangle$ is dense
  in $\aut(nK_\omega)$, then $\langle \overline{f}, \overline{g} \rangle = S_n$.
  In other words, $D_f \subseteq \A_f$.
\end{lem}
\begin{proof}
  Let $g \in D_f$. Then $\langle f, g \rangle$ is dense in $\aut(nK_\omega)$.
  Let $\sigma \in S_n$ be arbitrary. Then it is straightforward to verify that
  there is $q \in \aut(nK_\omega)^{<\omega}$ such that $\overline{q} = \sigma$.
  Since $\langle f, g\rangle$ is dense, it follows that there is a product $h
  \in \langle f, g \rangle$ which extends $q$. Therefore $\sigma = \overline{h}
  \in \langle \overline{f}, \overline{g} \rangle$ which implies that $g \in
  \A_f$.
\end{proof}

Let $f \in \aut(n K_\omega)$. Then $f$ is called \emph{non-stabilizing} if for
all $\Gamma \subsetneq nK_\omega$, all $x \in \Gamma$ and all $q \in \A_f^{<
\omega}$ there is $g \in [q] \cap \A_f$ such that $(x)h \notin \Gamma$ for some
$h \in \langle f, g \rangle$. We say that $f\in \aut(nK_{\omega})$ is
\emph{stabilizing} if it is not non-stabilizing.

\begin{prop}\label{prop-f-stabilizing}
  Let $f \in \aut(nK_\omega)$ be such that $\A_f\not=\varnothing$. Then $f$ is
  stabilizing if and only if there is a finite subset $\Gamma$ of $nK_\omega$
  such that $f$ stabilises \changed{$\Lambda$} setwise and
  \[
  |L_i \cap \changed{\Lambda} | = |L_j \cap \changed{\Lambda} |
  \]
  for all $i, j \in \{1,2,\dots,n\}$.
\end{prop}
\begin{proof}
  ($\Rightarrow$) Let $f$ be a stabilizing automorphism of $nK_\omega$.  By the
  definition of being non-stabilizing, there is $\Delta \subsetneq n K_\omega$,
  $x \in \Delta$ and $q \in \A_f^{< \omega}$ such that for all $g \in [q] \cap
  \A_f$ and all $h \in \langle f, g \rangle$ we have that $(x)h \in \Delta$. If
  necessary by taking an extension   of $q$, we may assume without loss of
  generality that $\overline{q} \in S_n$. Fix any  $g \in [q] \cap \A_f$, and
  let $\Gamma = \{(x)h : h \in \langle f, g \rangle \} \subseteq \Delta$. Then
  the subgroup $\langle f, g \rangle $ stabilises $\Gamma$. Hence $f$ also
  setwise stabilises $\Gamma$.  Let $i, j \in \{1, \ldots, n\}$ be arbitrary.
  Since $g \in \A_f$ we may choose $h \in \langle f, g \rangle$ such that
  $(i)\overline{h} = j$. By the definition, $\Gamma$ is setwise stabilised by
  $h$ and thus
  \[
    \left( L_i \cap \Gamma \right) h \subseteq L_j \cap \Gamma
    \quad \text{and} \quad \left( L_j \cap \Gamma \right) h^{-1}
    \subseteq L_i \cap \Gamma,
  \]
  as both $h$ and $h^{-1}$ are bijections. It follows that $|L_i \cap \Gamma| =
  |L_j \cap \Gamma |$. \changed{Since $\langle f, g \rangle$ also setwise
  stabilises $n K_\omega \setminus \Gamma$, the same argument shows that $|L_i
  \cap (n K_\omega \setminus \Gamma)| = |L_j \cap (n K_\omega \setminus
  \Gamma)|$.

  Finally, suppose that both $\Gamma$ and $n K_\omega \setminus \Gamma$ are
  infinite. Then for every $i \in \{1, \ldots, n\}$ the sets $(\Gamma \cap L_i)
  \setminus (\dom(q) \cup \ran(q))$ and $((n K_\omega \setminus \Gamma) \cap
  L_i)
  \setminus (\dom(q) \cup \ran(q) )$ are non-empty. Hence for every $i \in \{1,
  \ldots, n\}$ there are $x \in L_i \cap \Gamma$ and  an extension $g \in
  \aut(n K_\omega)$ of $q$ such that $(x)g \in n K_\omega \setminus \Gamma$,
  contradicting the choice of $\Gamma$. Therefore either $\Gamma$ or $n
  K_\omega \setminus \Gamma$ is finite, and since both sets are 
  stabilised setwise by $f$, one of them is the required set $\Lambda$.}

  ($\Leftarrow$) Let $m = |L_i \cap \changed{\Lambda}|$ for any, and all, $i\in
  \{1,2,\ldots, n\}$ and let $L_i \cap \changed{\Lambda} = \{ \gamma(i, j) : 1 \leq j \leq m
  \}$. Since $\A_f$
  is non-empty there is $\sigma \in S_n$ such that $\langle \overline{f}, \sigma
  \rangle = S_n$.
  Define a finite isomorphism $q : \changed{\Lambda} \to \changed{\Lambda}$ such that $\left(
  \gamma(i, j) \right) q = \gamma((i)\sigma, j)$ for all $j\in \{1,\ldots, m\}$.
  Then $\overline{q} = \sigma$ and so $q \in \A_f^{<\omega}$.  Moreover, $\changed{\Lambda}$
  is a union of cycles of $q$ and hence $\langle f, g\rangle$ stabilises
  $\changed{\Lambda}$ for  any $g \in [q]$. Therefore, $f$ is stabilizing.
\end{proof}

The following theorem is a restatement of Theorem~\ref{thm-main}(iii), and it
is the main result in this section. 

\begin{thm}\label{thm-main-II}
  Let $f \in \aut(n K_\omega)$. Then $f$ is non-stabilizing if and only if
  $D_f$ is comeagre in $\A_f$. Furthermore, if $f$ is non-stabilizing and
  $\Sigma$ is any finite subset of $nK_{\omega}$, then $D_f\cap \A_{f,\Sigma}$
  is comeagre in $\A_{f, \Sigma}$.
\end{thm}

If $f$ is stabilizing, and $D_f\cap \A_{f,\Sigma}$ is comeagre in $\A_{f,
\Sigma}$ for all $\Sigma$, then by Lemma~\ref{lem-no-sigma}, $D_f\cap \A_f$ is
comeagre in $\A_f$ and so, by Theorem~\ref{thm-main-II}, $f$ is
non-stabilizing, which is a contradiction.  Hence if $f$ is stabilizing, then
there exists $\Sigma$ such that $D_{f, \Sigma}\cap \A_{f, \Sigma}$ is not
comeagre in $\A_{f, \Sigma}$. It is therefore natural to ask: for which
stabilizing $f$ and finite sets $\Sigma$, is $D_f\cap \A_{f,\Sigma}$ is
comeagre in $\A_{f, \Sigma}$? 

We will prove Theorem~\ref{thm-main-II} in a series of lemmas.  
We begin by showing several ways to extend partial isomorphisms in $\A_{f,
\Sigma}^{<\omega}$, which we will have to do ad infinitum in the proof of
Theorem~\ref{thm-main-II}.

The first lemma follows immediately from the definitions, and the proof is
omitted.

\begin{lem} \label{lem-4-easy-one-point}
  Let $q \in \aut(n K_\omega)^{< \omega}$ be such that $\overline{q} \in S_n$, and
  let $h = q \cup \{(x, y)\}$. Then $h \in \aut(n K_\omega)^{< \omega}$ if and
  only if there is $a \in \{1, \ldots, n\}$ such that $x \in L_a \setminus
  \dom(q)$ and $y \in  L_{(a) \overline{q}} \setminus \ran(q)$.
\end{lem}


Roughly speaking, in the next lemma, we show how to extend a partial isomorphism
with a set of orbit representative to an automorphism with the same set of
orbit representatives.

\begin{lem} \label{lem-4-no-sigma-no-complete-extensions}
  Let $q \in \aut(n K_\omega)^{< \omega}$ be such that $\overline{q} \in S_n$,
  and let $\Sigma$ be a finite subset of $\dom(q)$ such that $|\Sigma \cap C|
  \leq 1$ for every component $C$ of $q$, with equality holding if $C$ is
  complete. Suppose that for each $i \in \{1, \ldots, n\}$ there is $j \in \{1,
  \ldots, n\}$ such that $(j)\overline{q}^m = i$, for some $m\in \Z$, and $L_j
  \cap \Sigma$ contains a point in an incomplete component of $q$. Then there
  is an extension $g \in \aut(n K_\omega)$ of $q$ such that $\Sigma$ is a set
  of orbit representatives of $g$, every incomplete component of $q$ is
  contained in an infinite orbit of $g$, and $(x)g \notin \dom(q)$ for all $x
  \in \ran(q) \setminus \dom(q)$.
\end{lem}

\begin{proof}
  For each $x \in\ran(q) \setminus \dom(q)$ there is $a \in \{1, \ldots, n\}$
  such that $x \in L_a$, and there is $y \in L_{(a)\overline{q}} \setminus
  \left( \dom(q) \cup \ran(q) \right)$. Then by
  Lemma~\ref{lem-4-easy-one-point} the mapping $q' = q \cup \{(x, y)\}$ is in
  $\aut(n K_\omega)^{< \omega}$ and $(x)q' = y \notin \dom(q)$. Repeating this
  for each vertex in $\ran(q) \setminus \dom(q)$ we obtain an extension $q''
  \in \aut(n K_\omega)^{< \omega}$ of $q$ such that $(x)q'' \notin \dom(q)$ for
  all $x \in \ran(q) \setminus \dom(q)$. Hence $(x)g = (x)q'' \notin \dom(q)$
  for every extension $g \in \aut(nK_\omega)$ of $q''$ and every $x \in \ran(q)
  \setminus \dom(q)$.

  Suppose that $O$ is an incomplete component of $q''$ such that $O \cap \Sigma
  = \varnothing$. Let $y \in O \cap \dom(q'') \setminus \ran(q'')$. Then there
  is $a \in \{1, \ldots, n\}$ such that $y \in L_a$. It follows from the
  hypothesis that there is $b \in \{1, \ldots, n\}$, $y_0 \in L_b \cap\ran(q'')
  \setminus \dom(q'')$ \changed{such that the component of $q''$ containing
  $y_0$ intersects $\Sigma$ non-trivially}, and $m \in \N$ such that $(b)
  \overline{q}^m = a$. Successively for each $i \in \{1, \ldots, m - 1 \}$
  choose
  \[
    y_i \in L_{(b)\overline{q}^i} \setminus \left( \dom(q'') \cup \ran(q'')
    \cup \{y_1, \ldots, y_{i - 1}\}\right),
  \]
  and let $y_m = y$. Then by repeated application of
  Lemma~\ref{lem-4-easy-one-point} we have that $q'' \cup \{(y_{i - 1}, y_i) :
  1 \leq i \leq m\} \in \aut(n K_\omega)^{< \omega}$. If we repeat this for
  every incomplete component of $q''$ which has empty intersection with
  $\Sigma$, we
  obtain $q_0 \in \aut(n K_\omega)^{< \omega}$ an extension of $q''$ such that
  every component of $q''$ intersect $\Sigma$ in exactly one point.

  Let $n K_\omega = \{x_i : i \in \N \}$, and suppose that for some $j \in \N$
  we have defined $q_j \in \aut(n K_\omega)^{< \omega}$ such that incomplete
  components of $q$ are contained in incomplete components of $q_j$, $\Sigma$
  consists of exactly one point from every component of $q_j$, and
  \[
    \{x_1, \ldots, x_j\} \subseteq \dom(q_j) \cap \ran(q_j).
  \]
  Suppose $x_{j + 1} \notin \dom(q_j) \cap \ran(q_j)$. There are three cases to
  consider.

  Suppose that $x_{j + 1} \in \ran(q_j) \setminus \dom(q_j)$. Then by
  Lemma~\ref{lem-4-easy-one-point} there is a one-point extension $q_{j + 1} =
  q_j \cup \{(x_{j + 1}, y) \} \in \aut(n K_\omega)^{< \omega}$ for some $y
  \notin \dom(q_j) \cup \ran(q_j)$.
  Suppose that $x_{j + 1} \in \dom(q_j) \setminus \ran(q_j)$. Then by
  Lemma~\ref{lem-4-easy-one-point} there is a one-point extension $q_{j + 1}^{-1} =
  q_j^{-1} \cup \{(x_{j + 1}, y) \} \in \aut(n K_\omega)^{< \omega}$ for some $y
  \notin \dom(q_j) \cup \ran(q_j)$.

  Finally, suppose that $x_{j + 1} \in L_a \setminus \left( \dom(q_j) \cup
  \ran(q_j) \right)$ for some $a$. It follows from the hypothesis that there is
  $b \in \{1, \ldots, n\}$, $y_0 \in L_b \cap\ran(q_j) \setminus \dom(q_j)$
  \changed{such that the component of $q_j$ containing $y_0$ intersects
  $\Sigma$ non-trivially}, and $m \in \N$ such that $(b) \overline{q}^m = a$.
  Successively for each $i \in \{1, \ldots, m - 1 \}$ choose
  \[
    y_i \in L_{(b)\overline{q}^i} \setminus \left( \dom(q_j) \cup \ran(q_j)
    \cup \{y_1, \ldots, y_{i - 1}\}\right).
  \]
  Also let $y_m = x_{j + 1}$. Then by repeated application of
  Lemma~\ref{lem-4-easy-one-point} we have that $q_j \cup \{(y_{i - 1}, y_i) :
  1 \leq i \leq m\} \in \aut(n K_\omega)^{< \omega}$. Now, we fall into the first
  case and we can define $q_{j + 1}$ as before.

  In all three cases, we have defined an extension $q_{j + 1}$ satisfying the
  inductive hypothesis. Let
  \[
    g = \bigcup_{j \in \N} q_j.
  \]
  Then $g \in \aut(n K_\omega)$ and since the orbits of $g$ are in one to one
  correspondence with incomplete components of $q_0$, it follows that $\Sigma$
  is a set of orbit representatives.
\end{proof}

\begin{cor} \label{cor-4-no-complete-extesion}
  Let $q \in \A_{f, \Sigma}^{< \omega}$ be such that $\Sigma \subseteq \dom(q)$
  and $\overline{q} \in S_n$. Then there is an extension $g \in \A_{f, \Sigma}$
  of $q$ such that every incomplete component of $q$ is contained in an
  infinite orbit of $g$, and $(x)g \notin \dom(q)$ for all $x \in \ran(q)
  \setminus \dom(q)$.
\end{cor}

\begin{proof}
  Since $q \in \A_{f, \Sigma}^{< \omega}$, the set $\Sigma$ intersects every
  incomplete component of $q$ in at most one point, and every complete
  component in exactly one point. 
  
  If $i\in \{1,\ldots, n\}$ is arbitrary, then, since every extension $h$ of
  $q$ in $\A_{f, \Sigma}$ has $|\Sigma|$ orbits, it follows there is at least one
   infinite orbit of $h$ with points in $L_i$. Since $\Sigma$ is a set of orbit
   representatives, there exists $x\in \Sigma \cap L_j$ for
   some $j\in \{1,\ldots, n\}$ such that $(j)\overline{q}^m = i$ for some $m\in
   \Z$. In particular, $x$ is on an incomplete component of $q$, and so $q$
   satisfies the hypothesis of
   Lemma~\ref{lem-4-no-sigma-no-complete-extensions} from which the corollary
   follows.
\end{proof}

In the next lemma, as a further consequence of
Lemma~\ref{lem-4-no-sigma-no-complete-extensions}, we show that the direct
implication of the first part of Theorem~\ref{thm-main-II}, is a consequence of
the second part. 

\begin{lem}\label{lem-4-all-sigma}
  Let $f \in \aut(nK_\omega)$ be such that $D_f \cap \A_{f, \Sigma}$ is
  comeagre in $\A_{f, \Sigma}$ for all finite sets $\Sigma \subseteq
  nK_\omega$. Then $D_f$ is comeagre in $A_f$.
\end{lem}

\begin{proof}
  Let $q \in \A_f^{< \omega}$. If necessary by extending $q$, we can assume
  that $\overline{q} \in S_n$. Then all extensions $h \in \aut(n K_\omega)^{<
  \omega}$ of $q$ are also in $\A_f^{< \omega}$. For all $i \in \{1, \ldots,
  n\}$, let $x_i \in L_i \setminus \left( \dom(q) \cup \ran(q) \right)$. Then
  by applying Lemma~\ref{lem-4-easy-one-point} repeatedly we can construct $h
  \in \A_f^{< \omega}$ an extension of $q$ such that each vertex $x_i$ is on a
  incomplete component of $h$. Fix any $\Sigma \subseteq n K_\omega$ such that
  $\Sigma$ intersects every component of $h$ exactly once. Since $\overline{h} \in
  S_n$, for each $i \in \{1, \ldots, n\}$ there is an incomplete component containing
  $x_i$, and by the choice of $\Sigma$ there is $j \in \{1, \ldots, n\}$ such
  that $\Sigma \cap L_j$ is non-empty and $(j)\overline{h}^m = i$ for some $m
  \in \Z$. Then by Lemma~\ref{lem-4-no-sigma-no-complete-extensions} there is
  $g$ an extension of $q$ with finitely many orbits. Therefore we are done by
  Lemma~\ref{lem-no-sigma}.
\end{proof}

\begin{lem} \label{lem-4-one-point-ext}
  Let $q \in \A_{f, \Sigma}^{< \omega}$ be such that $\Sigma \subseteq \dom(q)$
  and $\overline{q} \in S_n$. Suppose $h = q \cup \{(x, y)\}\in \aut(n
  K_\omega)^{< \omega}$ for some $x \notin \dom(q)$ and
  $y \notin \dom(q) \cup \ran(q)$ such that $x \neq y$. Then $h \in \A_{f,
  \Sigma}^{< \omega}$.
\end{lem}

\begin{proof}
  Since $q \in \aut(n K_\omega)^{< \omega}$ there is $r \in \aut(n K_\omega)^{<
  \omega}$ extending $q$, such that $x \in \ran(r) \setminus \dom(r)$. By
  Corollary~\ref{cor-4-no-complete-extesion} there is $g \in \A_{f, \Sigma}$
  such that every incomplete components of $r$ is contained in an infinite
  orbit of $g$ and $(x)g \notin \dom(r)$, and so $(x)g \notin \dom(q)$.
  \changed{If $g$ extends $h$ then we are done; so we assume that $(x)g \neq
  y$.} Note that if $(x)g = x$, then $\{x\}$ is an orbit of $g$ and therefore
  $x \in \Sigma$. However, $\Sigma \subseteq \dom(q)$, which contradicts the
  assumption that $x \notin \dom(q)$.  Hence $(x)g \neq x$.

  Since $x \notin \dom(q)$ and $g$ is an extension of $q$, it follows that
  $(x)g \notin \ran(q)$. Then $(x)g, y \notin \dom(q)\cup \ran(q)$ and since $h
  \in \aut(n K_\omega)^{< \omega}$ and $(x)g \in \aut(n K_\omega)$ it follows
  that $(x)g$ and $y$ are in the same connected component of $n K_\omega$. Then
  the transposition $\left((x)g \; y\right)$ swapping $(x)g$ and $y$ is in
  $\aut(n K_\omega)$ and so
  \[
    g' = \left( (x)g \; y\right)  g  \left( (x)g \; y\right).
  \]
  It follows from $(x)g \neq x$, $(x)g \neq y$, and $(x)g, y \notin \dom(q)
  \cup \ran(q)$ that $g'$ is an extension of $h$. Therefore $h \in \A_{f,
  \Sigma}^{< \omega}$.
\end{proof}

\begin{lem} \label{lem-4-amalgamation}
  Let $q \in \A_{f, \Sigma}^{< \omega}$ be such that $\Sigma \subseteq \dom(q)$
  and let $A, B$ be distinct incomplete components of $q$ such that at most
  one of $A$ and $B$ intersects $\Sigma$ non-trivially. Suppose that
  \[
    \overline{q|_{\dom(q) \setminus A}} =  \overline{q|_{\dom(q) \setminus B}}
    \in S_n
  \]
  and let $h = q \cup \{(x, y)\} \in \aut(n K_\omega)^{< \omega}$, for some $x
  \in A$ and $y \in B$. Then $h \in \A_{f, \Sigma}^{< \omega}$.
\end{lem}

\begin{proof}
  Since $h \in \aut(n K_\omega)^{< \omega}$, it follows that $x \notin \dom(q)$
  and $y \notin \ran(q)$.

  Assume without loss of generality that $B \cap \Sigma = \varnothing$ and $B =
  \{y_1, \ldots, y_m\}$ for some $m \in \N$ such that $y_1 = y$ and $(y_i)q =
  y_{i + 1}$ for all $i \in \{1, \ldots, m - 1\}$. 
  The proof of the case when $B \cap \Sigma \neq \varnothing$ can be obtained
  by apply the argument below to $q^{-1}$ and $h^{-1}$.
  
  We will define $k \in \{1, \ldots, m\}$ there is $h_k \in \A_{f,
  \Sigma}^{< \omega}$  extending $h_{k-1}$ such that $\Sigma \subseteq
  \dom(h_k)$, $\overline{h_k} \in S_n$, and 
  \[
    (x)h_k^i = y_i \text{ for } 1 \leq i \leq k,
    \qquad
    y_k \notin \dom(h_k),
    \qquad \text{and} \qquad
    y_i \notin \dom(h_k) \cup \ran(h_k) \text{ for } k < i.
  \]
  If $k = 1$, then we define $h_1 = h|_{\dom(h) \setminus B}$. \changed{By
  Lemma~\ref{lem-4-one-point-ext}, it follows that $h_1 = q|_{\dom(q) \setminus
  B} \cup \{(x, y)\} \in \A_{f, \Sigma}^{< \omega}$, and so} $h_1$ satisfies the
  required conditions.

  Suppose $k > 1$. Then by Lemma~\ref{lem-4-one-point-ext} we have that $h_{k +
  1} = h_k \cup \{(y_k, y_{k + 1})\} \in \A_{f, \Sigma}^{< \omega}$. Since
  $\dom(h_{k + 1}) = \dom(h_k) \cup \{y_k\}$ and $\ran(h_{k + 1}) = \ran(h_k)
  \cup \{y_{k + 1}\}$, it follows that $h_{k + 1}$ satisfies the required
  conditions. 
  
  Therefore after repeating this process $m$ times,  we obtain $h_m \in \A_{f,
  \Sigma}^{< \omega}$ which extends $h_1$. It follows from the definition of
  $h_m$ that $h_m = h$.
\end{proof}

Now we can characterize when the set $\A_{f, \Sigma}$ is non-empty.

\begin{lem}\label{lem-A-non-empty}
  Let $f\in \aut(nK_{\omega})$ and let $\Sigma$ be a finite subset of
  $nK_\omega$. Then $\A_{f, \Sigma}$ is non-empty if and only if there exists
  \changed{$\sigma \in S_n$ such that $\genset{\overline{f}, \sigma}= S_n$ and
  for all $i \in \{1,\ldots, n\}$}
  \[
    \left(\bigcup_{j\in \Z} L_{(i)\sigma ^ j} \right) \cap \Sigma \neq \varnothing.
  \]
\end{lem}
\begin{proof}
  $(\Rightarrow)$ Suppose that $g \in \A_{f, \Sigma}$. Since
  $g \in \A_{f, \Sigma} \subseteq \A_f$, it follows from the definition of
  $\A_f$ that $\langle \overline{f}, \overline{g} \rangle = S_n$. 
  \changed{Let} $i \in \{1, \ldots, n\}$. Then there is $x \in \Sigma$ and $m
  \in \N$ such that $(x)g^m \in L_i$, since $\Sigma$ is a set of orbit
  representatives. \changed{Hence}
  \[
    x \in L_{(i)\overline{g}^{-m}} \subseteq\bigcup_{j \in \Z}
    L_{(i)\overline{g}^j}.
  \]

  $(\Leftarrow)$ 
  It is routine to show
  that there is $q \in \aut(n K_\omega)^{< \omega}$ such that $\Sigma \subseteq
  \dom(q)$, $\overline{q} = \sigma$ and $q$ has precisely $|\Sigma|$ many
  components, all of which are incomplete, and $\Sigma$ intersects them in
  precisely one point.
  Since all components of $q$ are incomplete, it satisfies the hypothesis of
  Lemma~\ref{lem-4-no-sigma-no-complete-extensions} and hence there is $g \in
  \A_{f, \Sigma}$ an extension of $q$.
\end{proof}

\changed{By Proposition~\ref{prop-piccard}, in the case that $n \geq 3$,  there
exists $\sigma \in S_n$ such that $\genset{\overline{f}, \sigma}= S_n$ if and
only if $n \not = 4$ and $\overline{f}\not = \id$, or $n= 4$ and $\overline{f}
\notin \{\id, (1 \; 2)(3 \; 4), (1 \; 3)(2 \; 4), (1 \; 4)(2 \; 3) \}$.}

In the next lemma, we give a decompisition of $D_f\cap \A_{f, \Sigma}$ as an
intersection of set that we will later prove to be open and dense, under the
hypothesis of Theorem~\ref{thm-main-II}. 

\begin{lem}\label{lem-4-basis}
  Let $\mathcal{P} \subseteq \aut(n K_\omega)^{< \omega}$ be such that $p \in
  \mathcal{P}$ if and only if $\dom(p)$ and $\ran(p)$ are disjoint, and
  $\overline{p} = \id$. Then
  \[
    D_f \cap \A_{f, \Sigma} = \bigcap_{p \in \mathcal{P}} \{g \in \A_{f,
    \Sigma} : \langle f, g \rangle \cap [p] \neq \varnothing \}.
  \]
\end{lem}

\begin{proof}
  Recall that
  \[
    D_f \cap \A_{f, \Sigma} = \{g \in \A_{f, \Sigma}: \langle f, g\rangle
    \text{ is dense in } \aut(nK_\omega) \}
    = \bigcap_{q \in \aut(nK_\omega)^{<\omega}} \{g \in \A_{f, \Sigma} :
    \langle f, g \rangle \cap [q] \neq \varnothing \}.
  \]

  $(\subseteq)$ This follows immediately since $\mathcal{P} \subseteq
  \aut(nK_\omega)^{<\omega}$.

  $(\supseteq)$ Let $g \in \bigcap_{p \in \mathcal{P}} \{g \in \A_{f, \Sigma}:
  \langle f, g \rangle \cap [p] \neq \varnothing\}$, and let $q \in \aut(n
  K_\omega)^{< \omega}$ be arbitrary. Since $g \in \A_{f, \Sigma}$ there is $h
  \in \langle f, g \rangle$ such that $\overline{h} = \overline{q}^{- 1}$.

  Let $p \in \aut(n K_\omega)^{< \omega}$ be such that $\overline{p} = \id$,
  $\dom(p) = \dom(hq)$ and $\ran(p) \cap (\dom(hq)\cup \ran(hq)) =\varnothing$.
  Then $\dom(p^{-1}h\changed{q}) = \ran(p)$ and $\ran(p^{-1}h\changed{q}) = \ran(h\changed{q})$, so $p,
  p^{-1}hq \in \mathcal{P}$.  Hence there are $h_1, h_2 \in \langle f, g
  \rangle$ such that $h_1 \in [p]$ and $h_2 \in [p^{-1}hq]$. Therefore
  $h^{-1}h_1h_2 \in [q]$, so
  \[
    g \in \bigcap_{q \in \aut(n K_\omega)^{< \omega}} \{g \in \A_{f, \Sigma}:
    \langle f, g \rangle \cap [q] \neq \varnothing\},
  \]
  as required.
\end{proof}

Let $w$ be a freely reduced word over the alphabet $\{\alpha, \beta\}$, i.e. $w
= \alpha^{n_1} \beta^{n_2} \cdots \beta^{n_{2N}}$ for some $N \in \N$ and $n_1,
\ldots, n_{2N} \in \Z$ with $n_i \neq 0$ for all $i \in \{2, \ldots, 2N - 1\}$.
Also let $f \in \aut(n K_\omega)$ be fixed and suppose that $p \in \aut(n
K_\omega)^{< \omega}$. Then define
\[
  w(p) = p^{n_1} f^{n_2} p^{n_3} \cdots p^{n_{2N - 1}} f^{n_{2N}}
\]
where the product on the right hand side is the usual product of partial
permutations.  Note that $\aut(nK_\omega) \cup \aut(n K_\omega)^{< \omega}$
forms a subsemigroup of the semigroup of all isomorphisms between finite
induced subgraphs of $n K_\omega$.  Hence, if we denote by $F_{\alpha, \beta}$
the free group on the alphabet $ \{\alpha, \beta\}$, then $w(p)$ is simply the
image of $w$ under the semigroup homomorphism $\phi : F_{\alpha, \beta} \to
\aut(nK_\omega) \cup \aut(n K_\omega)^{< \omega}$ such that $(\alpha)\phi = p$
and $(\beta)\phi = f$.

\begin{lem}\label{lem-4-main}
  Let $n\in \N$ be such that $n > 1$ and let $f \in \aut(n K_\omega)$ be
  non-stabilizing. If $n = 2$ and $\overline{f} = \id$, then further suppose
  that $\fix(f)$ is finite. Let $\Gamma, \Delta \subseteq nK_\omega$ be finite
  and disjoint, and let $q \in \A_{f, \Sigma}^{< \omega}$ be such that
  $\overline{q} \in S_n$ and $\ran(q) \cap \Delta = \varnothing$. Then there is
  an extension $h \in \A_{f, \Sigma}^{< \omega}$ \changed{of $q$} and $w \in
  F_{ \alpha, \beta}$ such that
  \[
    \overline{w(h)} = \id, \qquad
    \ran(h) \cap \Delta = \varnothing, \qquad
    \Gamma\subseteq \dom\left( w(h) \right),\qquad \text{and} \qquad
    \left( \Gamma \right) w(h) \cap \dom(h) = \varnothing.
  \]
  Moreover, $\left( \Gamma \right) \changed{w}(h)  h^m \cap \dom(q) =
  \varnothing$ for all $m \in \Z$, i.e. no vertex in $ \left( \Gamma \right)
  \changed{w}(h)$ is on an incomplete component of $h$, which extends an
  incomplete component of $q$.
\end{lem}

The proof of Lemma~\ref{lem-4-main} is rather involved, so before giving its
proof we will demonstrate how the lemma can be used to prove
Theorem~\ref{thm-main-II}.

We will first prove an easy special case of Theorem~\ref{thm-main-II}.

\begin{lem}\label{lem-4-n=2}
  Let $f \in \aut(2 K_\omega)$ be non-stabilising such that $\overline{f} =
  \id$ and $\fix(f)$ is infinite, and let $\Sigma \subseteq 2K_\omega$ be
  finite. Then $D_f \cap \A_{f, \Sigma}$ is comeagre in $\A_{f, \Sigma}$.
\end{lem}

\begin{proof}
  By Lemmas \ref{lem-E-open} and \ref{lem-4-basis} we only need to show that
  $\{g \in \A_{f, \Sigma} : \langle f, g \rangle \cap [p] \not= \varnothing \}$
  is dense in $\A_{f, \Sigma}$ for all $p \in \mathcal{P}$. Let $q \in \A_{f,
  \Sigma}^{<\omega}$ and suppose, without loss of generality, that $\dom(p)
  \cup \ran(p) \cup \Sigma \subseteq \dom(q)$ and $\overline{q} \in S_2$.
  Since $\overline{f} = \id$, it follows that $\overline{q} = (1 \; 2)$.

  Let $L_1$ and $L_2$ be the connected components of $2 K_\omega$. If necessary
  by relabeling the connected components we may assume that $L_2 \cap \fix(f)$
  is infinite. It follows from Proposition~\ref{prop-f-stabilizing} that if
  $f$ has a finite cycle contained in $L_1$, then $f$ is stabilising. Hence
  all of the cycles of $f$ contained in $L_1$ are infinite.

  Let $m_1 \in \Z$ be such that $\left( L_1 \cap \dom(p) \right) f^{m_1}$ is
  disjoint from $\dom(q) \cup \ran(q)$. By Lemmas~\ref{lem-4-easy-one-point}
  and~\ref{lem-4-one-point-ext} there is $q_1 \in \A_{f, \Sigma}^{< \omega}$ an
  extension of $q$ such that $\left( \dom(p)\right)f^{m_1} \subseteq \dom(q_1)$
  and $\left(L_1 \cap \dom(p) \right)f^{m_1} q_1 \subseteq \fix(f) \setminus
  \dom(q_1)$, \changed{which is possible since $L_2 \cap \fix(f)$ is infinite
  and $(L_1)f^{m_1}q_1 \subseteq L_2$}. The extension $q_1$ can be chosen so
  that components of $q_1$ containing any vertices from $\left(L_1 \cap \dom(p)
  \right)f^{m_1}$ do not extend any of the components of $q$. Since $\left(L_2
  \cap \dom(p) \right) f^{m_1} q_1 \subseteq L_1$, there is $m_2 \in \Z$ such
  that $\left(L_2 \cap \dom(p)\right)f^{m_1} q_1 f^{m_2}$ is disjoint from
  $\dom(q_1) \cup \ran(q_1)$. Hence $\left(\dom(p)\right)f^{m_1}q_1f^{m_2} \cap
  \dom(q_1) = \varnothing$.

  Let $m_3 \in \Z$ be such that $\left( L_1 \cap \ran(p) \right) f^{m_3}$ is
  disjoint from $\dom(q_1) \cup \ran(q_1)$. By
  Lemmas~\ref{lem-4-easy-one-point} and~\ref{lem-4-one-point-ext} there is $q_2
  \in \A_{f, \Sigma}^{< \omega}$ \changed{an extension of $q_1$} such that
  $\left( \ran(p) \right)f^{m_3} \subseteq \ran(q_2)$, $\left(L_1 \cap \ran(p)
  \right)f^{m_3} q_2^{-1} \subseteq \fix(f) \setminus \ran(q_2)$, $\left(
  \dom(p) \right)f^{m_1}q_2f^{m_2}$ is disjoint from $\dom(q_2)$. The extension
  $q_2$ can be chosen so that components of $q_2$ containing any vertices from
  $\left(L_1 \cap \ran(p) \right)f^{m_3}$ do not extend any of the components
  of $q_1$, and \changed{also
  that every vertex of $(L_1 \cap \ran(p)) f^{m_3}$ is on a different
  incomplete components of $q_2$.} Then there is $m_4 \in \Z$ such that
  $\left(L_2 \cap \ran(p)\right)f^{m_3} q_2^{-1} f^{m_4}$ is disjoint from
  $\dom(q_2) \cup \ran(q_2) \cup (\dom(p))f^{m_1} q_2 f^{m_2}$. Hence
  \[
    \left(\dom(p)\right)f^{m_1}q_2f^{m_2} \cap \dom(q_2) = \varnothing
    \quad \text{and} \quad
    \left(\ran(p)\right)f^{m_3}q_2^{-1}f^{m_4} \cap \ran(q_2) = \varnothing.
  \]

  Let $\dom(p) = \{x_1, \ldots, x_k\}$. Then \changed{for all $i \in \{1,
  \ldots, k\}$ there are
  \[
    y_i \in 2 K_\omega
  \setminus \left( \dom(q_2) \cup \ran(q_2) \cup (\dom(p))f^{m_1} q_2 f^{m_2}
  \cup (\ran(p))f^{m_3} q_2^{-1} f^{m_4}
  \right)
  \] 
  such that $h' = q_2 \cup \{ ( (x_i)f^{m_1}q_2f^{m_2}, y_i) : 1 \leq i \leq
  k\} \in \A_{f, \Sigma}^{< \omega}$ by Lemmas~\ref{lem-4-easy-one-point}
  and~\ref{lem-4-one-point-ext}. Let $A$ be the incomplete component of $h'$
  containing $(x_1)f^{m_1}q_2f_{m_2}$ and let $B$ be the incomplete component
  of $h'$ containing $(x_1)p f^{m_3} q_2^{-1}f^{m_4}$. Then $y_1 \in A$, and so
  $|A| \geq 2$. If $|B| = 1$, then $h' \cup \{(y_1, (x_1)p f^{m_3} q_2^{-1}
  f^{m_4})\} \in \A_{f, \Sigma}^{< \omega}$ by
  Lemmas~\ref{lem-4-easy-one-point} and~\ref{lem-4-one-point-ext}, as
  $(x_1)f^{m_1}q_2f^{m_2}$ and $(x_1)pf^{m_3}q_2^{-1}f^{m_4}$ are in the same
  connected component of $2K_\omega$. 
  If $(x_1)p \in L_2$, then by the choice of $m_4$, $(x_1)p
  f^{m_3}q_2^{-1}f^{m_4} \notin \dom(h') \cup \ran(h')$, and so $|B| = 1$, and
  we have already considered this case.
  Suppose that $|B| \geq 2$. Then  $(x_1)p \in L_1$
  and by the choice of $q_2$ the incomplete component of
  $h'$ containing $(x_1)pf^{m_3}q_2^{-1}f^{m_4}$, in other words $B$, does not
  extend an incomplete component of $q_1$. Since $A$ is an incomplete component
  of $q_1$ with $y_1$ adjoined, it follows that $B$ intersects $\Sigma$
  trivially, and $A$ and $B$ are distinct.  Hence $\overline{h'|_{\dom(h')
  \setminus A}} = \overline{h'|_{\dom(h') \setminus B}} = (1 \; 2)$, and thus
  $h' \cup \{(y_1, (x_1)p f^{m_3} q_2^{-1} f^{m_4})\} \in \A_{f, \Sigma}^{<
  \omega}$ by Lemma~\ref{lem-4-amalgamation}. Repeating this argument for $i
  \in \{2, \ldots, k\}$, it can be shown that $h = q_2 \cup \{ (
  (x_i)f^{m_1}q_2f^{m_2}, y_i), (y_i, (x_i)p f^{m_3}q_2^{-1}f^{m_4} ) : 1 \leq
  i \leq k \} \in \A_{f, \Sigma}^{< \omega}$}. 
  Hence $f^{m_1}gf^{m_2}g^2 f^{-m_4} g f^{-m_3} \in [p]$ for every $g \in [h]
  \cap \A_{f, \Sigma}$.  Therefore $\{g \in \A_{f, \Sigma} : \langle f, g
  \rangle \cap [p] \not= \varnothing \}$ intesects $[q]$ non-trivially, and
  since $q$ was arbitrary, is dense in $\A_{f, \Sigma}$.
\end{proof}

Next, we give the proof of Theorem~\ref{thm-main-II} modulo the proof of
Lemma~\ref{lem-4-main}, which is given in the next section.

\begin{proof}[Proof of Theorem~\ref{thm-main-II}]
  If $\A_f = \varnothing$, then $f$ is non-stabilizing and $D_f$ is comeager in
  $\A_f$. Hence we may assume that $\A_f \neq \varnothing$.

  Suppose that $f$ is stabilizing.  By the
  definition, there is $\Gamma \subsetneq n K_\omega$, $x \in
  \Gamma$ and $q \in \A_f^{< \omega}$ such that for all $g \in [q] \cap \A_f$
  and all $h \in \langle f, g \rangle$ we have that $(x)h \in \Gamma$. Let $y
  \notin \Gamma$. \changed{Then} $p = \{(x, y)\} \in \aut(n K_\omega)^{< \omega}$.
  Then $\langle f, g \rangle \cap [p] = \varnothing$ and thus $g \notin D_f$
  implying that $D_f$ is not dense in $\A_f$. Hence $D_f \cap \A_f$ is not
  comeagre in $\A_f$.
  
  If $f$ is non-stabilising and $\Sigma$ is a finite subset of $nK_\omega$,
  then it suffices, by Lemma~\ref{lem-4-all-sigma}, to show that 
  $D_f\cap \A_{f, \Sigma}$ is comeagre in $\A_{f, \Sigma}$.
  If $\A_{f, \Sigma} = \varnothing$, the result is trivial. Hence we may assume that
  $\A_{f, \Sigma} \neq \varnothing$. If $n = 2$, $\overline{f} = \id$, and
  $\fix(f)$ is infinite we are done by Lemma~\ref{lem-4-n=2}. Hence we may,
  additionally assume that $n \geq 2$, and that if $n = 2$ and $\overline{f} =
  \id$, then $\fix(f)$ is finite.

  By Lemmas~\ref{lem-E-open} and~\ref{lem-4-basis} we only need to show that
  $\{g \in \A_{f, \Sigma} : \langle f, g \rangle \cap [p] \not= \varnothing \}$
  is dense in $\A_{f, \Sigma}$ for all $p \in \mathcal{P}$. Let $q \in \A_{f,
  \Sigma}^{<\omega}$ and suppose, without loss of generality, that $\dom(p)
  \cup \ran(p) \cup \Sigma \subseteq \dom(q)$ and $\overline{q} \in S_n$.

  Apply Lemma~\ref{lem-4-main} with $\Delta = \varnothing$ and $\Gamma =
  \dom(p)$. Then there is an extension $q_1' \in \A_{f, \Sigma}^{< \omega}$ of
  $q$ and $\omega_1 \in F_{\alpha, \beta}$ such that
  \[
    \overline{\omega_1(q_1')} = \id, \qquad
    \dom(p) \subseteq \dom\left( \omega_1(q_1')\right), \qquad \text{and} \qquad
    \left( \dom(p) \right) \omega_1(q_1') \cap  \dom(q_1') = \varnothing.
  \]

  Suppose $\left( \dom(p) \right)\omega_1(q_1') \setminus \ran(q_1')$ is
  non-empty. Let $y \in \left( \dom(p) \right)\omega_1(q_1') \setminus
  \ran(q_1')$ and let $a \in \{1, \ldots, n\}$ be such that $y \in L_a$. Then
  there is
  \[
    x \in L_{(a)\overline{q_1'}^{-1}} \setminus \left( \dom(q_1') \cup
    \ran(q_1') \cup \left( \dom(p) \right) \omega_1(q_1') \right).
  \]
  It follows from Lemma~\ref{lem-4-easy-one-point} that $q_1'' = q_1' \cup \{
  (x, y) \} \in \aut(nK_\omega)^{< \omega}$ and thus in $\A_{f, \Sigma}^{<
  \omega}$ by Lemma~\ref{lem-4-one-point-ext}. Then
  \[
    \overline{\omega_1(q_1'')} = \id, \qquad
    \dom(p) \subseteq \dom\left( \omega_1(q_1'')\right), \qquad \text{and} \qquad
    \left( \dom(p) \right) \omega_1(q_1'') \cap  \dom(q_1'') = \varnothing.
  \]
  Moreover, $\left| \left( \dom(p) \right)\omega_1(q_1') \setminus \ran(q_1')
  \right| > \left| \left( \dom(p) \right)\omega_1(q_1'') \setminus
  \ran(q_1'')\right|$, and if we do this extension for every vertex in $\left(
  \dom(p) \right)\omega_1(q_1') \setminus \ran(q_1')$, we can define an
  extension $q_1 \in \A_{f, \Sigma}^{< \omega}$ of $q_1'$ such that
  \begin{equation}\label{equation-7}
    \overline{\omega_1(q_1)} = \id, \qquad
    \dom(p) \subseteq \dom\left( \omega_1(q_1)\right), \qquad \text{and} \qquad
    \left( \dom(p) \right) \omega_1(q_1) \subseteq \ran(q_1) \setminus  \dom(q_1).
  \end{equation}
  Hence every vertex in $\left( \dom(p) \right) \omega_1(q_1)$ is on a
  incomplete component of $q_1$.

  If $\Delta = \left( \dom(p) \right) \omega_1(q_1)$ and $\Gamma = \ran(p)$,
  then $\ran(q_1^{-1}) = \dom(q_1)$ and  $\Delta$ are disjoint. Hence by
  Lemma~\ref{lem-4-main}, there is an extension $q_2^{-1} \in \A_{f, \Sigma}^{<
  \omega}$ of $q_1^{-1}$ and $\omega_2' \in F_{\alpha, \beta}$ such that
  $\overline{\omega_2'(q_2^{-1})} = \id$,
  \begin{gather*}
    \ran(q_2^{-1}) \cap \left( \dom(p) \right) \omega_1(q_1) = \varnothing, \\
    \ran(p) \subseteq \dom\left( \omega_2'(q_2^{-1})\right),\\
    \left(\ran(p) \right) \omega_2'(q_2^{-1})\cap \dom(q_2^{-1}) = \varnothing,
  \end{gather*}
  and no vertex in $\left( \ran(p) \right) \omega_2'(q_2^{-1})$ is on a
  incomplete component of $q_2^{-1}$ extending an incomplete component of
  $q_1^{-1}$.

  Since $\dom(p) \subseteq \dom\left( \omega_1(q_1) \right)$ by
  \eqref{equation-7}, and $q_2$ is an extension of $q_1$, it follows that
  $\left(\dom(p)\right) \omega_1(q_1) = \left(\dom(p)\right) \omega_1(q_2)$.
  Let $\omega_2 \in F_{\alpha, \beta}$ be such that $\omega_2(q_2) =
  \omega_2'(q_2^{-1})$, i.e.\ replace every occurrence of $\alpha$ in $\omega_2'$
  by $\alpha^{-1}$ and vica versa. Then $\overline{\omega_2(q_2)} = \id$,
  \begin{gather*}
    \dom(q_2) \cap \left( \dom(p) \right) \omega_1(q_2) = \varnothing, \\
    \ran(p) \subseteq \dom\left( \omega_2(q_2)\right),\\ \left(\ran(p) \right)
    \omega_2(q_2)\cap \ran(q_2) = \varnothing,
  \end{gather*}
  and no vertex in $\left( \ran(p) \right) \omega_2(q_2)$ is on a incomplete
  components of $q_2$ extending an incomplete component of $q_1$.

  \changed{Let $\{i(j, k) : k \in \{1, \ldots, m_j\}\}$ where 
  $j \in \{1, \ldots, l\} $ be the orbits of $\overline{q_2}$ and suppose
  that $(i(j, k)) \overline{q_2} = i(j, k + 1)$ for all $j \in \{1,
  \ldots, l\}$ and all $k
  \in \{1, \ldots, m_j - 1\}$. For each $i \in \{1, \ldots, n\}$} choose
  \[
    x_i \in L_i \setminus \left( \dom(q_2) \cap \ran(q_2) \cup
    \left(\dom(p)\right) \omega_1(q_2) \cup \left( \ran(p) \right)
    \omega_2(q_2) \right),
  \]
  \changed{and also for all $j \in \{1, \ldots, l\}$ choose
  \[
  x_{i(j, m_j + 1)} \in L_{i(j, 1)} \setminus \left( \{x_{i(j, 1)}\} \cup
    \dom(q_2) \cap \ran(q_2) \cup \left(\dom(p)\right) \omega_1(q_2) \cup
    \left( \ran(p) \right) \omega_2(q_2) \right),
  \]
  Then $h_0 = q_2 \cup \{ (x_{i(j, k)}, x_{i(j, k + 1)}) : j \in \{1, \ldots,
  l\} \ \text{and} \ k \in \{1, \ldots, m_j\} \} \in \aut(n K_\omega)^{<
  \omega}$} by Lemma~\ref{lem-4-easy-one-point} and also $h_0 \in \A_{f,
  \Sigma}^{< \omega}$ by Lemma~\ref{lem-4-one-point-ext}. \changed{Let $P$ be
  an arbitrary incomplete component of $h_0$. Since $x_{i(j, k)} \notin
  \dom(q_2) \cup \ran(q_2)$ for all $j$ and all $k$, it follows that $P$ is
  either a subset of $K = \{x_{i(j, k)} : j \in \{1, \ldots, l\} \ \text{and}
  \ k \in \{1, \ldots, m_j + 1\}\}$ or disjoint from $K$. If $P \subseteq K$,
  then $q_2 \subseteq h_0|_{\dom(h_0) \setminus P}$, and so
  $\overline{h_0|_{\dom(h_0) \setminus P}} = \overline{q_2} \in S_n$. Otherwise
  $P \cap K = \varnothing$, and so $\{x_i : i \in \{1, \ldots, n\} \} \subseteq
  \dom(h_0) \setminus P$. Hence, $\{ (x_{i(j, k)}, x_{i(j, k + 1)})) : j \in
  \{1, \ldots, l\} \ \text{and} \ k \in \{1, \ldots, m_j\}\} \subseteq
  h_0|_{\dom(h_0) \setminus P}$, which implies that $\overline{h_0|_{\dom(h_0)
  \setminus P}} = \overline{q_2} \in S_n$. It follows from the choice of
  vertices $x_i$ and $x_{i(j, m_j + 1)}$}, that $\overline{\omega_2(h_0)} = \id$,
  \begin{gather*}
    \dom(h_0) \cap \left( \dom(p) \right) \omega_1(h_0) = \varnothing, \\
    \ran(p) \subseteq \dom\left( \omega_2(h_0)\right),\\
    \left(\ran(p) \right) \omega_2(h_0)\cap \ran(h_0) = \varnothing.
  \end{gather*}

  Let $k$ be the order of $\overline{q} \in S_n$. We will now inductively
  construct an extension $h \in \A_{f, \Sigma}^{< \omega}$ of $h_0$ (and hence
  of $q$) such that $(x) \omega_1(h)  h^k  \omega_2(h)^{-1} = (x)p$
  for all $x \in \dom(p)$. Let $\dom(p) = \{x_1, \ldots, x_d \}$,
  and suppose that for $j \in \{0, \ldots, k - 2\}$ we have defined an
  extension $h_j \in \A_{f, \Sigma}^{< \omega}$ of $h_0$ such that
  \[
    \dom(h_j) \cap \left( \dom(p) \right) \omega_1(h_j)  h_j^j =
    \varnothing, \qquad \left( \ran(p) \right)
    \omega_2(h_j)\cap \ran(h_j) = \varnothing,
  \]
  and $\dom(p)$ and $\ran(p)$ are contained in $\dom(\omega_1(h_j) 
  h_j^j)$ and $\dom(\omega_2(h_j))$ respectively.

  Note that if $j = 0$, the inductive hypothesis is satisfied since $h_0^0$ is
  an identity on $\dom(h_0)$, the $\dom(h_0)$ is disjoint from $\left( \dom(p)
  \right) \omega_1(h_0)$, the set $\ran(h_0)$ is disjoint from $\left( \ran(p)
  \right) \omega_2(h_j)\cap \ran(h_j)$, and $\dom(p)$ and $\ran(p)$ are
  contained in $\dom(\omega_1(h_0))$ and $\dom(\omega_2(h_0))$ respectively.

  Suppose $j >0$. For all $i \in \{1, \ldots, d\}$, let $y_i =
  (x_i)\omega_1(h_j)  h_j^j$ and suppose that $a_i \in \{1, \ldots, n\}$
  such that $y_i \in L_{a_i}$. Then for each successive $i \in \{1, \ldots,
  d\}$ choose
  \[
    z_i \in L_{(a_i)\overline{h_j}} \setminus \left( \dom(h_j) \cup \ran(h_j)
    \cup \{y_1, \ldots, y_d\} \cup \{z_1, \ldots, z_{i - 1}\} \cup
    \left(\ran(p)\right)\omega_2(h_j) \right).
  \]

  We define $h_{j + 1} = h_j \cup \{(y_i, z_i): 1 \leq i \leq d\}$.  Since $z_i
  \in L_{(a_i)\overline{h_j}}$, it follows that $h_{j + 1} \in \aut(n
  K_\omega)^{< \omega}$ by Lemma~\ref{lem-4-easy-one-point} and hence $h_{j +
  1} \in \A_{f, \Sigma}^{< \omega}$ by Lemma~\ref{lem-4-one-point-ext}. Note
  that the choice of $z_i$ implies that none of the incomplete components of $h_j$ are
  amalgamated in $h_{j + 1}$.

  It is easy to see that $\dom(h_{j + 1}) = \dom(h_j) \cup \{y_1, \ldots,
  y_d\}$ and $\ran(h_{j + 1}) = \ran(h_j) \cup \{z_1, \ldots, z_d\}$.  Since
  $(x_i)\omega_1(h_{j + 1})  h_{j + 1}^j = (x_i) \omega_1(h_j) 
  h_j^j$ for all $i \in \{1, \ldots, d\}$
  \begin{align*}
    (x_i)\omega_1(h_{j + 1})  h_{j + 1}^{j + 1} &= (x_i)\omega_1(h_{j +
    1})  h_{j + 1}^j  h_{j + 1} \\
    &= (x_i)\omega_1(h_j)  h_j^j  h_{j + 1} \\
    &= (y_i)h_{j + 1} \\
    &= z_i \notin \dom(h_{j + 1}).
  \end{align*}
  Hence $\dom(h_{j +1}) \cap \left( \dom(p) \right) \omega_1(h_{j + 1}) 
  h_{j + 1}^{j + 1} = \varnothing$ and $\dom(p) \subseteq \omega_1(h_{j + 1})
   h_{j + 1}^{j + 1}$.

  It follows from $\ran(p) \subseteq \omega_2(h_0)$, that $\left( \ran(p)
  \right) \omega_2 (h_{j + 1}) = \left( \ran(p) \right) \omega_2 (h_j)$, and so
  $\left( \ran(p) \right) \omega_2(h_{j + 1}) \cap \ran(h_j) = \varnothing$.
  Since $z_i \notin \left( \ran(p) \right) \omega_2(h_j)$ for all $i \in \{1,
  \ldots, d\}$, it also follows that $\left(\ran(p) \right) \omega_2(h_{j +
  1})\cap \ran(h_{j + 1}) = \varnothing$. Finally,
  $\dom(p)$ and $\ran(p)$ are contained in $\dom(\omega_1(h_{j + 1}) 
  h_{j + 1}^{j + 1})$ and $\dom(\omega_2(h_{j + 1}))$ respectively, and so
  $h_{j + 1}$ satisfies the inductive hypothesis.

  By induction on $j$, we obtain an extension $h_{k - 1} \in \A_{f, \Sigma}^{<
  \omega}$ of $h_0$ (and thus $q$) such that
  \begin{equation}\label{equation-5}
    \dom(h_{k - 1}) \cap  \left( \dom(p) \right) \omega_1(h_{k - 1})  h_{k
    - 1}^{k - 1} = \varnothing,
    \qquad
    \left(\ran(p) \right) \omega_2(h_{k - 1})\cap \ran(h_{k - 1}) =
    \varnothing,
  \end{equation}
  and $\dom(p)$ and $\ran(p)$ are contained in $\dom(\omega_1(h_{k - 1})
   h_{k - 1}^{k - 1})$ and $\dom(\omega_2(h_{k - 1}))$ respectively.

  Define $h$ to be
  \[
    h_{k - 1} \cup \left\{ \left( (x_i)\omega_1(h_{k - 1})  h_{k - 1}^{k - 1},
    \left((x_i)p \right)\omega_2(h_{k - 1}) \right) : 1 \leq i \leq d \right\}.
  \]
  Recall that $k$ is the order of $\overline{q}$.
  Since $h_{k - 1}$ is an extension of $q$ and $\overline{q} \in S_n$, it follows that
  $\overline{h_{k - 1}} = \overline{q}$, thus $\overline{h_{k - 1}^k} = \id$. Also
  $\omega_1(h_{k - 1})$ and $\omega_2(h_{k - 1})$ are extensions of $\omega_1(q_1)$ and
  $\omega_2(q_2)$ respectively, hence
  \[
    \overline{\omega_1(h_{k - 1})} =\overline{\omega_1(q_1)} = \id =
    \overline{\omega_2(q_2)} = \overline{\omega_2(h_{k - 1})}.
  \]
  Then $x_i$, $(x_i)\omega_1(h_{k - 1})  h_{k - 1}^k$, and
  $\left((x_i)p\right) \omega_2(h_{k - 1})$ are in the same connected component
  of $nK_\omega$ for all $i$. Thus it follows from
  Lemma~\ref{lem-4-easy-one-point} and \eqref{equation-5}, that $h \in \aut(n
  K_\omega) ^{< \omega}$.

  We will now show that $h$ can be obtained from $h_{k - 1}$ by repeated
  applications of Lemma~\ref{lem-4-amalgamation}, and so $h \in \A_{f,
  \Sigma}^{< \omega}$.  First of all, note that $\Sigma \subseteq \dom(q)
  \subseteq \dom(h_{k - 1})$, and that no incomplete components of $h_0$, and thus of
  $q_2$, were amalgamated in $h_{k - 1}$. According to Lemma~\ref{lem-4-main},
  $q_2$ was chosen so that $\left((x_i)p \right) \omega_2(q_2)$ is not on a
  incomplete component of $q_2$ extending an incomplete component of $q_1$ for all $i \in \{1,
  \ldots, d\}$. Hence the vertex $\left((x_i)p \right) \omega_2(h_{k - 1})$ is
  not on an incomplete component of $h_{k - 1}$ extending an incomplete component of $q_1$ for
  all $i \in \{1, \ldots, d\}$.  Also since $\Sigma \subseteq \dom(q) \subseteq
  \dom(q_1)$, it follows that the intersection of any incomplete component of $h_{k -
  1}$ containing a vertex in $\left( \ran(p) \right) \omega_2(h_{k - 1})$ and
  $\Sigma$ is empty.

  By \eqref{equation-7} every vertex in $\left( \dom(p) \right) \omega_1(\changed{q_1})$
   is on an incomplete component of $q_1$ and since $\omega_1(h_{k - 1})  h_{k
  - 1}^{k - 1}$ is defined on $\dom(p)$ it follows that  every vertex in
  $\left( \dom(p) \right) \omega_1(h_{k - 1})  h_{k - 1}^{k - 1}$ is on a
  incomplete component of $h_{k - 1}$ extending an incomplete component of $q_1$.  Hence
  incomplete components of $h_{k - 1}$ containing vertices $\left( \ran(p) \right)
  \omega_2(h_{k - 1})$ are distinct from the incomplete components of $h_{k - 1}$
  containing the vertices $\left( \dom(p) \right) \omega_1(h_{k - 1}) 
  h_{k - 1}^{k - 1}$. Also recall that for every incomplete component $P$ of $h_0$ we
  have that
  \[
    \overline{h_0|_{\dom(h_0)\setminus P}} \in S_n.
  \]
  Since $h_{k - 1}$ is an extension of $h_0$ and no incomplete components of $h_0$
  were amalgamated, for any incomplete component $Q$ of $h_{k - 1}$
  \[
    \overline{h_{k - 1}|_{\dom(h_{k - 1})\setminus Q}} \in S_n.
  \]
  Thus we can apply Lemma~\ref{lem-4-amalgamation} to show that $h \in \A_{f,
  \Sigma}^{< \omega}$.

  Finally $h$ was defined so that
  \[
    \omega_1(h) h^k  \omega_2(h)^{-1} \in [p],
  \]
  and thus any extension $g \in [h] \cap \A_{f, \Sigma}$ also satisfies $g \in
  \{r \in \A_{f, \Sigma} : \langle f, r \rangle \cap [p] \neq \varnothing \}$.
  Therefore, $\{g \in \A_{f, \Sigma} : \langle f, g \rangle \cap [p] \neq
  \varnothing \}$ is dense in $\A_{f, \Sigma}$ as required.
\end{proof}

\subsection*{Proof of Lemma~\ref{lem-4-main}}

The purpose of this section is to prove Lemma~\ref{lem-4-main}.  We will first
prove a technical result relating to the behaviour of a non-stabilizing
isomorphism $f$ of $nK_\omega$. Recall that $f \in \aut(n K_\omega)$
\emph{non-stabilizing} if for all $\Gamma \subsetneq nK_\omega$, all $x \in
\Gamma$ and all $q \in \A_f^{< \omega}$ there is $g \in [q] \cap \A_f$ such
that $(x)h \notin \Gamma$ for some $h \in \langle f, g \rangle$. 

Let $f \in \aut(n K_\omega)$ be non-stabilizing and let $x\in nK_{\omega}$.
Then for every $q\in \aut(n K_\omega)^{<\omega}$ there is $g\in [q]\cap \A_f$
such that $(x)h\not\in \dom(q)$ for some $h\in \langle f, g \rangle$. It
follows that there is $N \in \N$ and $m_1, m_2, \ldots, m_{2N}\in \Z$ such that
$(x)\prod_{i=1}^N g^{m_{2i-1}} f^{m_{2i}} \notin \dom(q)$. If we assume that
the length of the product $\sum_{i=1}^{2N} |m_i|$ is minimal, then the image of
$x$ under any proper prefix of the product $\prod_{i=1}^N g^{m_{2i-1}}
f^{m_{2i}}$ belongs to $\dom(q)$. Therefore 
\[
  (x)\prod_{i=1}^N q^{m_{2i-1}} f^{m_{2i}} = (x)\prod_{i=1}^N g^{m_{2i-1}}
  f^{m_{2i}} \in n K_\omega \setminus \dom(q).
\]
In the next lemma we show that the powers $m_{2i-1}$ of $q$ in the above
equation, can be chosen to be positive.

\begin{lem}\label{lem-4-non-stab}
  Let $f$ be non-stabilizing and let $x \in nK_\omega$. Then for every $q
  \in \aut(nK_\omega)^{< \omega}$ there is $N \in \N$ and $m_1, m_2, \ldots,
  m_{2N} \in \Z$ such that $m_1, m_3, \ldots m_{2N - 1} > 0$ and
  \[
   (x)\prod_{i = 1}^N q^{m_{2i - 1}} f^{m_{2i}} \in n K_\omega \setminus \dom(q).
  \]
\end{lem}

\begin{proof}
  By the discussion above there are $K \in \N$ and $k_1, k_2, \ldots, k_{2K}
  \in \Z$ such that
  \[
    (x)\prod_{i=1}^K q^{k_{2i-1}} f^{k_{2i}} \in nK_\omega \setminus \dom(q).
  \]
  Suppose that $M \in \{0, 1, \ldots, K\}$ is the least value such that $(x)
  \prod_{i=1}^M q^{k_{2i-1}} f^{k_{2i}}$ is on an incomplete component of $q$, where
  $M=0$ in the case that $x$ is on an incomplete component. Then $y_t =(x)\prod_{i=1}^t
  q^{k_{2i-1}} f^{k_{2i}}$ is on a complete component of $q$ for all
  $t\in\{0,\ldots, M-1\}$. It follows that there exist $m_{2t + 1} > 0$ such
  that $(y_t)q^{m_{2t + 1}} = (y_t)q^{k_{2t + 1}}$ for all $t\in\{0,\ldots,
  M-1\}$. Additionally, define $m_{2i} = k_{2i}$ for all $i \in \{1, \ldots,
  M\}$.

  By the choice of $M$, $y = (x) \prod_{i=1}^M q^{m_{2i-1}} f^{m_{2i}} =
  (x)\prod_{i=1}^M q^{k_{2i-1}} f^{k_{2i}}$ is in an incomplete component of $q$. Hence
  there is $z$ in the incomplete component of $y$ under $q$ such that $z \in \ran(q)
  \setminus \dom(q)$ and there is $m_{2M + 1} \geq 0$ such that $(y)q^{m_{2M +
  1}}=z \notin \dom(q)$. Therefore
  \[
  (x)\left(\prod_{i=1}^{M} q^{m_{2i-1}} f^{m_{2i}}\right) q^{m_{2M+1}} \in
  nK_\omega \setminus \dom(q),
  \]
  as required.
\end{proof}

For the proofs of the next three lemmas we require the following notation.
First of all, recall that for a fixed $f \in \aut(n K_\omega)$, if $p \in \aut(n
K_\omega)^{< \omega}$ and $w = \alpha^{n_1} \beta^{n_2} \cdots \beta^{n_{2N}}
\in F_{\alpha, \beta}$ for some $N \in \N$ and $n_1, \ldots, n_{2N} \in \Z$,
then 
\[
  w(p) = p^{n_1} f^{n_2} p^{n_3} \cdots p^{n_{2N - 1}} f^{n_{2N}}
\]
where the product on the right hand side is the usual product of partial
permutations. Let $\Gamma, \Theta, \Phi, \Delta  \subseteq n K_\omega$ be a
finite subsets, let $p \in \A_{f, \Sigma}^{< \omega}$ and let $w \in F_{\alpha,
\beta}$. Suppose $x \in \Gamma$ and define $w_{p,x}$ to be the largest prefix
of $w$ such that $x \in \dom( w_{p,x} (p))$ and let $w_{p,x}$ be the empty word
if there are no such prefix. To make the notation less cluttered, whenever
possible, we will identify the word $w_{p, x}$ with its realisation in $\aut(n
K_\omega)^{< \omega}$, in other words with the partial isomorphism $w_{p,
x}(p)$. To avoid confusion, if $w, w' \in F_{\alpha, \beta}$, we denote that
$w$ and $w'$ are equal by $w \equiv w'$. Note that if $w_{p,x}$ is a proper
prefix of $w$ (i.e. $|w_{p, x}| < |w|$), since $f$ is an isomorphism we have
that $(x) w_{p, x} \notin \dom(p)$ and $w_{p, x} \alpha$ is a prefix of
$w$. 

Suppose that $\Theta \subseteq \Gamma$.  Then we say that $p$ satisfies
$\S(\Gamma, \Theta, \Phi, \Delta, w)$ if the following conditions are
satisfied:
\begin{enumerate}
  \item $\overline{w(p)} = \id$;
  \item $\ran(p) \cap \Delta = \varnothing$;
  \item $\dom\left( w(p) \right) \cap \Gamma = \Theta$;
  \item the image of $\Theta$ under $w(p)$ is disjoint from $\dom(p)$;
  \item $(x)w_{p,x} \neq (y)w_{p, y}$ for all $x, y \in \Gamma$ such that
    $x\neq y$;
  \item $(x) w_{p,x} p^m \in n K_\omega \setminus \Phi$ for all 
    $x \in \Gamma$ and $m \in \Z$ such that $x \in \dom(w_{p, x} p^m)$.
\end{enumerate}
Finally, define $\mathbf{b}(w)$ to be the total number of occurrences of
$\beta$ and $\beta^{-1}$ in the freely reduced word $w$.

Using the definition of $\S(\Gamma, \Theta, \Phi, \Delta, w)$ we can now
restate Lemma~\ref{lem-4-main}. \changed{In the case that $\Gamma = \Theta$, it
follows that $w_{p, x} = w$ for all $x\in \Gamma$. Hence (5), in this case, is
a consequence of $w(p)$ being a finite isomorphism.}

\begin{lem}\label{lem-4-main-restate}
  Let $n \in \N$ be such that $n > 1$ and let $f \in \aut(n K_\omega)$ be
  non-stabilising. If $n = 2$ and $\overline{f} = \id$, then further suppose
  that $\fix(f)$ is finite. Let $\Gamma, \Delta \subseteq nK_\omega$ be finite
  and disjoint, and let $q \in \A_{f, \Sigma}^{< \omega}$ be such that
  $\overline{q} \in S_n$ and $\ran(q) \cap \Delta = \varnothing$. Then there is
  an extension $h \in \A_{f, \Sigma}^{< \omega}$ \changed{of $q$} and $w \in
  F_{ \alpha, \beta}$ satisfying $\S(\Gamma, \Gamma, \dom(q), \Delta, w)$.
\end{lem}

The proof of Lemma~\ref{lem-4-main-restate} will be split into three parts.
\changed{We say that a word $w \in F_{\alpha, \beta}$ \textit{starts with} a
letter $\gamma \in \{\alpha, \beta\}$ if there is $w' \in F_{\alpha, \beta}$
such that $w = \gamma w'$.}

\begin{lem} \label{lem-4-base-case}
  Let $n \in \N$ be such that $n > 1$ and let $f \in \aut(n K_\omega)$ be
  non-stabilising. If $n = 2$ and $\overline{f} = \id$, then further suppose
  that $\fix(f)$ is finite. Let $\Gamma, \Delta \subseteq n K_\omega$ be
  finite, and let $q \in \A_{f, \Sigma}^{< \omega}$ be such that $\ran(q) \cap
  \Delta = \varnothing$. Then there is an extension $h \in \A_{f, \Sigma}^{<
  \omega}$ of $q$ and $w \in F_{\alpha, \beta}$ not containing $\alpha^{-1}$
  and starting with $\alpha$ such that  $h$ satisfies $\S(\Gamma, \varnothing,
  \dom(q), \Delta, w)$.
\end{lem}

\begin{proof}
  If necessary by extending $q$ using Lemma~\ref{lem-4-easy-one-point} and
  Lemma~\ref{lem-4-one-point-ext}, we may  assume that $\overline{q} \in S_n$
  and $\Sigma, \Gamma \subseteq \dom(q)$. In the case that $n = 2$ and
  $\overline{f} = \id$, we also assume that $\fix(f) \subseteq \dom(q)$.

  \changed{Let $d = |\Gamma|$}. We will now inductively define a sequence $q_0,
  \ldots, q_d \in \A_{f, \Sigma}^{< \omega}$ of extensions of $q$, and a
  sequence $\lambda^{(0)},\ldots, \lambda^{(d)}$ of words in $F_{\alpha,
  \beta}$ so that $h = q_d$ and $w = \lambda^{(d)}$ are as required. Let $q_0 =
  q$, let $\Gamma_0 = \varnothing$, and let $\lambda^{(0)} \changed{ = \alpha}$.
  Suppose that for some $j \in \{0, \ldots, d - 1\}$ we have $\Gamma_j
  \subseteq \Gamma$, a word $\lambda^{(j)}$ in $F_{\alpha, \beta}$ starting
  with $\alpha$ and not containing $\alpha^{-1}$, and $q_j \in \A_{f,
  \Sigma}^{< \omega}$ such that $|\Gamma_j| = j$ and
  \begin{enumerate}[(I)]
    \item $\ran(q_j) \cap \Delta = \varnothing$;
    \item $(u){\lambda^{(j)}}_{q_j,u} \neq (v){\lambda^{(j)}}_{q_j, v}$ for all $u, v
      \in \Gamma_j$ with $u \neq v$;
    \item $(u) {\lambda^{(j)}}_{q_j, u} q_j^m \notin \dom(q)$ for all $m
      \in \Z$ such that $u \in \dom({\lambda^{(j)}}_{q_j, u} q_j^m)$ and all $u
      \in \Gamma_j$;
    \item ${\lambda^{(j)}}_{q_j, u} \not\equiv \lambda^{(j)}$ for all $u \in \Gamma_j$.
  \end{enumerate}

  Let $x \in \Gamma \setminus \Gamma_j$ be arbitrary and let $\Gamma_{j + 1} =
  \Gamma_j \cup \{x\}$. The first step in the proof is to find $\nu \in
  F_{\alpha, \beta}$ so that $x \notin \dom \left( \lambda^{(j)} \nu \alpha
  (q_j)\right)$, and find $m \in \N$ such that $m > |\lambda^{(j)} \nu|$ and it
  so that we can define
  \begin{equation}\label{equation-w-j+1-def}
    \lambda^{(j + 1)} \equiv \lambda^{(j)} \nu \alpha^m \beta \alpha.
  \end{equation}
  
  In order to define $\nu$ consider two cases. If $x \in \dom \left(
  \lambda^{(j)}(q_j) \right)$, then by Lemma~\ref{lem-4-non-stab} that there is
  $\nu \in F_{\alpha, \beta}$ such that $\alpha^{-1}$ is not contained in $\nu$
  and the image of $x$ under $\lambda^{(j)} \nu (q_j)$ is in $n K_\omega
  \setminus \dom(q_j)$. Otherwise, $x \notin \dom \left( \lambda^{(j)}(q_j)
  \right)$, in which case let $\nu$ be the empty word. Hence in both cases 
  \begin{equation}\label{equation-x-not-in-dom}
    x \notin \dom \left( \lambda^{(j)}  \nu  \alpha(q_j)\right).
  \end{equation}

  To define $m$ we will again consider two separate cases. If $n = 2$ and
  $\overline{f} = \id$, let $m \changed{>} |\lambda^{(j)} \nu |$ be arbitrary.
  Otherwise, either $n = 2$ and $\overline{f} = (1 \; 2)$ or $n \geq 3$. Let
  $L_1, \ldots, L_n$ be the connected components of $n K_\omega$, and let $a
  \in \{1, \ldots, n\}$ so that $x \in L_a$. Consider any extension $g \in
  \aut(n K_\omega)$ of $q_j$, and let $b$ be the image of $a$ under the
  permutation $\overline{\left(\lambda^{(j)} \nu \right) (g)}$. Since
  $\overline{q_j} \in S_n$, it follows that $b$ is independent of the extension
  $g$. We will show that in this case we can choose $m > |\lambda^{(j)} \nu|$
  to be such that
  \begin{equation}\label{equation-choice-m}
    (b) \overline{q_j}^m \in \supp(\overline{f}).
  \end{equation}

  If $n = 2$ and  $\overline{f} = (1\; 2)$, then any $m > |\lambda^{(j)} \nu|$
  satisfies \eqref{equation-choice-m}. Let $n \geq 3$ be arbitrary, and let $O$
  be the orbit of $\overline{q_j}$ containing $b$. Suppose that $\overline{f}$
  fixes $O$ pointwise. If $|O| \leq 2$, then since $n \geq 3$, there is $c \in
  \{1, \ldots, n\} \setminus O$, and so $(b \; c) \notin \langle \overline{f},
  \overline{q_j} \rangle$. If $|O| \geq 3$, then the symmetric group on $|O|$
  is not cyclic, and so there is a $\sigma \in S_n$ such that $\supp(\sigma)
  \subseteq O$ and $\sigma \notin \langle \overline{q_j}|_O \rangle$. Then
  $\sigma \notin \langle \overline{f}, \overline{q_j} \rangle$. However, both
  cases are impossible since $\langle \overline{f}, \overline{q_j} \rangle =
  S_n$. Hence $\overline{f}$ does not fix $O$ pointwise. Hence we may choose $m
  > |\lambda^{(j)} \nu|$ to satisfy \eqref{equation-choice-m}. Let $\lambda^{(j
  + 1)}$ be as in \eqref{equation-w-j+1-def}. For brevity, denote the prefix
  $\lambda^{(j)} \nu \alpha^m \beta$ of $\lambda^{(j + 1)}$ by $\rho$.

  Next we show how to construct $q_{j + 1} \in \A_{f, \Sigma}^{< \omega}$ from
  $q_j$.  In order to do so, we need to consider a possible complication,
  namely the existence of $y \in \Gamma_j$ such that $(y) {\lambda^{(j +
  1)}}_{q_j, y} = (x){\lambda^{(j + 1)}}_{q_j, x}$. The case where such $y$
  does not exists is slightly easier and can be proved in a very similar
  fashion, simply ignoring any mention of $y$ in the following argument (to be
  more precise (i), (ii), (iv), and (v) are exactly the same, (vii) and (viii)
  are unnecessary, and in (iii) and (vi) the vertex $u$ can be any vertex in
  the set $\Gamma_j$). Hence we will omit this case.  Suppose there is $y \in
  \Gamma_j$ such that $(y){\lambda^{(j + 1)}}_{q_j, y} = (x){\lambda^{(j +
  1)}}_{q_j, x}$. It follows from (II) that such $y$ is unique. Since
  ${\lambda^{(j + 1)}}_{q_j, x}$ is a partial isomorphism and $x \neq y$, it
  follows that ${\lambda^{(j + 1)}}_{q_j, x} \neq {\lambda^{(j + 1)}}_{q_j,
  y}$, and so ${\lambda^{(j + 1)}}_{q_j, x} \not\equiv {\lambda^{(j +
  1)}}_{q_j, y}$. Condition (IV) implies that  ${\lambda^{(j)}}_{q_j, y}$ is a
  proper prefix of $\lambda^{(j)}$, and so $y \notin \dom \left( \rho(q_j)
  \right)$. Also from \eqref{equation-x-not-in-dom}, we have that 
  \begin{equation}\label{equation-w-j1-x}
    |{\lambda^{(j + 1)}}_{q_j, x}| \leq |\lambda^{(j)} \nu| < |\rho|.
  \end{equation}
  Hence $|{\lambda^{(j + 1)}}_{q_j, x}|, |{\lambda^{(j + 1)}}_{q_j, y}| <
  |\rho|$. There are two cases to consider: either $|{\lambda^{(j + 1)}}_{q_j,
  x}| > |{\lambda^{(j + 1)}}_{q_j, y}|$ or $|{\lambda^{(j + 1)}}_{q_j, x}| <
  |{\lambda^{(j + 1)}}_{q_j, y}|$.

  Consider $|{\lambda^{(j + 1)}}_{q_j, x}| > |{\lambda^{(j + 1)}}_{q_j, y}|$. We proceed by
  inductively constructing a sequence $r_0, \ldots, r_{|\rho|}$ of extensions
  of $q_j$, so that $r_0 = q_j$ and $r_{|\rho|}$ is the required $q_{j + 1}$.
  Let $r_0 = q_j$. For $k \in \{0, \ldots, |\rho|\}$ let the inductive
  hypothesis be as follows: there is an extension $r_k \in \A_{f, \Sigma}^{<
  \omega}$ of $r_{k - 1}$ (or $q_j$ if $k = 0$) such that 
  \begin{enumerate}[(i)]
    \item $k \leq |{\lambda^{(j + 1)}}_{r_k, x}| \leq |\rho|$;

    \item $\ran(r_k) \cap \Delta = \varnothing$;

    \item ${\lambda^{(j + 1)}}_{r_k,u} \equiv {\lambda^{(j)}}_{q_j,u}$ for $u
      \in \Gamma_j \setminus \{ y \}$;

    \item $(u){\lambda^{(j + 1)}}_{r_k,u} \neq (v){\lambda^{(j + 1)}}_{r_k, v}$
      for $u, v \in \Gamma_j$ with $u \neq v$;

    \item $(u) {\lambda^{(j + 1)}}_{r_k, u} r_k^m \in n K_\omega \setminus \dom(q)$ for all $m \in \Z$
      such that $u \in \dom({\lambda^{(j + 1)}}_{r_k, u} r_k^m)$ and all $u \in
      \Gamma_j$;

    \item $(x) {\lambda^{(j + 1)}}_{r_k, x} \notin \dom(r_k) \cup
      \left\{(u){\lambda^{(j)}}_{q_j, u}: u \in \Gamma_j \setminus \{y\} \right\}$.
      Moreover if $k > 0$ and we can write ${\lambda^{(j + 1)}}_{r_k, x} \equiv \tau
      \beta^i$ for some $i \in \Z \setminus\{0\}$, and $\tau \in
      F_{\alpha, \beta}$ such that $\tau$ ends with a letter $\alpha$ and the
      image of $x$ under $\tau(r_k)$ is in  $\supp(f^i)$, then $(x){\lambda^{(j +
      1)}}_{r_k, x} \notin \dom(r_k) \cup \ran(r_k)$;

    \item if $k > 0$ and $(x){\lambda^{(j + 1)}}_{r_{k - 1}, x} \neq (y){\lambda^{(j +
      1)}}_{r_{k - 1}, y}$ then $(x){\lambda^{(j + 1)}}_{r_k, x} \neq (y){\lambda^{(j +
      1)}}_{r_k, y}$.

    \item $|{\lambda^{(j + 1)}}_{r_k, x}| > |{\lambda^{(j + 1)}}_{r_k, y}|$. Moreover, if
      $(x){\lambda^{(j + 1)}}_{r_k, x} = (y){\lambda^{(j + 1)}}_{r_k, y}$, then
      $|{\lambda^{(j + 1)}}_{r_k, y}| \geq |{\lambda^{(j + 1)}}_{q_j, y}| + k$;
  \end{enumerate}

  We will first demonstrate that the base case, $k = 0$, holds. The condition
  (i) is satisfied by $r_0$ by~\eqref{equation-w-j1-x}, and condition (ii) is
  satisfied because $r_0 = q_j$ satisfies (I). Since $q_j$ satisfies (IV) we
  have that ${\lambda^{(j)}}_{q_j, u} \not\equiv \lambda^{(j)}$  and thus $u
  \notin \dom \left(\lambda^{(j)}(q_j) \right)$ which then implies that
  ${\lambda^{(j + 1)}}_{q_j, u} \equiv {\lambda^{(j)}}_{q_j, u}$ for all $u \in
  \Gamma_j$. Hence (iii) is satisfied by $r_0$. Since ${\lambda^{(j +
  1)}}_{q_j, u} \equiv {\lambda^{(j)}}_{q_j, u}$ for all $u \in \Gamma_j$, the
  conditions (iv) and (v) are the same as conditions (II) and (III)
  respectively. Recall that $(x) {\lambda^{(j + 1)}}_{q_j, x} \in  n K_\omega
  \setminus \dom(q_j)$ by the definition of ${\lambda^{(j + 1)}}_{q_j, x}$, and
  that if $(x){\lambda^{(j + 1)}}_{q_j, x} = (u) {\lambda^{(j + 1)}}_{q_j, u}$
  where $u \in \Gamma_j$ then $u = y$ by (II). Hence $r_0$ satisfies
  the first part of (vi), while $r_0$ satisfies second part of (vi), (vii), and
  second part of (viii) trivially, since $k = 0$. Finally, the first part of
  (viii) is just the assumption of this case. Therefore $r_0$ satisfies the
  inductive hypothesis.

  Next we show how to obtain $r_{k + 1}$ from $r_k$. Suppose that for some $k
  \in \{0, \ldots, |\rho| - 1 \}$ we have $r_k \in \A_{f, \Sigma}^{< \omega}$
  which satisfies (i) -- (viii). We consider the case ${\lambda^{(j +
  1)}}_{r_k, x} \equiv \rho$ and ${\lambda^{(j + 1)}}_{r_k, x}$ being a proper
  prefix of $\rho$ separately.

  \textbf{Case 1:} We begin by considering the case where ${\lambda^{(j +
  1)}}_{r_k, x}$ is a proper prefix of $\rho$. Let $z = (x){\lambda^{(j +
  1)}}_{r_k, x}$. Since ${\lambda^{(j + 1)}}_{r_k, x}$ is a proper prefix of
  $\lambda^{(j + 1)}$, it follows that  $z \notin \dom(r_k)$ and ${\lambda^{(j
  + 1)}}_{r_k, x} \alpha$ is a prefix of $\lambda^{(j + 1)}$. Recall that
  $\mathbf{b}(\lambda^{(j + 1)})$ is the total number of occurrences of letters
  $\beta$ and $\beta^{-1}$ in the word ${\lambda^{(j + 1)} \in F}_{\alpha,
  \beta}$. Let $c \in \{1, \ldots, n\}$ be so that $z \in L_c$, and choose
  \[
    z' \in L_{(c)\overline{r_k}} \setminus \bigcup_{i = -
    \mathbf{b}(\lambda^{(j + 1)})}^{\mathbf{b}(\lambda^{(j + 1)})} \left(
    \Delta \cup \left\{ (u){\lambda^{(j + 1)}}_{r_k, u}: u \in \Gamma_j
    \right\} \cup  \dom(r_k) \cup \ran(r_k) \cup \{z\} \right)f^{-i}.
  \]
  Since $z \notin \dom(r_k)$ and $z' \notin \dom(r_k) \cup \ran(r_k)$ it
  follows from Lemmas~\ref{lem-4-easy-one-point} and~\ref{lem-4-one-point-ext}
  that $r_{k + 1} = r_k \cup \{(z, z')\} \in \A_{f, \Sigma}^{< \omega}$. Then
  there is some $i \in \Z$ such that
  \begin{equation}\label{equation-w-j+1-r-k+1}
    {\lambda^{(j + 1)}}_{r_{k + 1}, x} \equiv {\lambda^{(j + 1)}}_{r_k, x}
    \alpha \beta^i.
  \end{equation}
  Hence $|{\lambda^{(j + 1)}}_{r_{k + 1}, x}| > |{\lambda^{(j + 1)}}_{r_k, x}|
  \geq k$. We will now show that ${\lambda^{(j + 1)}}_{r_{k + 1},x}$ is a
  prefix of $\rho$. Suppose that ${\lambda^{(j + 1)}}_{r_{k + 1},x}$ is not a
  prefix of $\rho$. Since ${\lambda^{(j + 1)}}_{r_{k + 1},x}$ is a prefix of
  $\lambda^{(j + 1)}$, it follows that ${\lambda^{(j + 1)}}_{r_{k + 1},x} =
  \lambda^{(j + 1)}$. Hence the fact that $\lambda^{(j + 1)} \equiv \rho
  \alpha$ and \eqref{equation-w-j+1-r-k+1} imply that ${\lambda^{(j +
  1)}}_{r_k, x} \alpha \beta^i \equiv {\lambda^{(j + 1)}}_{r_{k + 1},x} =
  \lambda^{(j + 1)} = \rho \alpha$, thus $i = 0$ and ${\lambda^{(j + 1)}}_{r_k,
  x} = \rho$, which contradicts the assumption of this case. Therefore,
  ${\lambda^{(j + 1)}}_{r_{k + 1}, x}$ is prefix of $\rho$, and so (i) is
  satisfied by $r_{k + 1}$. 

  It follows from the definition of $r_{k + 1}$ that
  \begin{equation}\label{equation-new-dom-ran}
    \dom(r_{k + 1}) = \dom(r_k) \cup \{z\} 
    \quad \text{and} \quad
    \ran(r_{k + 1}) = \ran(r_k) \cup \{z'\}.
  \end{equation}
  Since the vertex $z'$ was chosen outside $\Delta$ we have that (ii) is
  satisfied by $r_{k + 1}$.

  Let $u \in \Gamma_j \setminus \{y\}$. It follows from (vi) for $r_k$ that $z
  = (x){\lambda^{(j + 1)}}_{r_k, x} \neq (u){\lambda^{(j)}}_{q_j, u}$, and
  since $r_k$ satisfies (iii), it follows that $z \neq  (u){\lambda^{(j +
  1)}}_{r_k, u}$. Also ${\lambda^{(j + 1)}}_{r_k, u} = \lambda^{(j)}_{q_j, u}$
  is a proper prefix of $\lambda^{(j)}$, and so a proper prefix of $\lambda^{(j
  + 1)}$, by (iii) and (IV). Then $(u){\lambda^{(j + 1)}}_{r_k,u} \notin
  \dom(r_k)$ and ${\lambda^{(j + 1)}}_{r_k, u} \alpha$ is a prefix of
  $\lambda^{(j + 1)}$, and thus $(u){\lambda^{(j + 1)}}_{r_k,u} \notin
  \dom(r_{k + 1})$ by \eqref{equation-new-dom-ran}. Hence ${\lambda^{(j +
  1)}}_{r_{k + 1}, u} \equiv {\lambda^{(j + 1)}}_{r_k, u}$, and since $r_k$
  satisfies (iii)
  \begin{equation}
    \label{equation-equiv}
    {\lambda^{(j + 1)}}_{r_{k + 1},u} \equiv {\lambda^{(j + 1)}}_{r_k,u} \equiv
    {\lambda^{(j)}}_{q_j,u}.
  \end{equation}
  Therefore $r_{k + 1}$ satisfies (iii).

  In order to prove that $r_{k + 1}$ satisfies (iv), we consider two cases.
  Suppose that $z = (x){\lambda^{(j + 1)}}_{r_k, x} \neq (y){\lambda^{(j +
  1)}}_{r_k, y}$. It follows by (i) and (viii) that $|{\lambda^{(j +
  1)}}_{r_k, y}| < |\rho|$. Hence, ${\lambda^{(j + 1)}}_{r_k, y}$ is a proper
  prefix of $\lambda^{(j + 1)}$, and so $(y){\lambda^{(j + 1)}}_{r_k,y} \notin
  \dom(r_k)$ and ${\lambda^{(j + 1)}}_{r_k, y} \alpha$ is a prefix of
  $\lambda^{(j + 1)}$, and so $(y){\lambda^{(j + 1)}}_{r_k,y} \notin \dom(r_{k
  + 1})$ by \eqref{equation-new-dom-ran}. Hence ${\lambda^{(j + 1)}}_{r_{k +
  1},y} \equiv {\lambda^{(j + 1)}}_{r_k,y}$, in other words
  \begin{equation}\label{equation-w-j+1-y-same}
    z \neq (y){\lambda^{(j + 1)}}_{r_k, y} \implies 
    {\lambda^{(j + 1)}}_{r_{k + 1},y} \equiv {\lambda^{(j + 1)}}_{r_k,y}.
  \end{equation}
  Combining with the previous paragraph ${\lambda^{(j + 1)}}_{r_{k + 1}, u}
  \equiv {\lambda^{(j + 1)}}_{r_k, u}$ for all $u \in \Gamma_j$. Therefore,
  $r_{k + 1}$ satisfies (iv), since $r_k$ does.

  Otherwise, suppose that $z = (x){\lambda^{(j + 1)}}_{r_k, x} = (y){\lambda^{(j + 1)}}_{r_k,
  y}$. Since $(z')f^i \notin \dom(r_{k + 1})$ for all $i \in
  \{-\mathbf{b}(\rho), \ldots, \mathbf{b}(\rho)\}$ by the choice of $z'$ and
  \eqref{equation-new-dom-ran}, there exists $i \in \{-\mathbf{b}(\rho),
  \ldots, \mathbf{b}(\rho)\}$ such that $(y){\lambda^{(j + 1)}}_{r_{k + 1}, y} =
  (z')f^i$, and so ${\lambda^{(j + 1)}}_{r_{k + 1}, y} \equiv {\lambda^{(j + 1)}}_{r_k, y}
  \alpha \beta^i$, in other words
  \begin{equation}\label{equation-w-j+1-y}
    (x){\lambda^{(j + 1)}}_{r_k, x} = (y){\lambda^{(j + 1)}}_{r_k, y} \implies
    {\lambda^{(j +
    1)}}_{r_{k + 1}, y} \equiv {\lambda^{(j + 1)}}_{r_k, y} \alpha \beta^i \
    \text{for some} \ i \in \{ -\mathbf{b}(\rho), \ldots, \mathbf{b}(\rho) \}.
  \end{equation}
  The vertex $z'$ was chosen so that $(z')f^i \neq (u) {\lambda^{(j +
  1)}}_{r_k, u}$ for all $u \in \Gamma_j \setminus \{y\}$. Since ${\lambda^{(j
  + 1)}}_{r_{k + 1}, u} \equiv {\lambda^{(j + 1)}}_{r_k, u}$ for all $u \in
  \Gamma_j \setminus \{y\}$ and $r_k$ satisfies (iv), it then follows that
  $r_{k + 1}$ satisfies (iv).

  Let $u \in \Gamma_j \setminus \{y \}$ be arbitrary.
  Then $(u){\lambda^{(j + 1)}}_{r_{k+1}, u} = (u){\lambda^{(j
  + 1)}}_{r_{k}, u}$ by \eqref{equation-equiv}. 
  Since $z'\not\in\dom(r_k)$, no two components of $r_k$ become subsets of the
  same component of $r_{k+1}$. It follows that, for any $m\in \Z$, $(u){\lambda^{(j +
  1)}}_{r_{k+1}, u} r_{k+1} ^ m$ equals  either $(u){\lambda^{(j +
  1)}}_{r_{k}, u} r_{k} ^ m$ or $z'$, neither of which belongs to $\dom(q)$.
  Hence (v) holds for all $u\in \Gamma_j\setminus \{y\}$. 
  
  By \eqref{equation-w-j+1-y-same}, if $z \neq (y){\lambda^{(j + 1)}}_{r_k, y}$
  then ${\lambda^{(j + 1)}}_{r_{k + 1}, y} \equiv {\lambda^{(j + 1)}}_{r_k,
  y}$, and so using the argument of the previous paragraph, $(y){\lambda^{(j +
  1)}}_{r_{k + 1}, y} r_{k+1} ^ m\not \in \dom(q)$ for all $m\in \mathbb{Z}$.
  Hence to show that $r_{k + 1}$ satisfies (v) it remains to consider the case
  where $z = (x){\lambda^{(j + 1)}}_{r_k, x}= (y){\lambda^{(j + 1)}}_{r_k, y}$.
  It follows from \eqref{equation-w-j+1-y} that  $(y){\lambda^{(j + 1)}}_{r_{k
  + 1}, y} = (z')f^i$ for some $i \in \{-\mathbf{b}(\rho), \ldots,
  \mathbf{b}(\rho)\}$. If $(z')f^i \notin \dom(r_{k + 1}) \cup \ran(r_{k +
  1})$, then no component of $r_{k + 1}$, and thus $q$, contains the vertex
  $(z')f^i = (y){\lambda^{(j + 1)}}_{r_{k + 1}, y}$, and so $r_{k + 1}$
  satisfies (v). Suppose that $(z')f^i \in \dom(r_{k + 1}) \cup \ran(r_{k +
  1})$. But $z'$ was chosen so that $(z')f^i \notin \dom(r_k) \cup
  \ran(r_k)\cup \{z\}$, which implies $(z')f^i = z'$ and so $(y){\lambda^{(j +
  1)}}_{r_{k + 1}, y} = z'$. From its definition, the component of $r_{k + 1}$
  containing $z' = (y){\lambda^{(j + 1)}}_{r_{k + 1}, y}$ is the component of
  $r_k$ containing $z = (y){\lambda^{(j + 1)}}_{r_k, y}$ together with the
  vertex $z'$. In other words $(y){\lambda^{(j + 1)}}_{r_{k+1}, y}r_{k+1}^m$,
  equals $(y){\lambda^{(j + 1)}}_{r_k, y}r_k^m$ or $z'$, if defined. Since
  $(y){\lambda^{(j + 1)}}_{r_k, y}r_k^m\in n K_\omega \setminus \dom(q)$ for
  all $m\in \Z$, it follows that  $(y){\lambda^{(j + 1)}}_{r_{k+1},
  y}r_{k+1}^m\not \in \dom(q)$ for all $m\in \Z$. Thus $r_{k+1}$ satisfies
  condition (v).

  By \eqref{equation-w-j+1-r-k+1}, ${\lambda^{(j + 1)}}_{r_{k + 1}, x} \equiv
  {\lambda^{(j + 1)}}_{r_k, x} \alpha \beta^i$ for some $i \in
  \{-\mathbf{b}(\rho), \ldots, \mathbf{b}(\rho)\}$. 
  Hence $(x){\lambda^{(j + 1)}}_{r_{k + 1}, x} = (z')f^i \not \in 
  \dom(r_k)\cup \{z\} \cup \left\{(u){\lambda^{(j + 1)}}_{r_k, u}: u \in
  \Gamma_j \right\}$ by the choice of $z'$. By (iii), $(u){\lambda^{(j +
  1)}}_{r_k, u} = (u){\lambda^{(j + 1)}}_{q_j, u}$ for all $u\in
  \Gamma_j\setminus\{y\}$ and $\dom(r_{k+1}) = \dom(r_k)\cup \{z\}$, and so the
  first part of (vi) is satisfied by $r_{k+1}$. 
  To check the second part of (vi), suppose that ${\lambda^{(j+1)}}_{r_{k+1},
  x}\equiv \tau \beta^i$ for some $i \in \Z\setminus\{0\}$ and $\tau \in
  F_{\alpha, \beta}$ such that $\tau$ ends with a letter $\alpha$ and the image
  of $x$ under $\tau(r_{k+1})$ is in $\supp(f^i)$. Then, by \eqref{equation-w-j+1-r-k+1}, 
  $\tau = {\lambda^{(j+1)}}_{r_k, x}\alpha$ and the last part of the assumption
  from the previous sentence becomes
  $z' = (x) {\lambda^{(j + 1)}}_{r_k, x} r_{k+1} \in \supp(f^i)$.
  Then $(x){\lambda^{(j + 1)}}_{r_{k + 1}, x} = (z')f^i \neq z'$. Since
  $(z')f^i \notin \dom(r_k) \cup \ran(r_k)
  \cup \{z\}$ by the choice of $z'$, it follows from
  \eqref{equation-new-dom-ran} that $(x){\lambda^{(j + 1)}}_{r_{k + 1}, x}
  \notin \dom(r_{k + 1}) \cup \ran(r_{k + 1})$. Therefore, $r_{k + 1}$
  satisfies (vi).

  By \eqref{equation-w-j+1-y-same} if $z = (x){\lambda^{(j + 1)}}_{r_k, x} \neq
  (y){\lambda^{(j + 1)}}_{r_k, y}$, then ${\lambda^{(j + 1)}}_{r_{k + 1}, y}
  \equiv {\lambda^{(j + 1)}}_{r_k, y}$, so $(y){\lambda^{(j + 1)}}_{r_k, y} =
  (y){\lambda^{(j + 1)}}_{r_{k + 1}, y}$. It follows from
  \eqref{equation-w-j+1-r-k+1} that there is $i \in \{-\mathbf{b}(\rho), \ldots,
  \mathbf{b}(\rho)\}$ so that $(x){\lambda^{(j + 1)}}_{r_{k + 1}, x} =
  (z')f^i$. Hence by the choice of $z'$
  \[
    (x){\lambda^{(j + 1)}}_{r_{k + 1}, x} = (z')f^i \neq  (y){\lambda^{(j +
    1)}}_{r_k, y} = (y){\lambda^{(j + 1)}}_{r_{k + 1}, y},
  \]
  and so (vii) holds for $r_{k + 1}$.

  Finally, we will show that $r_{k + 1}$ satisfies (viii). Suppose that
  $(y){\lambda^{(j + 1)}}_{r_k, y} \neq (x){\lambda^{(j + 1)}}_{r_k, x}$. Then
  ${\lambda^{(j + 1)}}_{r_{k + 1}, y} \equiv {\lambda^{(j + 1)}}_{r_k, y}$ by
  \eqref{equation-w-j+1-y-same}. Since $|{\lambda^{(j + 1)}}_{r_k, x}| <
  |{\lambda^{(j + 1)}}_{r_{k + 1}, x}|$ and $r_k$ satisfies (viii), it follows
  that  $r_{k + 1}$ satisfies (viii) as well. The other case is when
  $(y){\lambda^{(j + 1)}}_{r_k, y} = (x){\lambda^{(j + 1)}}_{r_k, x}$. Then
  ${\lambda^{(j + 1)}}_{r_{k + 1}, y} \equiv {\lambda^{(j + 1)}}_{r_k, y}
  \alpha \beta^i$ for some $i \in \{ -\mathbf{b}(\rho), \ldots,
  \mathbf{b}(\rho)\}$ by~\eqref{equation-w-j+1-y}. Since ${\lambda^{(j +
  1)}}_{r_k, y}$ is a proper prefix of ${\lambda^{(j + 1)}}_{r_k, x}$ by (viii)
  applied to $r_k$, it follows that ${\lambda^{(j + 1)}}_{r_k, y} \alpha$ is a
  prefix of ${\lambda^{(j + 1)}}_{r_k, x}$, \changed{ and so $ {\lambda^{(j +
  1)}}_{r_k, x} =  {\lambda^{(j + 1)}}_{r_k, y} \alpha \beta^{i'}$ for some $i'
  \in \Z$. Suppose that ${\lambda^{(j + 1)}}_{r_{k + 1}, y}$ is not a prefix of
  ${\lambda^{(j + 1)}}_{r_k, x}$, in other words either $i > 0$ and $i' \in
  \{0, \ldots, i - 1\}$; or $i < 0$ and $i' \in \{i + 1, \ldots, 0\}$. Then
  either ${\lambda^{(j + 1)}}_{r_k, x}\beta$ or ${\lambda^{(j + 1)}}_{r_k,
  x}\beta^{-1}$ must be a prefix of $\lambda^{(j + 1)}$, which contradicts
  \eqref{equation-w-j+1-r-k+1}}. Hence ${\lambda^{(j + 1)}}_{r_{k + 1}, y}$ is
  a prefix of ${\lambda^{(j + 1)}}_{r_k, x}$, and thus
  \[
    |{\lambda^{(j + 1)}}_{r_{k + 1}, y}| \leq |{\lambda^{(j + 1)}}_{r_k, x}| <
    |{\lambda^{(j + 1)}}_{r_{k + 1}, x}|.
  \]
  Therefore, $r_{k + 1}$ satisfies first part of (viii).

  In order to show the second part of (viii), suppose that $(x){\lambda^{(j +
  1)}}_{r_{k + 1}, x} = (y){\lambda^{(j + 1)}}_{r_{k + 1}, y}$. Since (vii)
  holds for $r_{k + 1}$ we have that $(x){\lambda^{(j + 1)}}_{r_k, x} =
  (y){\lambda^{(j + 1)}}_{r_k, y}$ and thus $|{\lambda^{(j + 1)}}_{r_k, y}|
  \geq |{\lambda^{(j + 1)}}_{q_j, y}| + k$ by (viii) for $r_k$.
  Also~\eqref{equation-w-j+1-y} implies that $|{\lambda^{(j + 1)}}_{r_k, y}| <
  |{\lambda^{(j + 1)}}_{r_{k + 1}, y}|$. Therefore $|{\lambda^{(j + 1)}}_{r_{k
  + 1}, y}| \geq |{\lambda^{(j + 1)}}_{q_j, y}| + k + 1$ and thus $r_{k + 1}$
  satisfies (viii) and hence this case is complete.

  \textbf{Case 2:} Suppose ${\lambda^{(j + 1)}}_{r_k, x} \equiv \rho$. It
  follows from \eqref{equation-x-not-in-dom} that $|{\lambda^{(j + 1)}}_{r_0,
  x}| < |\rho|$, and so $k > 0$. Let $r_{k + 1} = r_k$. Then $r_{k + 1}$
  trivially satisfies conditions (i) -- (vii) and the first part of condition
  (viii). To show second part of (viii) we will consider two cases. Suppose
  that $n = 2$ and $\overline{f} = \id$. Since $\lambda^{(j + 1)} \equiv \rho
  \alpha$ and ${\lambda^{(j + 1)}}_{r_k, x} \equiv \rho$ it follows that the
  image of $x$ under $\rho (r_k)$ is $(x){\lambda^{(j + 1)}}_{r_k, x} \in n
  K_\omega \setminus \dom(r_k)$. Let $\changed{t} \in nK_\omega$ be the image
  \changed{of} $x$ under $\lambda^{(j)} \nu \alpha^m (r_k)$. Then $\rho \equiv
  \lambda^{(j)} \nu \alpha^m \beta$ implies that $(\changed{t})f$ is the image
  of $x$ under $\rho$, and so if $\changed{t} \in \fix(f)$ 
  \[
    \changed{t} = (\changed{t})f = (x) w_{r_k, x} \in n K_\omega \setminus \dom(r_k)
  \]
  by the assumption that $w_{r_k, x} = \rho$. However, we have assumed at the
  beginning of the proof that $\fix(f) \subseteq \dom(q)$, which is a
  contradiction since $\dom(q) \subseteq \dom(r_k)$. Hence $\changed{t} \in
  \supp(f)$. Otherwise, either $n = 2$ and $\overline{f} = (1 \; 2)$, or $n
  \geq 3$. Recall that $a, b \in \{1, \ldots, n\}$ are such that $x \in L_a$
  and $b$ is the image of $a$ under $\overline{\lambda^{(j)} \nu (r_k)}$. Then
  the image of $a$ under $\overline{\lambda^{(j)} \nu \alpha^m(r_k)}$ is in
  $\supp(\overline{f})$ by \eqref{equation-choice-m}, and so \changed{the
  imaged of $x$ under $\lambda^{(j)} \nu \alpha^m(r_k)$ is in $\supp(f)$} in
  both cases. Hence it follows from the second part of (vi) that 
  \begin{equation}\label{equation-notin-dom-ran}
    (x){\lambda^{(j + 1)}}_{r_k, x} \notin \dom(r_k) \cup \ran(r_k).
  \end{equation}

  Next, using \eqref{equation-notin-dom-ran}, will show that $(x){\lambda^{(j +
  1)}}_{r_k, x} \neq (y){\lambda^{(j + 1)}}_{r_k , y}$, which then implies that
  $r_{k + 1}$ satisfies the second half of (viii), and this case will be
  complete.
  Suppose that $(x){\lambda^{(j + 1)}}_{r_k, x} = (y){\lambda^{(j +
  1)}}_{r_k , y}$. Since ${\lambda^{(j + 1)}}_{r_k, x} \equiv \rho \equiv
  \lambda^{(j)} \nu \alpha^m \beta$ and $|{\lambda^{(j +
  1)}}_{r_0, x}| \leq |\lambda^{(j)} \nu|$
  by~\eqref{equation-x-not-in-dom}, the fact that at any inductive step
  incomplete components of $q_j$ were extended by at most one point, implies
  that $k \geq m$. Since $m$ was chosen so that $m > |\lambda^{(j)}
  \nu|$, and $r_k$ satisfies (viii)
  \[
    |\rho| = |{\lambda^{(j + 1)}}_{r_k, x}| \geq |{\lambda^{(j + 1)}}_{r_k,
    y}| \geq |{\lambda^{(j + 1)}}_{q_j, y}| + k > m > |\lambda^{(j)}
    \nu|.
  \]
  Hence ${\lambda^{(j + 1)}}_{r_k, y}$ is a prefix of $\rho$, and
  $\lambda^{(j)} \nu$ is a prefix of ${\lambda^{(j + 1)}}_{r_k, y}$. The former
  and the fact that $\lambda^{(j+1)} = \rho\alpha$ implies that ${\lambda^{(j +
  1)}}_{r_k, y}$ is a proper prefix of $\lambda^{(j + 1)}$, and so
  ${\lambda^{(j + 1)}}_{r_k, y} \alpha$ is a prefix of $\lambda^{(j + 1)}$ and
  $y \notin \dom(\lambda^{(j + 1)})$. Since $\lambda^{(j + 1)} \equiv
  \lambda^{(j)} \nu \alpha^m \beta \changed{\alpha}$, there is $i \in \{1,
  \ldots, m - 1\}$ such that ${\lambda^{(j + 1)}}_{r_k, y} \equiv \lambda^{(j)}
  \nu \alpha^i$. Hence $(y) {\lambda^{(j + 1)}}_{r_k, y} \in \ran(r_k)$. But
  this contradicts~\eqref{equation-notin-dom-ran}, and so we conclude that
  $(x){\lambda^{(j + 1)}}_{r_k, x} \neq (y){\lambda^{(j + 1)}}_{r_k , y}$.
  Therefore $r_{k + 1}$ satisfies the second part of (viii), since $r_{k + 1} =
  r_k$, as required.

  Hence by induction there is $q_{j + 1} = r_{|\rho|} \in \A_{f, \Sigma}^{<
  \omega}$ satisfying conditions (i) -- (viii). We will now show that $q_{j
  + 1}$ satisfies (I) -- (IV).

  It follows from (ii) that $q_{j + 1}$ satisfies (I). Suppose that
  $(x){\lambda^{(j + 1)}}_{q_{j + 1}, x} = (y) {\lambda^{(j + 1)}}_{q_{j + 1},
  y}$. Then by (i) and (viii) we have
  \[
    |\rho| = |{\lambda^{(j + 1)}}_{q_{j + 1}, x}| > |{\lambda^{(j + 1)}}_{q_{j
    + 1}, y}| \geq |{\lambda^{(j + 1)}}_{q_j, y}| + |\rho|.
  \]
  which is a contradiction. Hence it follows from (iii), (iv), and (vi) that
  $q_{j + 1}$ satisfies (II). It follows from (v) that we only need to verify
  (III) for $x$. From (i) we have that ${\lambda^{(j + 1)}}_{q_{j + 1}, x}
  \equiv \rho$, and so $(x) {\lambda^{(j + 1)}}_{q_{j + 1}, x} \notin \dom(q_{j
  + 1}) \cup \ran(q_{j + 1})$ by (vi) and the choice of $\rho$, and so (III)
  holds for $q_{j + 1}$. Finally, condition (IV) follows from (i), (iii),
  (viii) and the fact that $q_j$ satisfies (IV). Therefore, $q_{j + 1}$
  satisfies the inductive hypothesis.

  Consider the case where $|{\lambda^{(j + 1)}}_{q_j, x}| < |{\lambda^{(j + 1)}}_{q_j, y}|$.
  The above argument applies if we switch the roles of $x$ and $y$, i.e.\ let
  $\Gamma_j' = \Gamma_j \cup \{y\} \setminus \{x\}$, and $\lambda^{(j)}{}' \equiv
  \lambda^{(j + 1)}$. Then $q_j$, $\lambda^{(j)}{}'$ and $\Gamma_j'$ satisfy
  conditions (I) -- (IV) and we can proceed as before.

  Hence by induction there is $h = q_d$ satisfying (I) -- (IV).  Since $\langle
  \overline{f}, \overline{h} \rangle = S_n$ there is $w \in F_{\alpha, \beta}$
  which does not contain $\alpha^{-1}$, $\lambda^{(d)}$ is a prefix of $w$ and
  $\overline{w(h)} = \id$.  Then from (I) -- (IV) it follows that conditions
  (2), (3), (5) and (6) of $\S(\Gamma, \varnothing, \dom(q), \Delta, w)$ are
  satisfied by $h$.  Since $\Theta = \varnothing$, condition (4) of $\S(\Gamma,
  \varnothing, \dom(q), \Delta, w)$ follows trivially from (3) of $\S(\Gamma,
  \varnothing, \dom(q), \Delta, w)$. Hence $h$ satisfies $\S(\Gamma,
  \varnothing, \dom(q), \Delta, w)$.
\end{proof}

The next lemma is the second step in the proof of
Lemma~\ref{lem-4-main-restate}.

\begin{lem}\label{lem-4-fill-prod}
  Let $n \in \N$ be such that $n > 1$, let $f \in \aut(n K_\omega)$ be
  non-stabilising, let $q \in \A_{f, \Sigma}^{< \omega}$ be such that
  $\overline{q} \in S_n$, and let $w \in F_{\alpha, \beta}$ be a word which
  does not contain $\alpha^{-1}$ and which starts with $\alpha$.  Suppose
  $\Gamma, \Phi,  \subseteq \dom(q)$, $\Theta \subseteq \Gamma$, and $x \in
  \Gamma \setminus \Theta$. If $q$ satisfies $\S(\Gamma, \Theta, \Phi, \Delta,
  w)$, then there is an extension $h \in \A_{f, \Sigma}^{< \omega}$ of $q$ such
  that $h$ satisfies $\S(\Gamma, \Theta \cup \{ x\}, \Phi, \Delta, w)$.
\end{lem}

\begin{proof}
  For all $k \in \{0, \ldots, |w|\}$, define $\rho_k$ to be a prefix of $w$ of
  length $k$. Recall that for all $u \in \Gamma$ we identify the word $w_{q,
  u}$ with its realisation $w_{q, u}(u)$. In the same way, if $q_k$ is a partial
  isomorphism, then we identify the word $\rho_k$ with the partial isomorphism
  $\rho_k(q_k)$.

  It follows from condition (3) of $\S(\Gamma, \Theta, \Phi, \dom(q), w)$ and
  the fact that $x \in \Gamma \setminus \Theta$, that $x \notin \dom \left(
  w(q) \right)$, and so $w_{q, x}$ is a proper prefix of $w$. Let $M$ be such
  that $M - 1 = |w_{q, x}|$, or in other words $M$ is the smallest non-negative
  integer such that $x \notin \dom\left( \rho_M (q)\right)$. Then $M \leq |w|$.
  Since $x \in \Gamma \subseteq \dom(q)$ and $w$ starts with $\alpha$, it
  follows that $M > 1$, and so $M \in \{2, \ldots, |w|\}$. Since $w_{q, x}$ is
  a proper prefix of $w$, it follows that $w_{q, x} \alpha$ is a prefix of $w$
  and $(x)w_{q, x} \in nK_\omega \setminus \dom(q)$. Hence $\rho_M = \rho_{M -
  1} \alpha$ and the image of $x$ under $\rho_{M - 1} (q)$ is in $ n K_\omega
  \setminus \dom(q)$. 

  We will inductively construct a sequence $q_{M - 1} = q, q_M, \ldots, q_{|w|}
  \in \A_{f, \Sigma}^{< \omega}$ such that if $j \in \{M, \ldots, |w|\}$ then
  $q_j$ is an extension of $q_{j - 1}$ and the following conditions are
  satisfied
  \begin{enumerate}[(i)]
    \item $\ran(q_j) \cap \Delta = \varnothing$;
    \item $w_{q_j, u} \equiv w_{q, u}$ and $(u)w_{q_j, u} \in n K_\omega
      \setminus \dom(q_j)$ for all $u \in \Gamma \setminus \{x\}$;
    \item $(x)\rho_j f^i \in n K_\omega \setminus \dom(q_j)$ for all $i \in \{
      - \mathbf{b}(w) + \mathbf{b}(\rho_j), \ldots, \mathbf{b}(w) -
      \mathbf{b}(\rho_j)\}$;
    \item $(x)w_{q_j, x} \neq (u)w_{q_j, u}$ for all $u \in \Gamma \setminus
      \{x\}$;
    \item $(u) w_{q_j, u} q_j^m \in n K_\omega \setminus \Phi$ for all $u \in
      \Gamma$ and for all $m \in \Z$ such that $u \in \dom(w_{q_j, u} q_j^m)$ .
  \end{enumerate}
  Then $h = q_{|w|}$ will be the required extension of $q$.
  
  Let $y$ be the image of $x$ under $\rho_{M - 1} = w_{q, x}$ and
  suppose $y \in L_a$ for some $a \in \{1, \ldots, n\}$. Recall that
  $\mathbf{b}(w)$ is the number of occurrences of letters $\beta$ and
  $\beta^{-1}$ in the word $w$. We may choose:
  \[
    z \in L_{(a)\overline{q}} \setminus \bigcup_{i = -
    \mathbf{b}(w)}^{\mathbf{b}(w)} \left(\dom(q) \cup \ran(q) \cup \{y\} \cup
    \Delta \cup \left\{ (u)w_{q, u} : u \in \Gamma \right\} \right)f^{-i}.
  \]
  and define $q_M = q \cup \{(y, z)\}$. Then 
 $q_M \in \A_{f, \Sigma}^{< \omega}$ by
  Lemmas~\ref{lem-4-easy-one-point} and~\ref{lem-4-one-point-ext}, since
  $y \notin \dom(q)$ and $z \notin \dom(q) \cup \ran(q)$.

  First, we will show that $q_M$ satisfies conditions (i) to (v). Since
  $\ran(q_M) = \ran(q) \cup \{z\}$ and $z$ was chosen outside $\Delta$, it
  follows that $q_M$ satisfies (i). Let $u \in \Gamma \setminus \{x\}$. If $u
  \notin \Theta$, then, from (3) of $\S(\Gamma, \Theta, \Phi, \Delta,
  w)$, $u\not\in \dom(w(q))$ and so $w_{q, u}$ is a proper prefix of $w$. It
  follows that $(u)w_{q, u} \in n K_\omega \setminus \dom(q)$. On the other
  hand, if $u \in \Theta$, then 
  $w_{q, u} = w$ and $(u)w_{q,u} \in n K_\omega \setminus \dom(q)$ by (3) and
  (4) of $\S(\Gamma, \Theta, \Phi, \Delta, w)$. Hence in both cases $(u)w_{q,
  u} \in n K_\omega \setminus \dom(q)$. Since $\dom(q_M) \setminus \dom(q) =
  \{y\}$ and $(u) w_{q, u} \neq (x) w_{q, x} = y$ by part (5) of $\S(\Gamma,
  \Theta, \Phi, \Delta, w)$, it follows that $(u) w_{q, u} \in n K_\omega
  \setminus \dom(q_M)$, and so $w_{q_M, u} \equiv w_{q, u}$, proving (ii). Let
  $i \in \{-\mathbf{b}(w) + \mathbf{b}(\rho_M), \ldots \mathbf{b}(w) -
  \mathbf{b}(\rho_M)\}$. Since $\dom(q_M) = \dom(q) \cup \{y\}$, it follows
  from the choice of $z$ that 
  \[
    (x)\rho_M f^i = (y)q_M f^i = (z)f^i \in n K_\omega \setminus \dom(q_M).
  \]
  Hence $q_M$ satisfies condition (iii). Let $u \in \Gamma \setminus \{x\}$.
  Note that since $q_M$ satisfies (iii) there is $k \in \{- \mathbf{b}(w),
  \ldots, \mathbf{b}(w)\}$ such that $w_{q_M, x} = w_{q, x} \alpha \beta^k$,
  and so $(x)w_{q_M, x} = (z)f^k$. It follows from the choice of $z$, and the
  fact that $q_M$ satisfies (ii) that 
  \[
    (x)w_{q_M, x} = (z)f^k \neq (u) w_{q, u} = (u) w_{q_M, u}.
  \]
  Hence $q_M$ satisfies (iv).

  Finally, to show that $q_M$ satisfies (v) consider two cases --- $u = x$ and
  $u \in \Gamma \setminus \{x\}$. Suppose that $u = x$ and $m \in \Z$ is such
  that $x \in \dom(w_{q_M, x} q_M^m)$. As shown before $(x)w_{q_M, x} = (z)f^k$
  for some $k \in \{- \mathbf{b}(w), \ldots, \mathbf{b}(w)\}$. From the choice
  of $z$ it follows that $(x)w_{q_M, x} = (z)f^k \notin \dom(q)\cup \ran(q)
  \cup \{y\}$. Suppose $(z)f^k \neq z$. Then $(x)w_{q_M, x} = (z)f^k \notin
  \dom(q_M) \cup \ran(q_M)$, and so $m = 0$. Since $\Phi \subseteq
  \dom(q) \subseteq \dom(q_M)$, this implies that 
  \[
    (x)w_{q_M, x}q_M^m = (z) f^k \in n K_\omega \setminus \Phi.
  \]
  Suppose that $(z)f^k = z$, in other words $(x)w_{q_M, x} = z$. Since $z
  \notin \dom(q_M)$, it follows that $x \notin \dom(w_{q_M, x} q_M^m)$ for all
  $m > 0$. If $m = 0$ then $(x)w_{q_M, x} q_M^m = (z)f^k \in n K_\omega
  \setminus \Phi$ by the choice $z$ and since $\Phi \subseteq \dom(q)$. Suppose
  that $m < 0$. Then $m + 1 \leq 0$ and it follows from the definition of $q_M$
  that $\dom(q_M^{m + 1})$ is either $\dom(q^{m + 1})$ or $\dom(q^{m + 1})
  \cup \{(z)q_M^{m + 1}\}$. Note that $y \in \dom(q_M^{m + 1})$ implies $y \in
  \dom(q^{m + 1})$. It follows that from (6) of $\S(\Gamma, \Theta, \Phi,
  \Delta, w)$ that 
  \[
    (x)w_{q_M, x} q_M^m = (z)q_M^m = (y) q_M^{m + 1} =
    (y)q^{m + 1} = (x) w_{q, x} q^{m + 1} \in n K_\omega \setminus  \Phi.
  \]
  Hence $q_M$ satisfies (v) for $u = x$.

  Suppose that $u \in \Gamma \setminus \{x\}$ and $m \in \Z$ is such that $u \in
  \dom(w_{q_M, u} q_M^m)$. Since $q_M$ satisfies (ii), it follows that
  $(u)w_{q_M, u} = (u)w_{q, u}$. If $m \leq 0$, or $m > 0$ and there is no $m'
  \in \{0, \ldots, m - 1\}$ with $(u)w_{q, u}q^{m'} = y$, then $(u)w_{q_M, u}
  q_M^m = (u)w_{q, u} q^m \in nK_\omega \setminus \Phi$ by (6) of $\S(\Gamma,
  \Theta, \Phi, \Delta, w)$. Otherwise, $m > 0$ and there is $m' \in \{0,
  \ldots, m - 1\}$ such that $(u)w_{q, u}q^{m'} = y$, in which case $(u)w_{q_M,
  u}q_M^{m' + 1} = z \notin \dom(q_M)$. Hence $m = m' + 1$, and since 
  $\Phi \subseteq \dom(q) \subseteq \dom(q_M)$, it follows that $(u)w_{q_M,
  u}q_M^m \in n K_\omega \setminus \Phi$. Therefore, $q_M$ satisfies (v) and
  thus the inductive hypothesis holds.

  In the case where $M = |w|$, $q_{|w|}$ already satisfies conditions (i) to
  (v). Hence suppose that $M < |w|$ and suppose that for some $j \in \{M,
  \ldots, |w| - 1 \}$ there is an extension $q_j \in \A_{f, \Sigma}^{< \omega}$ 
  of $q_{j - 1}$ satisfying conditions (i) to (v). We have two cases to
  consider --- either $\rho_{j + 1} = \rho_j \beta^\varepsilon$ or $\rho_{j +
  1} = \rho_j \alpha^\varepsilon$ for some $\varepsilon \in \{-1, 1\}$.

  First consider the case $\rho_{j + 1} = \rho_j \beta^\varepsilon$, where
  $\varepsilon \in \{-1, 1\}$. Let $q_{j + 1} = q_j$. Then conditions (i),
  (ii), (iv), and (v) are trivially satisfied by $q_{j + 1}$. In order to show
  that $q_{j + 1}$ satisfies (iii), let $i \in \Z$ be such that $i \in \{ -
  \mathbf{b}(w) + \mathbf{b}(\rho_{j + 1}), \ldots, \mathbf{b}(w) -
  \mathbf{b}(\rho_{j + 1}) \}$. Then $|i + \varepsilon| \leq \mathbf{b}(w) -
  \mathbf{b}(\rho_{j + 1}) + 1  = \mathbf{b}(w) - \mathbf{b}(\rho_j)$, and so
  \[
    (x)\rho_{j + 1}f^i = (x)\rho_j f^{i + \varepsilon} \in n K_\omega \setminus
    \dom(q_j) = n K_\omega \setminus \dom(q_{j + 1}).
  \]
  Hence $q_{j + 1}$ satisfies condition (iii), and so the induction hypothesis.

  Otherwise $\rho_{j + 1} = \rho_j \alpha^\varepsilon$ for some $\varepsilon
  \in \{-1, 1\}$, and so $\rho_{j + 1} = \rho_j \alpha$ since $w$ does not
  contain $\alpha^{-1}$. Let $y = (x)\rho_j$, and let $a \in \{1, \ldots, n\}$ be
  such that $y \in L_a$. Choose
  \[
    z \in L_{(a)\overline{q_j}} \setminus \bigcup_{i = -
    \mathbf{b}(w)}^{\mathbf{b}(w)} \left( \dom(q_j)\cup\ran(q_j) \cup \{y\}
    \cup \Delta \cup \left\{ (u)w_{q_j, u} : u \in \Gamma \right\} \right)
    f^{-i}.
  \]
  Since $y \notin \dom(q_j)$ by (iii) and $z \notin \dom(q_j) \cup \ran(q_j)$,
  it follows from Lemmas~\ref{lem-4-easy-one-point}
  and~\ref{lem-4-one-point-ext} that $q_{j + 1} = q_j \cup \{(y, z)\} \in
  \A_{f, \Sigma}^{< \omega}$. Observe that
  \begin{equation}\label{equation-new-dom-ran-2}
    \dom(q_{j + 1}) = \dom(q_j) \cup \{y\} \quad \text{and} \quad
    \ran(q_{j + 1}) = \ran(q_j) \cup \{z\}.
  \end{equation}
  The vertex $z$ was chosen so that $z \notin \Delta$, and so $q_{j + 1}$
  satisfies (i).

  It follows from (iii) that $x \in \dom \left(\rho_j \right)$
  and $x \notin \dom \left( \rho_{j + 1}(q_j) \right)$, thus 
   $w_{q_j, x} \equiv \rho_j$. Let $u \in \Gamma \setminus \{x\}$. Since $q_j$
   satisfies (iv)
   \[
    (u) w_{q_j, u} \neq (x) w_{q_j, x} = (x) \rho_j = y.
  \]
  It then follows from $(u) w_{q_j, u} \in n K_\omega \setminus \dom(q_j)$
  and~\eqref{equation-new-dom-ran-2} that $(u) w_{q_j, u} \in nK_\omega
  \setminus \dom(q_{j + 1})$, and so $w_{q_{j + 1}, u} \equiv w_{q_j, u}$.
  Then $(u) w_{q_{j + 1}, u} \in nK_\omega \setminus \dom(q_{j + 1})$, and
  since $q_j$ satisfies (ii) it follows that $q_{j + 1}$ also satisfies (ii).

  Let $i \in \{ - \mathbf{b}(w), \ldots, \mathbf{b}(w)\}$. Then
  by~\eqref{equation-new-dom-ran-2} and the fact that $z$ was chosen so that
  $(z)f^i \notin \dom(q_j) \cup \{y\}$
  \[
    (x)\rho_{j + 1} f^i = (z)f^i \in n K_\omega \setminus \dom(q_{j + 1}).
  \]
  Hence $q_{j + 1}$ satisfies (iii).

  It follows from the fact that $q_{j + 1}$ satisfies (iii), that  $w_{q_{j +
  1}, x} \equiv w_{q_j, x} \alpha \beta^k$ for some $k \in \{ -\mathbf{b}(w),
  \ldots, \mathbf{b}(w)\}$, and so
  \begin{equation}\label{equation-x-under-new-w}
    (x)w_{q_{j + 1}, x} = (z)f^k.
  \end{equation}
  By the choice of $z$ and the fact that $q_{j + 1}$ satisfies (ii)
  \[
    (x)w_{q_{j + 1}, x} = (z)f^k \neq (u) w_{q_j, u} = (u) w_{q_{j + 1}, u}
  \]
  for every $u \in \Gamma\setminus\{x\}$. Hence $q_{j + 1}$ satisfies (iv).

  Finally, to show that $q_{j + 1}$ satisfies (v) consider two cases --- $u =
  x$ and $u \in \Gamma \setminus \{x\}$. Suppose that $u = x$ and $m \in \Z$ is
  such that $x \in \dom(w_{q_{j + 1}, x} q_{j + 1}^m)$. From the choice of $z$
  and~\eqref{equation-x-under-new-w} it follows that $(x)w_{q_{j + 1}, x} =
  (z)f^k \notin \dom(q)\cup \ran(q) \cup \{y\}$. Suppose $(z)f^k \neq z$. Then
  $(x)w_{q_{j + 1}, x} = (z)f^k \notin \dom(q_{j + 1}) \cup \ran(q_{j + 1})$,
  and so $m = 0$, in which case $\Phi \subseteq \dom(q) \subseteq \dom(q_{j +
  1})$ implies that 
  \[
    (x)w_{q_{j + 1}, x}q_{j + 1}^m = (z) f^k \in n K_\omega \setminus \Phi.
  \]
  Suppose that $(z)f^k = z$, in other words $(x)w_{q_{j + 1}, x} = z$. Since $z
  \notin \dom(q_{j + 1})$, it follows that $x \notin \dom(w_{q_{j + 1}, x} q_{j
  + 1}^m)$ for all $m > 0$. If $m = 0$, then $(x)w_{q_{j + 1}, x} q_{j + 1}^m =
  (z)f^k \in n K_\omega \setminus \Phi$ by the choice $z$ and since $\Phi
  \subseteq \dom(q)$. Suppose that $m < 0$. Then $m + 1 \leq 0$ and it follows
  from the definition of $q_{j + 1}$ that $\dom(q_{j + 1}^{m + 1})$ is either
  $\dom(q_j^{m + 1})$ or $\dom(q_j^{m + 1}) \cup \{(z)q_{j + 1}^{m + 1}\}$.
  Note that $y \in \dom(q_{j + 1}^{m + 1})$ implies $y \in \dom(q_j^{m + 1})$.
  It follows that from (6) of $\S(\Gamma, \Theta, \Phi, \Delta, w)$ that 
  \[
    (x)w_{q_{j + 1}, x} q_{j + 1}^m = (z)q_{j + 1}^m = (y) q_{j + 1}^{m + 1} =
    (y)q_j^{m + 1} = (x) w_{q_j, x} q^{m + 1} \in n K_\omega \setminus  \Phi.
  \]
  Hence $q_{j + 1}$ satisfies (v) for $u = x$.

  Suppose that $u \in \Gamma \setminus \{x\}$ and $m \in \Z$ such that $u \in
  \dom(w_{q_{j + 1}, u} q_{j + 1}^m)$. Since $q_j$ and $q_{j + 1}$ satisfies
  (ii), it follows that $(u)w_{q_{j + 1}, u} = (u)w_{q, u} = (u)w_{q_j, u}$. If
  $m \leq 0$, or $m > 0$ and there is no $m' \in \{0, \ldots, m - 1\}$ with
  $(u)w_{q_j, u}q^{m'} = y$, then $(u)w_{q_{j + 1}, u} q_{j + 1}^m = (u)w_{q_j,
  u} q_j^m \in nK_\omega \setminus \Phi$ since $q_j$ satisfies (v). Otherwise,
  $m > 0$ and there is $m' \in \{0, \ldots, m - 1\}$ such that $(u)w_{q_j,
  u}q_j^{m'} = y$, in which case $(u)w_{q_{j + 1}, u}q_{j + 1}^{m' + 1} = z
  \notin \dom(q_{j + 1})$. Hence $m = m' + 1$, and since $\Phi \subseteq
  \dom(q) \subseteq \dom(q_{j + 1})$, it follows that $(u)w_{q_{j + 1}, u}q_{j
  + 1}^m \in n K_\omega \setminus \Phi$. Therefore, $q_{j + 1}$ satisfies (v)
  and thus the inductive hypothesis.

  By induction there is $h = q_{|w|} \in \A_{f, \Sigma}^{< \omega}$
  satisfying (i) -- (v). We will show that $h$ satisfies $\S(\Gamma, \Theta\cup
  \{x\}, \Phi, \Delta, w)$ and will refer to parts (1) to (6) of this condition
  by writing (1) to (6), where appropriate, without reference to $\S(\Gamma,
  \Theta\cup \{x\}, \Phi, \Delta, w)$ in the rest of the proof.
  
  Since $h$ is an extension of $q$ and $\overline{q} \in
  S_n$, it follows that $\overline{h} = \overline{q}$. Hence 
  \[
    \overline{w(h)} = \id,
  \]
  and so $h$ satisfies (1).
  Since $h$ satisfies (i) and (v), it also satisfies (2) and (6).
  Since $w = \rho_{|w|}$ condition (iii) implies that $x \in \dom(w(h))$, and
  so $x \in \dom(w(h)) \cap \Gamma$. If $u \in \Gamma \setminus \{x\}$, then
  $w_{h, u}
  \equiv w_{q, u}$ by (ii), and so $u \in \dom(w(h)) \cap \Gamma$ if and only
  if $u \in \dom(w(q)) \cap \Gamma$. Therefore, $\dom(w(h)) \cap \Gamma =
  \Theta \cup \{x\}$ as $q$ satisfies $\S(\Gamma, \Theta, \Phi,
  \Delta, w)$, in other words $h$ satisfies (3). By (iii) the image of $x$
  under $w(h)$ is in $n K_\omega \setminus \dom(h)$, and by (ii) the image of
  $u \in \Theta$ under $w(h) = w_{q, u}$ is also in $n K_\omega \setminus
  \dom(h)$. Hence $h$ satisfies condition (4). It then follows from (ii), (iv)
  and the fact that $q$ satisfies (5) of $\S(\Gamma, \Theta, \Phi,
  \Delta, w)$ that $(u)w_{h, u} = (v)w_{h, v}$ only if
  $u = v$ for all $u, v \in \Gamma$, and thus $h$ satisfies (5). Hence $h$
  satisifes $\S(\Gamma, \Theta \cup \{x\}, \Phi, \Delta, w)$, as required. 
\end{proof}

\begin{proof}[Proof of Lemma~\ref{lem-4-main-restate}]
  If necessary by extending $q$ using Lemmas~\ref{lem-4-easy-one-point}
  and~\ref{lem-4-one-point-ext}, we can assume that $\Gamma \subseteq \dom(q)$.

  Let $d = |\Gamma|$. By Lemma~\ref{lem-4-base-case}, there is a freely reduced
  word $w \in F_{\alpha, \beta}$ not containing $\alpha^{-1}$ and starting with
  $\alpha$, and an extension $q_0 \in \A_{f, \Sigma}^{< \omega}$ of $q$
  satisfying $\S(\Gamma, \varnothing, \dom(q), \Delta, w)$.  Suppose that for
  some $j \in \{0, 1, \ldots,  d - 1\}$  we have already extended $q = q_0$ to
  $q_j \in \A_{f, \Sigma}^{<\omega}$ such that there is $\Gamma_j \subseteq
  \Gamma$, with $|\Gamma_j| = j$, and $q_j$ satisfies $\S(\Gamma, \Gamma_j,
  \dom(q), \Delta, w)$. Let $x \in \Gamma \setminus \Gamma_j$ and let
  $\Gamma_{j + 1} = \Gamma_j \cup \{x\}$. Then, by condition (3) of $\S(\Gamma,
  \Gamma_j, \dom(q), \Delta, w)$, $x \notin \dom \left( w(q_j)
  \right)$. Hence if we let $\Theta = \Gamma_j$ and $\Phi = \dom(q)$ then by
  Lemma~\ref{lem-4-fill-prod} there is $q_{j + 1} \in \A_{f, \Sigma}^{<
  \omega}$ and extension of $q_j$ satisfying $\S(\Gamma, \Gamma_{j + 1},
  \dom(q), \Delta, w)$.

  Therefore, by induction on $j$ we obtain $h = q_d \in \A_{f, \Sigma}^{<
  \omega}$ which satisfies $\S(\Gamma, \Gamma, \dom(q), \Delta, w)$, as
  required.
\end{proof}

\bibliography{undirected}{}

\begin{thebibliography}{1}

\bibitem{darji2008highly}
U.~B. Darji and J.~D. Mitchell.
\newblock Highly transitive subgroups of the symmetric group on the natural
  numbers.
\newblock {\em Colloquium Mathematicum}, 112(1):163--173, 2008.

\bibitem{darji2011aa}
U.~B. Darji and J.~D. Mitchell.
\newblock Approximation of automorphisms of the rationals and the random graph.
\newblock {\em Journal of Group Theory}, 14(3):361--388, 2011.

\bibitem{Erdos:1963qf}
P.~Erd{\H{o}}s and A.~R{\'e}nyi.
\newblock Asymmetric graphs.
\newblock {\em Acta Math. Acad. Sci. Hungar}, 14:295--315, 1963.

\bibitem{Hodges1997aa}
Wilfrid Hodges.
\newblock {\em A shorter model theory}.
\newblock Cambridge University Press, Cambridge, 1997.

\bibitem{kechris1995classical}
A.~S. Kechris.
\newblock {\em Classical Descriptive Set Theory}.
\newblock Graduate Texts in Mathematics. Springer-Verlag, 1995.

\bibitem{lachlan1980aa}
A.~H. Lachlan and Robert~E. Woodrow.
\newblock Countable ultrahomogeneous undirected graphs.
\newblock {\em Transactions of the American Mathematical Society},
  262(1):51--94, 1980.

\bibitem{piccard1939aa}
Sophie Piccard.
\newblock Sur les bases du groupe symetrique et du groupe alternant.
\newblock {\em Mathematische Annalen}, 116(1):752--767, 1939.

\end{thebibliography}
\bibliographystyle{plain}
\end{document}